\documentclass[a4paper,reqno,11pt]{amsart}

\usepackage[margin=15mm]{geometry}
\usepackage{amssymb}
\usepackage{enumerate}
\usepackage{mathrsfs}

\numberwithin{equation}{section}

\usepackage{tikz}
\usetikzlibrary{arrows,decorations.pathmorphing,decorations.pathreplacing}
\tikzset{black/.style={circle,fill=black,inner sep=3pt,outer sep=3pt},white/.style={circle,fill=white,draw=black,inner sep=3pt,outer sep=3pt}}
\usetikzlibrary{cd}

\input xy
\xyoption{all}

\usepackage{pstricks}

\usepackage{comment}
\definecolor{darkblue}{rgb}{0,0,0.7}
\newcommand{\defn}[1]{\textsl{\color{darkblue} #1}}

\RequirePackage[
   pdftex,
   colorlinks = true,
   citecolor = red, % \cite{...}
   pagebackref,
   bookmarksopen=true
]{hyperref}
\usepackage[hyperpageref]{backref}
\renewcommand*{\backref}[1]{}  
   \renewcommand*{\backrefalt}[4]{
      \ifcase #1 
         Not cited.
      \or
         Cited on page #2.
      \else
         Cited on pages #2.
      \fi}
\newtheorem{theorem}{Theorem}[section]
\newtheorem{corollary}[theorem]{Corollary}
\newtheorem{lemma}[theorem]{Lemma}
\newtheorem{proposition}[theorem]{Proposition}
\theoremstyle{definition}
\newtheorem{definition}[theorem]{Definition}
\newtheorem{definition-proposition}[theorem]{Definition-Proposition}
\newtheorem{remark}[theorem]{Remark}
\newtheorem{example}[theorem]{Example}

\newtheorem{question}[theorem]{Question}

\newtheorem{algorithm}[theorem]{Algorithm}
\renewcommand{\mod}{\operatorname{mod}}
\newcommand{\proj}{\operatorname{proj}}
\newcommand{\inj}{\operatorname{inj}}
\newcommand{\add}{\operatorname{\mathrm{add}}}

\newcommand{\idim}{\operatorname{idim}}
\newcommand{\pdim}{\operatorname{pdim}}
\newcommand{\gldim}{\operatorname{gldim}}

\newcommand{\Ext}{\operatorname{Ext}}
\newcommand{\Tor}{\operatorname{Tor}}
\newcommand{\End}{\operatorname{End}}
\newcommand{\Hom}{\operatorname{Hom}}

\renewcommand{\top}{\operatorname{\mathrm{top}}}
\newcommand{\soc}{\operatorname{\mathrm{soc}}}
\newcommand{\rad}{\operatorname{\mathrm{rad}}}

\newcommand{\Ker}{\operatorname{Ker}\nolimits}
\newcommand{\Cok}{\operatorname{Cok}\nolimits}
\newcommand{\Img}{\operatorname{Im}\nolimits}

\newcommand{\id}{\operatorname{id}}
\newcommand{\op}{\mathrm{op}}
\newcommand{\kD}{\mathbb{D}}
\newcommand{\inc}{\mathsf{inc}}
\newcommand{\dec}{\mathsf{dec}}
\newcommand{\relvx}{\mathrm{relv}}
\newcommand{\e}{\varepsilon}
\newcommand{\tr}{\operatorname{\mathrm{tr}}}
\newcommand{\rej}{\operatorname{\mathrm{rej}}}
\newcommand{\lex}{\mathsf{l}\text{-}\mathrm{ex}}
\newcommand{\rex}{\mathsf{r}\text{-}\mathrm{ex}}

\newcommand{\tilt}{\operatorname{\mathsf{tilt}}}
\newcommand{\ptilt}{\operatorname{\mathsf{pretilt}}}
\newcommand{\cotilt}{\operatorname{\mathsf{cotilt}}}
\newcommand{\pcotilt}{\operatorname{\mathsf{precotilt}}}
\newcommand{\chtilt}{\operatorname{\mathsf{chtilt}}}
\newcommand{\IStilt}{\operatorname{\mathsf{IStilt}}}
\newcommand{\qhs}{\operatorname{\mathsf{qhs}}}

\newcommand{\filt}{\mathcal{F}}
\newcommand{\std}{\Delta}
\newcommand{\costd}{\nabla}

\newcommand{\bbA}{\mathbb{A}}
\newcommand{\Z}{\mathbb{Z}}

\newcommand{\calI}{\mathcal{I}}
\newcommand{\calN}{\mathcal{N}}
\newcommand{\calO}{\mathcal{O}}
\newcommand{\calP}{\mathcal{P}}
\newcommand{\calS}{\mathcal{S}}
\newcommand{\calT}{\mathcal{T}}
\newcommand{\calU}{\mathcal{U}}
\newcommand{\calX}{\mathcal{X}}

\newcommand{\sfe}{\mathsf{e}}
\newcommand{\sfl}{\mathsf{l}}
\newcommand{\sfr}{\mathsf{r}}

\newcommand{\sm}[1]{\begin{smallmatrix}#1\end{smallmatrix}}
\begin{document}
\title[Tilting theoretic approach to quasi-hereditary structures]{Tilting theoretic approach to quasi-hereditary structures}

\author{Takahide Adachi}
\address{T.~Adachi: Faculty of Global and Science Studies, Yamaguchi University, 1677-1 Yoshida, Yamaguchi 753-8541, Japan}
\email{tadachi@yamaguchi-u.ac.jp}

\author{Aaron Chan}
\address{A.~Chan: Graduate School of Mathematics, Nagoya University, Furo-cho, Chikusa-ku, Nagoya, Aichi 464-8601, Japan}
\email{aaron.kychan@gmail.com}

\author{Yuta Kimura}
\address{Y.~Kimura : Department of Mechanical Engineering and Informatics, Faculty of Engineering, Hiroshima Institute of Technology, 2-1-1 Miyake, Saeki-ku Hiroshima 731-5143, Japan}
\email{y.kimura.4r@cc.it-hiroshima.ac.jp}

\author{Mayu Tsukamoto}
\address{M.~Tsukamoto: Graduate school of Sciences and Technology for Innovation, Yamaguchi University, 1677-1 Yoshida, Yamaguchi 753-8512, Japan}
\email{tsukamot@yamaguchi-u.ac.jp}

\date{\today}

\subjclass[2020]{}

\dedicatory{Dedicated to Claus Michael Ringel on the occasion of his 80th birthday}

\begin{abstract}
A quasi-hereditary algebra is an algebra equipped with a certain partial order $\unlhd$ on its simple modules. Such a partial order -- called a quasi-hereditary structure -- gives rise to a characteristic tilting module $T_{\unlhd}$ by a classical result due to Ringel. 
A fundamental question is to determine which tilting modules can be realised as characteristic tilting modules.
We answer this question by using the notion of IS-tilting module, which is a pair $(T,\unlhd)$ of a tilting module $T$ and a partial order $\unlhd$ on its direct summands such that iterative idempotent truncation along $\unlhd$ always reveals a simple direct summand.
Specifically, we show that a tilting module $T$ is characteristic if, and only if, there is some $\unlhd$ so that $(T,\unlhd)$ is IS-tilting; in which case, we have $T=T_{\unlhd}$.
This result enables us to study quasi-hereditary structures using tilting theory. 

As an application of the above result, we show that, for an algebra $A$, all tilting modules are characteristic if, and only if, $A$ is a quadratic linear Nakayama algebra.
Furthermore, for such an $A$, we provide a decomposition of the set of its tilting modules that can be used to derive a recursive formula for enumerating its quasi-hereditary structures. 
Finally, we describe the quasi-hereditary structures of $A$ via `nodal gluing' and binary tree sequences.
\end{abstract}

\maketitle
\tableofcontents

\section{Introduction}

The notion of quasi-hereditary algebras was introduced in the seminal work of Cline, Parshall, and Scott \cite{CPS88}, motivated by the representation theory of complex Lie algebras and algebraic groups. 
Since its introduction, quasi-hereditary algebras frequently appear in the representation theory of finite-dimensional algebras. 
Notable examples include hereditary algebras, Auslander algebras \cite{DR89a}, and more generally, algebras of global dimension at most two \cite{DR89}, ($q$-)Schur algebras \cite{CPS88, D98}, algebras arising from rational surfaces \cite{HP14} and exceptional sequences \cite{Kr17}, etc. 
The study of quasi-hereditary algebras has since been developed into a myriad of areas: discovery of new examples and constructions \cite{DR89b, BF89, UY90, X94, Put98, LS13, IR11, CE18, BK18, GS19, FKR22, Goto}, in-depth study of their homological algebra \cite{ADL03, Ma10, MO05, Kr17, Con16, T19, T20} and its applications \cite{DR89b, I03, R10}, appearance and applications in algebraic geometry \cite{BDG17, BB24, HP14, KK17, O18}, higher homological $A_\infty$-structures \cite{KKO14, BK18, CK25}, interaction with algebraic Lie theory \cite{CPS93, D98, MO05, Kl15, Cou20, BS24}, etc.

Being quasi-hereditary depends on a choice of partial orders on the set of isomorphism classes of simple modules.
This partial order needs to satisfy certain condition that is determined by a special class of modules called the \emph{standard modules}; see subsection \ref{subsec:qhs char tilt} for details.  These modules generalise the Verma modules of a complex semisimple Lie algebra, Weyl modules in the representation theory of algebraic groups, among others.
They stratify the derived category into those of the vector spaces, which allow more systematic study of the homological algebra.
By a \emph{quasi-hereditary structure} we mean an equivalence class  of partial orders, where two partial orders are equivalent if their associated standard modules are the same.
Taking universal extensions of the standard modules one obtain a  tilting module which is often called \emph{characteristic tilting} \cite{R91} in this context.

In tilting theory, tilting modules serve a role analogous to that of progenerators in Morita theory, enabling the transfer of homological structures between categories. Tilting theory is now central to modern developments in areas such as quiver representations, modular representations of finite groups, applications in algebraic geometry via derived categories, abstract structures in homological algebra -- for instance, triangulated and DG categories, etc.
Riedtmann and Schofield \cite{RieS91} showed that one can construct new tilting modules from old by a process which is nowadays known as \emph{mutation}.  Moreover, Happel and Unger \cite{HU05} showed that the set of tilting modules over an algebra can be partially ordered, and the covering relation can be described by mutation.  This allows very effective computation of tilting modules.

While each quasi-hereditary structure can be used to construct a tilting module, we are also interested in the reverse direction -- starting with the set of all tilting modules and determine (or even, classify) the quasi-hereditary structures of an algebra.  At the heart of this problem lies the following question.

\begin{question}\label{quest:main}
How to determine if a tilting module can appear as the characteristic tilting module associated to a quasi-hereditary structure?
\end{question}

The first main goal of this paper is to answer this question by introducing the following condition.

\begin{definition}[Definition \ref{def:IS-tilt}]
Let $A$ be a finite-dimensional algebra and let $\{e_x\mid x\in\Lambda\}$ be a complete set of primitive orthogonal idempotents.
Let $\{S(x) \mid x\in \Lambda\}$ be a set of simple $A$-modules.
A pair $(T, \unlhd)$ of an $A$-module $T$ and a partial order $\unlhd$ on $\Lambda$ is called an \emph{IS-tilting module} if $T$ is a tilting $A$-module and there exists an indecomposable decomposition $T=\bigoplus_{x\in\Lambda}T(x)$ satisfying
\begin{align}
T(x)\varepsilon_{x}\cong S(x)\varepsilon_{x} \; \textnormal{in}\; \mod (\varepsilon_{x} A\varepsilon_{x}),    \notag
\end{align}
where $\displaystyle\varepsilon_{x}:=\sum_{y\in\Lambda\textnormal{\;s.t.\;}y\not\lhd x}e_{y}$. 
\end{definition}

The key of this condition is that after each stage of idempotent truncation (applying the Schur functor), there is always a simple \emph{direct summand} in the resultant which provides a natural choice for the next composition factor to eliminate.
This is the reason for the use of the term `IS', which abbreviates `\underline{i}terative elimination of \underline{s}imple direct summand'. 
Note that characteristic tilting modules are IS-tilting and this order of eliminating simple modules yields precisely the associated quasi-hereditary structure.
Now we can state the answer to Question \ref{quest:main} as follows. 

\begin{theorem}[Theorem \ref{thm:qh-tilting}, Proposition \ref{prop:isttilt-chtilt} and Theorem \ref{thm:bij-qh-IStilting}]\label{mainthm1}
Let $A$ be a basic finite-dimensional algebra. 
Then the following statements hold. 
\begin{itemize}
\item[(1)] If $(T, \unlhd)$ is an IS-tilting $A$-module, then $(A,\unlhd)$ is a quasi-hereditary algebra. 
\item[(2)] Characteristic tilting modules coincide with IS-tilting modules. More precisely, there exists a bijection between the set of IS-tilting $A$-modules and the set of quasi-hereditary structures of $A$.
\end{itemize}
\end{theorem}

Part (2) of the above theorem is a kind of dual to a result of Ringel \cite[Appendix]{R91}, which determines whether a tilting module is characteristic tilting by considering iterative idempotent quotients $T/T(1-\e_{x}) A$ instead of iterative idempotent truncation $T\e_{x}$.  The difficulty in utilising Ringel's result is that it requires checking the quotient is once again tilting over the corresponding quotient algebra $A/A(1-\e_{x})A$.  In contrast, being IS-tilting can be checked by only examining composition factors of each indecomposable direct summand, without the need to verify that $T\e_x$ is tilting over $\e_{x} A\e_{x}$.
We also remark that the formulation of IS-tilting together with this result of Ringel acts as a \emph{tilting analogue} of the inductive mechanism of quasi-hereditary algebras:  For a quasi-hereditary algebra $(A,\unlhd)$, both $(\e_{x} A\e_{x}, \unlhd)$ and $(A/A(1-\e_{x})A, \unlhd)$ are also quasi-hereditary.

In the remaining of the paper, we demonstrate some uses of IS-tilting modules in understanding quasi-hereditary structures.
The first use is to determine which class of algebra have \emph{all} of its tilting modules appear as characteristic tilting.
It turns out that by replacing `characteristic tilting' with `IS-tilting', we can answer this question completely.

\begin{theorem}[Theorem \ref{thm:tilt=IStit=>gentle-LNaka}]\label{mainthm2}
Let $A$ be a basic finite-dimensional algebra. 
Then all tilting $A$-modules are IS-tilting if, and only if, $A$ is a quadratic linear Nakayama algebra, that is, it is a bound quiver algebra whose quiver is a linearly oriented quiver of type $\bbA$ and whose relations are generated by quadratic monomials.
\end{theorem}

We remark that the proof of the only-if direction makes great use of tilting mutation -- this demonstrates how powerful tools in tilting theory allow us to solve problems related to quasi-hereditary structures.
Also, Theorem \ref{mainthm2} is a generalisation of \cite[Theorem 4.7]{FKR22}, which shows that all tilting modules are characteristic tilting for the path algebra of a linearly oriented quiver $\Vec{\bbA}_{n}$.

Now, by combining Theorems \ref{mainthm1}(2) and \ref{mainthm2}, the classification of quasi-hereditary structures for a quadratic linear Nakayama algebra can then be reduced to that of tilting modules. 
One of Gabriel's famous results states that the number of tilting modules over the path algebra $A$ of a linearly oriented quiver $\Vec{\bbA}_{n}: 1\rightarrow 2 \rightarrow \cdots \rightarrow n$ is given by $n$-th Catalan number. As a consequence of this result, the set $\tilt A$ of tilting $A$-modules can be decomposed as $\tilt A =\bigsqcup_{i=1}^{n}\tilt_{P(i)\oplus (P(i)/P(n))}A$, where $P(i)$ denotes the indecomposable projective module corresponding to a vertex $i$ and $\tilt_{X}A:=\{ T\in\tilt A\mid X\in \add T\}$.
Our next main result is the extension of this to quadratic linear Nakayama algebras.

\begin{theorem}[Theorem \ref{thm:decomp}] \label{mainthm3}
Let $A$ be a quadratic linear Nakayama algebra whose Gabriel quiver is $\Vec{\bbA}_{n}$ and let $\ell$ be the maximal integer such that $P(\ell)$ is injective.
Then we have 
\begin{align}
\tilt A =\bigsqcup_{i=\ell}^{n}\tilt_{P(i)\oplus (P(i)/P(n))}A.\notag
\end{align}
\end{theorem}

Furthermore, each $\tilt_{P(i)\oplus (P(i)/P(n))}A$ can be reduced to tilting modules over idempotent-quotient algebras or idempotent subalgebras (Theorems \ref{thm:glue-bij-n} and \ref{thm:glue-bij-except-n}).
As an application, we have a recursive formula of the number of tilting $A$-modules, which coincides with the number of quasi-hereditary structures on $A$ (Corollary \ref{cor:count-by-tilting}). 
In fact, one can use this decomposition to show that $\tilt A$ is a lattice -- we will defer this to a later paper where we give more detailed studies in the poset structure of $\tilt A$.

To finish our investigation on quasi-hereditary structures of linear quadratic Nakayama algebras, we give their explicit descriptions.
The key to this is to describe the quasi-hereditary structures of an algebra obtained by `gluing a node', i.e. identifying the sink vertex of one algebra $A$ to the source vertex of another algebra $B$ and add all quadratic relations around the glued vertex $v$; see subsection \ref{subsec:nodal gluing} for details.

\begin{theorem}[Theorem \ref{thm:pro-qhs-node-gluing}]\label{mainthm4}
Let $C$ be an algebra obtained by gluing a node $v$ from algebras  $A$ and $B$. 
Then there exist bijections between the following sets.
\begin{itemize}
\item The set of quasi-hereditary structures of $C$.
\item The set of pairs of quasi-hereditary structures of $A$ and $B$ that satisfy the condition $(\std)$ of Definition \ref{def:gluing conditions}.
\item The set of pairs of quasi-hereditary structures of $A$ and $B$ that satisfy the condition $(\costd)$ of Definition \ref{def:gluing conditions}.
\end{itemize}
\end{theorem}

For a quadratic linear Nakayama algebra $A$, it can be obtained by iterative nodal gluing of the path algebra $\Bbbk \Vec{\bbA}_{n}$. 
In \cite[Theorem 4.7]{FKR22}, it was shown that quasi-hereditary structure of $\Bbbk \Vec{\bbA}_{n}$ can be described and classified by binary trees.
Consequently, we can apply Theorem \ref{mainthm4} to obtain a combinatorial description of quasi-hereditary structures of $A$ via certain sequence of binary trees (Corollary \ref{cor:qhs-qLNaka-tilt}), and conversely, how to construct a tilting $A$-module from such a sequence.
As a by-product, we obtain an alternative proof of the recursive formula (Corollary \ref{cor:count-by-tilting}) of the number of quasi-hereditary structures of $A$ (Proposition \ref{prop:recurrence-formula-qhs}).

Finally, it should be remarked that the counting of so-called quasi-hereditary orderings in \cite{ZZ} is fundamentally different from our counting result (Corollary \ref{cor:count-by-tilting}, Proposition \ref{prop:recurrence-formula-qhs}). The quasi-hereditary orderings discussed in \cite{ZZ} are specifically linear orders. It is possible for several quasi-hereditary orderings to produce the same standard modules and characteristic tilting modules. From this perspective, quasi-hereditary structures offer a genuine refinement of quasi-hereditary orderings.

\subsection*{Structure of the paper} 
This paper is structured as follows. In Section \ref{sec:prelim}, after recalling some fundamentals about partial order, we collect the essential definitions and results surrounding tilting modules and their mutation in subsection \ref{subsec:tilt mut}.  Then we will recall in subsection \ref{subsec:qhs char tilt} the definition of quasi-hereditary algebras, as well as various facts about characteristic tilting modules.  In subsection \ref{subsec:recolle subcats}, we recall various useful results regarding the passages between the module categories of $A, eAe$ and $A/AeA$.

Section \ref{sec:qhs vs tilting} is devoted to showing Theorem \ref{mainthm1}.  We will investigate a condition that is more general than IS-tilting in the first three subsections.  Then we will focus on IS-tilting modules in the latter two subsections.

The entirety of Section \ref{sec:all tilt = IS-tilt <=>} is used to prove Theorem \ref{mainthm2}.  The first subsection explains the main points in the argument and in particular show one of the direction.  The other direction is shown in the remaining subsections, where we construct a specific sequence of mutation to show that algebras with acyclic Gabriel quiver and are not quadratic linear Nakayama must admit a tilting module without simple direct summand.

Section \ref{sec:tilt(qLNaka)} will be divided into two parts.  We first investigate, in subsection \ref{subsec:co/tilt gluing} some sufficient conditions where (co)tilting modules of an algebra $A$ can be glued from, and decomposed into, (co)tilting modules of $eAe$ and $A/AeA$.  Then in subsection \ref{subsec:glue-tilt(qLNaka)}, we apply this to quadratic linear Nakayama algebras to obtain the decomposition (Theorem \ref{mainthm3}) of the set of tilting modules.  This yields a recursive construction of all tilting modules over quadratic linear Nakayama algebras from those of the path algebra of a linearly oriented quiver $\Vec{\bbA}_{n}$(Theorems \ref{thm:glue-bij-n} and \ref{thm:glue-bij-except-n}); in particular, we obtain a recursive formula in counting the tilting modules (Corollary \ref{cor:count-by-tilting}).

Finally, in Section \ref{sec:qhs description}, we first investigate a `gluing' construction on quasi-hereditary structures for algebras obtained by gluing a node (Theorem \ref{mainthm4}).  Then we apply this to quadratic linear Nakayama algebras to obtain a purely combinatorial description -- certain sequences of binary trees -- of all of their quasi-hereditary structures.

\subsection*{Acknowledgment}

We thank Yuichiro Goto for explaining his work \cite{Goto}, which started this project. 
TA is supported by JSPS Grant-in-Aid for Scientific Research JP24K06671.
AC is supported by JSPS Grant-in-Aid for Scientific Research (C) JP24K06666.
YK was supported by Grant-in-Aid for JSPS Fellows JP22J01155.
MT is supported by JSPS Grant-in-Aid for Early-Career Scientists JP23K12959.

\section{Preliminaries}\label{sec:prelim}

Throughout this paper, we use the following notation. 
Let $\Bbbk$ be a field.
We will specify when $\Bbbk$ is required to be algebraically closed.
For a finite-dimensional algebra $A$ over $\Bbbk$, we will always assume that $\Lambda$ is the indexing set of the isomorphism classes of simple $A$-modules.
We fix $\{ e_{x} \}_{x\in \Lambda}$ to be a complete set of primitive orthogonal idempotents of $A$. For $\emptyset\neq\Gamma \subseteq \Lambda$, let $e_{\Gamma}:=\sum_{x\in \Gamma}e_{x}$ and  $e_{\emptyset}:=0$. 

We assume that all modules are finitely generated right modules, unless otherwise specified.
For each $x\in \Lambda$, let $P(x)=e_{x}A, I(x)=\kD(Ae_{x})$ be the associated indecomposable projective and injective module, and let $S(x)$ be the simple top of $P(x)$.
Here, $\kD:=\Hom_{\Bbbk}(-,\Bbbk)$ denotes the $\Bbbk$-linear dual.

For any $A$-module $M$ and simple $A$-module $S$, $[M:S]$ denotes the multiplicity of $S$ as a composition factor of $M$.  
For an $A$-module $M$, denote by $|M|$ the number of non-isomorphic indecomposable direct summands of $M$.
We identify $\Ext^{0}$ with $\Hom$.  We often use the expression `$\Ext_{A}^{>0}(M,N)=0$' to mean `$\Ext_{A}^{k}(M,N)=0$ for all $k>0$', and similarly for $\Tor_{>0}^{A}(M,N)=0$.

Denote by $\mod A$ the category of $A$-modules.
All subcategories of $\mod A$ considered in this paper are full.
Denote by $\proj A$ the subcategory of projective $A$-modules and by $\inj A$ the subcategory of injective $A$-modules.
For $M\in\mod A$, denote by $\add M$ the subcategory of all modules isomorphic to finite direct sum of direct summands of $M$.

\subsection{Partial order}

We review the basics of partial order and establish some terminologies for later use. Throughout this paper, all posets we consider are finite, unless otherwise specified.

Let $(\Lambda,\unlhd)$ be a poset. When there is no confusion about the underlying set, we will just use $\min\unlhd$ and $\max\unlhd$ to denote the set of minimal and maximal elements respectively.  When using the symbol $\lhd$, we mean a strict relation, i.e. $x\lhd y$ implies that $x\neq y$.
A \defn{total order}, or \defn{linear order}, is a partial order for which any two elements are comparable.
A \defn{linear extension} $\leq$ of a partial order $\unlhd$ is a total order that refines a partial order, i.e. $x\unlhd y$ implies that $x\leq y$ for all $x,y\in \Lambda$. 

We say that $y$ \defn{covers} $x$  if $x\lhd y$, and $x\unlhd z\unlhd y$ implies $z=x$ or $z=y$; we write $x\lhd^{\bullet} y$ for such a case.
The \defn{Hasse quiver} of $\Lambda$ is defined as the directed graph whose vertices are the elements of $\Lambda$ and whose arrows $x\rightarrow y$ correspond to the covering relations $y \lhd^{\bullet} x$.

For a subset $\Gamma\subseteq \Lambda$, denote by $\unlhd|_\Gamma$ the restriction of $\unlhd$ on $\Gamma$.
Recall that $\Gamma$ is a \defn{up-set} of $\Lambda$ if it is a sub(po)set of $\Lambda$ that is closed under being larger, i.e. $y\unrhd x\in \Gamma$ implies that $y\in \Gamma$.  One can define a \defn{down-set} dually.

Suppose that $(\Lambda,\unlhd_\Lambda), (\Gamma,\unlhd_\Gamma)$ are two posets. Then the (Cartesian) \defn{product poset} of them is given by $(\Lambda\times\Gamma, \unlhd:=\unlhd_\Lambda\times\unlhd_{\Gamma})$, where $(\lambda,\gamma) \unlhd (\lambda',\gamma')$ if $\lambda \unlhd_{\Lambda} \lambda'$ and $\gamma\unlhd_{\Gamma} \gamma'$.

\subsection{Tilting modules and mutation}\label{subsec:tilt mut}

In this subsection, we recall a definition and basic properties of tilting modules.

\begin{definition}
An $A$-module $T$ is called a \defn{tilting module} if the following three conditions are satisfied:
\begin{itemize}
\item[(T1)] The projective dimension of $T$ is finite.
\item[(T2)] $\Ext_{A}^{>0}(T,T)=0$.
\item[(T3)] There exists an exact sequence 
\begin{align}
0\rightarrow A \rightarrow T_{0} \rightarrow T_{1} \rightarrow \cdots \rightarrow T_{n}\rightarrow 0 \notag 
\end{align}
such that all $T_{i}\in \add T$.
\end{itemize}
If (T1) and (T2) are satisfied, it is called \defn{pretilting}. 
An $A$-module is said to be \defn{basic} if it is a direct sum of pairwise non-isomorphic indecomposable modules. 
Denote by $\ptilt A$ (respectively, $\tilt A$) the set of basic pretilting (respectively, tilting) $A$-modules up to isomorphism. 

\defn{Precotilting modules} and \defn{cotilting modules} are defined dually, and so are the sets $\pcotilt A$ and $\cotilt A$.
\end{definition}

For an $A$-module $X$, we have a partition of $\tilt A$ as $\tilt A=\tilt_{X}A\sqcup \tilt^{X}A$, where
\begin{align}
\tilt_{X}A:=\{ T\in \tilt A \mid X\in\add T \}, \text{ and } \tilt^{X}A:=\tilt A\setminus \tilt_{X}A=\{ T\in \tilt A \mid X\notin\add T\}. \notag
\end{align}

We collect basic properties of tilting modules.

\begin{proposition}\label{prop:property-tilting}
Let $A$ be a finite-dimensional algebra.
Then the following statements hold.
\begin{itemize}
\item[(1)] $(${\cite{HU05}}$)$ For $A$-modules $M$ and $N$, we write $M\geq N$ if $\Ext_{A}^{>0}(M,N)=0$. Then $\geq$ is a partial order on $\tilt A$.
\item[(2)] $(${\cite[Theorem 1.19]{M86}, \cite[Theorem 10]{W88}}$)$
Let $T$ be a tilting $A$-module. Then the Grothendieck group of $A$ is isomorphic to the Grothendieck group of $\End_{A}(T)$. In particular, we have $|T|=|A|$.
\item[(3)] $(${\cite[Theorem 1]{RS89}}$)$ Assume that $A$ is of finite representation type (i.e. the number of the isomorphism classes of indecomposable $A$-modules is finite).
Let $M$ be an $A$-module satisfying the conditions (T1) and (T2).
Then there exists an $A$-module $N$ such that $M\oplus N$ is a tilting $A$-module. 
In particular, $M$ is a tilting $A$-module if and only if (T3') $|T|=|A|$ holds.
\end{itemize}
\end{proposition}

\begin{remark}\label{rem:tilting/qln}
In the latter part of the paper, we will study tilting modules over a quadratic linear Nakayama algebra $A$.
Since $A$ is representation-finite and of finite global dimension, it follows from Proposition \ref{prop:property-tilting}(3) that an $A$-module $T$ is tilting if and only if it satisfies (T2) $\Ext_{A}^{>0}(T,T)=0$ and (T3') $|T|=|A|$. In this case, all tilting modules are cotilting, and vice versa.
\end{remark}

Recall the notion of tilting mutation and its basic properties, see \cite{HU05, RieS91}.
For $M\in\mod A$ and $\calU$ a subcategory of $\mod A$, a morphism $f:M\rightarrow U$ is called a \defn{left $\calU$-approximation} of $M$ if $U\in\calU$ and $\Hom_{A}(f,U')$ is surjective for all $U'\in \calU$.
It is, furthermore, \defn{minimal} if each morphism $g:U\rightarrow U$ satisfying $gf=f$ is an automorphism. 

\begin{proposition}\label{prop:tilting-mutation}
Let $T=X\oplus U$ be a tilting module with $\add X \cap \add U = 0$ and $f:X\rightarrow U'$ a left $\add U$-approximation.
Assume that $f$ is injective.
Then the following statements hold.
\begin{itemize}
\item[(1)] $\mu_{X}(T):=\Cok f \oplus U$ is a tilting module.
\item[(2)] If $X$ is basic and $f$ is left minimal, then $\Cok f$ is basic and $|X|=|\Cok f|$ holds.
\end{itemize}
We call $\mu_{X}(T)$ the (left) \defn{mutation} of $T$ at $X$, and say that it is \defn{irreducible} if $X$ is furthermore indecomposable.
\end{proposition}

It is known that the Hasse quiver of $\tilt A$ coincides with the mutation quiver whose vertices are tilting $A$-modules and whose arrows are irreducible (left) mutations, that is, we draw an arrow $T\rightarrow U$ if $U$ is an irreducible mutation of $T$.
It is shown in \cite[Corollary 2.2]{HU05} that if $A$ is representation-finite, then all tilting $A$-modules can be obtained by iterative mutation starting from $A$.

\begin{example}\label{eg:A_n^!}
Let $A$ be a radical square zero linear Nakayama algebra whose Gabriel quiver is
\begin{align}
1\xrightarrow{\alpha_{1}} 2 \xrightarrow{\alpha_{2}} \cdots \xrightarrow{\alpha_{n-1}} n.\notag
\end{align}
By (T3), every indecomposable projective-injective module is a direct summand of every tilting module. 
All indecomposable projective $A$-modules, except $P(n)$, are also injective.  
Let $P:=P(1)\oplus \cdots \oplus P(n-1)$. 
Then each basic tilting $A$-module is given by $P\oplus X$ for some indecomposable $A$-module $X$.
For each $i\in [2,n]$, there exists an exact sequence
\begin{align}
0\rightarrow S(i)\xrightarrow{f} P(i-1)\rightarrow S(i-1)\rightarrow 0\notag
\end{align}
with $f$ a minimal left $\add P$-approximation.
By mutation, the Hasse quiver of $\tilt A$ is of the form:
\begin{align}
\tilt A  = \left( A=P\oplus S(n) \rightarrow P\oplus S(n-1) \rightarrow \cdots \rightarrow P\oplus S(1) = DA\right).\notag
\end{align}
\end{example}

\subsection{Quasi-hereditary structures and characteristic tilting modules}\label{subsec:qhs char tilt}

In this subsection, we recall the definition and some basic facts of quasi-hereditary algebras.

Let $A$ be a finite-dimensional $\Bbbk$-algebra.
Recall that for modules $X,P\in\mod A$, the \defn{trace} $\tr_{X}(P)$ of $X$ in $P$ is the submodule of $P$ given by the sum of $\mathrm{Im}(f)$ over all $A$-module homomorphisms $f:X\rightarrow P$.
Recall also that the \defn{reject} $\rej_{X}(P)$ of $X$ in $P$ is the submodule of $P$ given by the intersection of $\Ker(f)$ over all $A$-module homomorphisms $f:P\rightarrow X$.
Note that there is a surjection $N^{\oplus m} \twoheadrightarrow \tr_{N}(M)$ for some $m\ge 0$ and also an injection $M/\rej_{N}(M)\hookrightarrow N^{\oplus m'}$ for some $m'\ge 0$.

Fix a partial order $\unlhd$ on $\Lambda$. 
For $i\in \Lambda$, define \defn{standard $A$-modules} $\std(i)$ and \defn{costandard $A$-modules} $\costd(i)$ with respect to $\unlhd$ as follows.
\begin{align}
\std(i) & := P(i)/\tr_{P}(P(i)) \text{, where }P:=\bigoplus_{j\ntrianglelefteq i} P(j), \notag\\
\costd(i) & := \rej_{I}(I(i)) \text{, where }I:=\bigoplus_{j\ntrianglelefteq i} I(j).\notag
\end{align}
Namely, $\std(i)$ is the maximal quotient module of $P(i)$ such that each composition factor $S(j)$ satisfies $j\unlhd i$, and $\costd(i)$ is the maximal submodule of $I(i)$ such that each composition factor $S(j)$ satisfies $j\unlhd i$.

Suppose that $\calX$ is a set of indecomposable $A$-modules, or an additive subcategory of $\mod A$. We say that an $A$-module $M$ admits an \defn{$\calX$-filtration} if there exists a filtration
\begin{align}
M=M_{0} \supset M_{1} \supset \cdots \supset M_{l-1}\supset M_{l}=0 \notag
\end{align}
such that $M_{j-1}/M_{j}\in \calX$ for all $1\leq j \leq l$.
Let $\filt(\calX)$ be the full subcategory of $\mod A$ consisting of the $A$-modules that admit an $\calX$-filtration.
We will usually consider $\calX$ for the following cases:
\begin{itemize}
\item $\std:=\{\std(i)\mid i\in\Lambda\}$ be the set of standard modules. 
\item $\costd:=\{ \costd(i)\mid i\in \Lambda\}$ be the set of costandard modules. 
\item $\std(\rhd i):=\{ \std(j) \mid j\rhd i\}$, $\std(\lhd i):=\{ \std(j) \mid j\lhd i\}$, and similarly for $\costd(\rhd i)$ and $\costd(\lhd i)$.
\end{itemize}

We recall the definition of quasi-hereditary algebras. 

\begin{definition}[{\cite{CPS88}}]\label{def:qha}
Let $\unlhd$ be a partial order on $\Lambda$. 
The pair $(A,\unlhd)$ is called a \defn{quasi-hereditary algebra} if it satisfies the following conditions:
\begin{itemize}
\item[(qh1)] $[\std(i):S(i)]=1$ for each $i\in \Lambda$.
\item[(qh2)] For each $i\in \Lambda$, the canonical short exact sequence $0\rightarrow K(i) \rightarrow P(i)\rightarrow  \std(i)\rightarrow 0$ has $K(i)\in\filt(\std(\rhd i))$. 
\end{itemize}
We call a partial order $\unlhd$ on $\Lambda$ \defn{quasi-hereditary} if $(A,\unlhd)$ is a quasi-hereditary algebra.
\end{definition}

By \cite[Statement 9]{DR89}, being quasi-hereditary is a left-right symmetric notion, i.e. $(A,\unlhd)$ being quasi-hereditary is equivalent to $(A^\op,\unlhd)$ being quasi-hereditary.  This is also equivalent to having the dual version of (qh1) and (qh2) being satisfied:
\begin{itemize}
\item[(qh1${}^{\op}$)] $[\costd(i):S(i)]=1$ for each $i\in \Lambda$.
\item[(qh2${}^{\op}$)] For each $i\in \Lambda$, the canonical short exact sequence $0\rightarrow \costd(i) \rightarrow I(i)\rightarrow  C(i)\rightarrow 0$ has $C(i)\in\filt(\costd(\rhd i))$. 
\end{itemize}

For $M\in \filt(\std)$, denote by $(M:\std(i))$ the multiplicity of $\std(i)$ in a $\std$-filtration of $M$; this is well-defined for a quasi-hereditary algebra (see \cite{DR92, D98}).  Similarly, we have well-defined multiplicity $(M:\costd(i))$ of $\costd(i)$ for any $M\in \filt(\costd)$. 
Recall that for a quasi-hereditary algebra, we have the following \defn{Brauer--Humphreys reciprocity}; this generalises Bernstein--Gel'fand--Gel'fand reciprocity for the BGG category $\calO$ of complex semisimple Lie algebras, c.f. \cite[A.5]{D98}.

\begin{proposition}[{\cite[Theorem 3.11]{CPS88}, \cite[Lemma 2.5]{DR92}}]\label{prop:BH-recip}
For a quasi-hereditary algebra $(A,\unlhd)$, we have the following equalities for any $i,j\in\Lambda$.
\begin{itemize}
\item[(1)] $(P(i):\std(j))=[\costd(j):S(i)]$.
\item[(2)] $(I(i):\costd(j))=[\std(j):S(i)]$.
\end{itemize}
\end{proposition}

Let $\unlhd'$ be another partial order on $\Lambda$ and $\std':=\{\std'(i)\mid i\in \Lambda\}$ the set of the standard modules with respect to $\unlhd'$.
We say that two partial orders $\unlhd$ and $\unlhd'$ on $\Lambda$ are \defn{equivalent}, and write $\unlhd\sim\unlhd'$, if $\std=\std'$ holds, that is, $\std(i)=\std'(i)$ for each $i\in \Lambda$.
Note that, for a refinement $\unlhd'$ of a quasi-hereditary partial order $\unlhd$, the partial order $\unlhd'$ is also quasi-hereditary and $\unlhd'\sim\unlhd$ by \cite{DR92}. 

Denote by $[\unlhd]$ an equivalence class of a quasi-hereditary partial order $\unlhd$.
Following \cite{FKR22}, we call this class $[\unlhd]$ a \defn{quasi-hereditary structure} of $A$.
We remark that this terminology is used in \cite{Goto} for total order $\le$ for which $(A,\le)$ is quasi-hereditary; this is different from our usage here.
Denote by $\qhs A$ the set of quasi-hereditary structures of $A$.
By left-right symmetry, we have $\qhs A = \qhs (A^{\op})$.

For $[\unlhd_{1}], [\unlhd_{2}]\in \qhs A$, we write $[\unlhd_{1}] \preceq [\unlhd_{2}]$ if $\filt(\costd_{1})\subseteq \filt(\costd_{2})$, or equivalently, $\filt(\std_{1})\supseteq \filt(\std_{2})$, where $\costd_{i}$ is the set of costandard modules with respect to $\unlhd_{i}$ and $\std_{i}$ is the set of standard modules with respect to $\unlhd_{i}$. Then $(\qhs A,\preceq)$ becomes a poset by \cite[Lemma 2.18]{FKR22}. 

Moreover, as shown in \cite[Proposition 2.9]{FKR22}, there is a unique coarsest order in each quasi-hereditary structure, given by the so-called \defn{minimal adapted order} in \cite{DR92, FKR22}, also known as \defn{essential order} in \cite{Cou20}.
For any quasi-hereditary partial order $\unlhd$, one can obtain the equivalent minimal adapted order $\unlhd^{\mathrm{m}}\sim \unlhd$ as follows.

\begin{definition}
Let $\unlhd$ be a partial order on $\Lambda$ and take $i,j\in \Lambda$.
\begin{itemize}
\item[(1)] We write $i \unlhd^{\dec} j$ if $[\std(j):S(i)]\neq 0$.
\item[(2)] We write $i \unlhd^{\inc} j$ if $[\costd(j):S(i)]\neq 0$.
\item[(3)] Define $\unlhd^{\rm m}$ to be the transitive closure of $\unlhd^{\dec}\cup\unlhd^{\inc}$, i.e. $i \unlhd^{\rm m} j$ if there exists a finite sequence $i = i_1 \unlhd^{\ast_{1}} i_{2} \unlhd^{\ast_{2}} \dots \unlhd^{\ast_{l-1}} i_{l} = j$ of elements of $\Lambda$, where $\ast_{a}\in\{\dec, \inc\}$ for all $1\leq a<l$.
\end{itemize}
\end{definition}

\begin{remark}\label{rmk:covering relation => there is path}
When $\unlhd$ is a minimal adapted order, then, by definition, a covering relation $y\lhd^\bullet x$ must be given by either $y\lhd^{\dec}x$ or $y\lhd^{\inc} x$; in particular, there must be a path from $x$ to $y$, or from $y$ to $x$, on the quiver of $A$ that is non-vanishing in $A$.
\end{remark}

Apart from the standard and costandard modules, a quasi-hereditary algebra $(A,\unlhd)$ also comes with another set of structural modules.

\begin{definition-proposition}[\cite{R91}]\label{prop:ringel}
Let $(A, \unlhd)$ be a quasi-hereditary algebra. 
Then there exists a basic tilting $A$-module $T$ such that $\add T=\filt(\std)\cap\filt(\costd)$. 
Moreover, there exists an indecomposable decomposition $T=\bigoplus_{i\in\Lambda} T(i)$ such that 
\begin{align}
[T(i):S(i)]=1\text{, and } [T(i):S(j)]\neq 0 \text{ implies }j\unlhd i.
\end{align}
The module $T$ is called the \defn{characteristic tilting module} associated to $(A,\unlhd)$.
When the partial order $\unlhd$ needs to be clarified, we write $T_{\unlhd}:=T$.
Denote by $\chtilt A$ the set of all possible characteristic tilting $A$-modules up to isomorphism.
\end{definition-proposition}

The following says that there is no difference between the posets $\chtilt A$ and $\qhs A$; this can be found in \cite{FKR22} but is essentially a result of \cite{R91}.

\begin{proposition}[{\cite[Lemmas 2.18, 2.20]{FKR22}}]\label{prop:FKR-lems}
For quasi-hereditary partial orders $\unlhd_{1}$ and $\unlhd_{2}$, the relation $\unlhd_{1}\sim \unlhd_{2}$ holds if, and only if, their associated characteristic tilting modules $T_{\unlhd_{1}}$ and $T_{\unlhd_{2}}$ are isomorphic. 
In particular, the map
\begin{align}
\qhs A \longrightarrow \chtilt A\notag
\end{align}
given by $[\unlhd]\mapsto T_{\unlhd}$ is a poset isomorphism.
\end{proposition}

The bijection below, established via a combinatorial argument, provides a motivating example for various main results of this paper.

\begin{proposition}[{\cite[Theorem 4.7]{FKR22}}]\label{prop:FKR-thm4.7}
Assume that $A$ is the path algebra of quiver
\begin{align}
\vec{\bbA}_n: 1\rightarrow 2\rightarrow \cdots \rightarrow n.\notag
\end{align}
Then all tilting $A$-modules are characteristic tilting.
In particular, there exists a poset isomorphism $\qhs A \rightarrow \tilt A$.
\end{proposition}

\subsection{Properties on the functors induced by an idempotent}\label{subsec:recolle subcats}

Fix an idempotent $\e\in A$.
Let $\sfe:=(-)\otimes_{A}A\e=\Hom_{A}(\e A,-):
\mod A \rightarrow \mod \e A\e$ be the \defn{idempotent truncation} functor (Schur functor) associated to $\e$.
Note that there is an isomorphism of algebras $\e A\e\cong\End_{A}(\e A)$.
Moreover, $\sfe$ has a left adjoint $\sfl:=(-)\otimes_{\e A\e}\e A$ and a right adjoint $\sfr:=\Hom_{\e A\e}(A\e,-)$, which induces a recollement of module categories
\begin{align}
\xymatrix{
\mod (A/A\e A)\ar[rrrr]^{\text{inc}}&&&&\mod A\ar[rrrr]^{\sfe=(-)\otimes_{A}A\e}\ar@/_20pt/[llll]\ar@/^20pt/[llll]
&&&& \mod \e A\e.\ar@/_20pt/[llll]_{\sfl=(-)\otimes_{\e A\e}\e A}
\ar@/^20pt/[llll]^{\sfr=\Hom_{\e A\e}(A\e,-)}
}\notag
\end{align}

The following lemma is frequently used in this paper. 

\begin{lemma}[{\cite[Proposition 1.3]{APT92}, \cite[Remark 2.1.2]{CPS96}}]\label{lem:str-idemp}
For an idempotent $\e\in A$, the following statements are equivalent. 
\begin{itemize}
\item[(1)] The multiplication map $A\e\otimes_{\e A\e}\e A\rightarrow A\varepsilon A$ is an isomorphism and $\Tor^{\e A\e}_{>0}(A\e, \e A)=0$. 
\item[(2)] $\Ext_{A/A\e A}^{i}(X,Y)\cong \Ext_{A}^{i}(X,Y)$ for all $X, Y\in\mod (A/A\e A)$ and $i\geq 0$. 
\item[(3)] $\Ext_{A}^{>0}(A/A\e A,Y)=0$ for all $Y\in\mod (A/A\e A)$.
\item[(4)] $\Tor^{A}_{>0}(X, A/A\e A)=0$ for all $X\in\mod (A/A\e A)$.
\end{itemize}
\end{lemma}

\begin{definition}
Let $A$ be a finite-dimensional algebra and $\varepsilon$ an idempotent of $A$. 
We call $A\e A$ a \defn{stratifying ideal} if it satisfies the equivalent conditions in Lemma \ref{lem:str-idemp}. In addition, if $\varepsilon A\varepsilon$ is semisimple, then $A\varepsilon A$ is called a \defn{heredity ideal} of $A$.
\end{definition}

A stratifying ideal is also called strong idempotent ideal in \cite{APT92}.
Remark that $A\varepsilon A$ is a heredity ideal if and only if $A\varepsilon A\in \proj A$ and $\varepsilon A\varepsilon$ is semisimple by \cite[Corollary 5.3]{APT92}. 

The following lemma is useful.

\begin{lemma}\label{lem:proj-stridemp}
For an idempotent $\varepsilon\in A$, we consider the following conditions: 
\begin{itemize}
\item[(1)] $A\varepsilon A\in\proj A$;
\item[(2)] $A\varepsilon A\in\add(\varepsilon A)$; 
\item[(3)] $A\varepsilon\in\proj(\varepsilon A\varepsilon)$;
\item[(4)] $A\varepsilon A$ is a stratifying ideal.
\end{itemize}
Then (1)$\Rightarrow$(2)$\Rightarrow$(3) and (1)$\Rightarrow$(4) hold. 
\end{lemma}

\begin{proof}
(1)$\Rightarrow$(2)$\Rightarrow$(3) is a basic result.  See, for example, \cite[Lemma 3.8]{T20}. (1)$\Rightarrow$(4) follows from \cite[Statement 7]{DR89}.
\end{proof}

For any $k\ge 0$ and $\Gamma\subseteq\Lambda$, denote by $\calP^{\leq k}(\Gamma)$ the subcategory of $\mod A$ consisting of the $A$-modules $X$ admitting a projective resolution $\cdots \rightarrow  P_{1}\rightarrow P_{0}\rightarrow X \rightarrow 0$ with the first $k+1$ terms in $\add(e_{\Gamma} A)=\add(\bigoplus_{j\in \Gamma}P(j))$.  
We use the convention that $\calP^{\leq \infty}(\Gamma)$, which consists of modules admitting a projective resolution with all terms in $\add e_{\Gamma} A$. 
Dually, for $0\le k\le\infty$, denote by $\calI^{\leq k}(\Gamma)$ the subcategory of $A$-modules admitting an injective coresolution with the first $k+1$ terms in $\add(\mathbb{D}(Ae_{\Gamma}))= \add(\bigoplus_{j\in \Gamma}I(j))$.
These two categories can be alternatively characterised as follows.

\begin{lemma}[{\cite[Proposition 2.4, Theorem 3.6]{APT92}}]\label{lem:apt-pfin}
Let $\Gamma\subseteq \Lambda$.
For each $1\leq k\leq \infty$, we have \begin{align}
\calP^{\leq k}(\Gamma) &= \{ M \in \mod A \mid \Ext_{A}^{i}(M, S(x))=0 \text{ for all }0\leq i \leq k, \text{ for all }x\notin \Gamma\}\notag\\
&= \{ M \in \mod A \mid \Ext_{A}^{i}(M, N)\cong \Ext_{\varepsilon A\varepsilon}^{i}(M\varepsilon, N\varepsilon)\text{ for all }0\leq i \leq k-1, \text{ for all }N\in \mod A\}\notag
\end{align}
and 
\begin{align}
\calI^{\leq k}(\Gamma) &=  \{ N \in \mod A \mid \Ext_{A}^{i}(S(x), N)=0\text{ for all }0\leq i \leq k, \text{ for all }x\notin \Gamma\}\notag\\
&= \{ N \in \mod A \mid \Ext_{A}^{i}(M, N)\cong\Ext_{\varepsilon A\varepsilon}^{i}(M\varepsilon, N\varepsilon)\text{ for all }0\leq i \leq k-1, \text{ for all }M\in \mod A\}.\notag
\end{align}
\end{lemma}

The following equivalences play an important role in this paper.

\begin{lemma}[{\cite[Proposition 5.2]{A74}, \cite[Proposition 3.7]{APT92}}]\label{lem:aus-apt}
The following statements hold for $\e=e_{\Gamma}$.
\begin{itemize}
\item[(1)] The functors $\sfe$ and $\sfl$ induce an equivalence of categories
\begin{align}
\xymatrix{\calP^{\leq 1}(\Gamma)\ar@<-1mm>[r]_-{\sfe} & \mod (\e A\e)\ar@<-1mm>[l]_-{\sfl}}.\notag
\end{align}
Moreover, the restriction induces an equivalence of (exact) categories
\begin{align}
\xymatrix{\calP^{\leq \infty}(\Gamma)\ar@<-1mm>[r]_-{\sfe} & \mod^{\lex}(\e A \e):=\{ V\in \mod (\e A\e)\mid \Tor_{>0}^{\e A\e}(V,\e A)=0\}\ar@<-1mm>[l]_-{\sfl}}.\notag
\end{align}
\item[(2)] The functors $\sfe$ and $\sfr$ induce an equivalence of categories
\begin{align}
\xymatrix{\calI^{\leq 1}(\Gamma)\ar@<1mm>[r]^-{\sfe} & \mod (\e A\e)\ar@<1mm>[l]^-{\sfr}}.\notag
\end{align}
Moreover, the restriction induces an equivalence of (exact) categories
\begin{align}
\xymatrix{\calI^{\leq \infty}(\Gamma)\ar@<1mm>[r]^-{\sfe} & \mod^{\rex}(\e A \e):=\{ V\in \mod (\e A\e)\mid \Ext_{\e A\e}^{>0}(A\e, V)=0\}\ar@<1mm>[l]^-{\sfr}}.\notag
\end{align}
\end{itemize}
\end{lemma}

The equivalences in Lemma \ref{lem:aus-apt} induce a bijection between the sets of tilting objects. 
We say that $T$ is a \defn{pretilting object} of $\calP^{\leq \infty}(\Gamma)$ if it is pretilting $A$-module contained in $\calP^{\leq \infty}(\Gamma)$.
In addition, if it satisfies 
\begin{itemize}
\item[(TP3)] there exists an exact sequence $0\rightarrow \e A \rightarrow T_{0}\rightarrow T_{1}\rightarrow \cdots \rightarrow T_{l}\rightarrow 0$ such that all $T_{j}\in \add T$, 
\end{itemize}
then we call $T$ a \defn{tilting object} of $\calP^{\leq \infty}(\Gamma)$. 
Let $\ptilt\calP^{\leq\infty}(\Gamma)$ (respectively, $\tilt\calP^{\leq\infty}(\Gamma)$) be the set of basic pretilting (respectively, tilting) objects in $\calP^{\leq \infty}(\Gamma)$ up to isomorphisms. 
One can define a \defn{precotilting object} and a \defn{cotilting object} in $\calI^{\leq \infty}(\Gamma)$ dually; denote by $\pcotilt \calI^{\leq \infty}(\Gamma)$ (respectively, $\cotilt \calI^{\leq \infty}(\Gamma)$) the associated set.

\begin{proposition}\label{prop:ex-tilting}
The following statements hold for an idempotent $\e:=e_{\Gamma}$.
\begin{itemize}
\item[(1)] We have mutually inverse bijections 
\begin{align}
\xymatrix{\ptilt\calP^{\leq\infty}(\Gamma)\ar@<-1mm>[r]_-{\sfe} &\ptilt^{\lex}(\e A\e):=\ptilt(\e A\e)\cap\mod^{\lex}(\varepsilon A\varepsilon)\ar@<-1mm>[l]_-{\sfl}}.\notag
\end{align}
Moreover, the restrictions induce mutually inverse bijections 
\begin{align}
\xymatrix{\tilt\calP^{\leq\infty}(\Gamma)\ar@<-1mm>[r]_-{\sfe} &\tilt^{\lex}(\e A\e):=\tilt(\e A\e)\cap\mod^{\lex}(\varepsilon A\varepsilon)\ar@<-1mm>[l]_-{\sfl}}.\notag
\end{align}
\item[(2)] We have mutually inverse bijections
\begin{align}
\xymatrix{\pcotilt\calI^{\leq\infty}(\Gamma)\ar@<1mm>[r]^-{\sfe} &\pcotilt^{\rex}(\e A\e):=\pcotilt(\e A\e)\cap\mod^{\rex}(\varepsilon A\varepsilon)\ar@<1mm>[l]^-{\sfr}}.\notag
\end{align}
Moreover, the restrictions induce mutually inverse bijections
\begin{align}
\xymatrix{\cotilt\calI^{\leq\infty}(\Gamma)\ar@<1mm>[r]^-{\sfe} &\cotilt^{\rex}(\e A\e):=\cotilt(\e A\e)\cap\mod^{\rex}(\varepsilon A\varepsilon)\ar@<1mm>[l]^-{\sfr}}.\notag
\end{align}
\end{itemize}
\end{proposition}

\begin{proof}
Since the proof is similar, we only prove (1). 
Let $T\in\ptilt\calP^{\leq \infty}(\Gamma)$. 
Then it follows from Lemma \ref{lem:apt-pfin} that $\Ext_{A}^{i}(T,N)\cong \Ext_{\varepsilon A\varepsilon}^{i}(T\varepsilon, N\varepsilon)$ holds for all $N\in\mod A$ and $i\geq 0$. 
By $T\in\ptilt\calP^{\leq \infty}(\Gamma)$, we have $T\varepsilon\in\ptilt(\varepsilon A\varepsilon)$, and hence $T\varepsilon\in\ptilt(\varepsilon A\varepsilon)\cap \mod^{\lex}(\varepsilon A\varepsilon)$ by Lemma \ref{lem:aus-apt}(1).  
Let $V\in\ptilt^{\lex}(\varepsilon A\varepsilon)$. 
Then $\sfl(V)\in\calP^{\leq \infty}(\Gamma)$ by Lemma \ref{lem:aus-apt}(1). 
Since it follows from Lemma \ref{lem:apt-pfin} that 
\[\Ext_{A}^{i}(\sfl(V), N)\cong \Ext_{\varepsilon A\varepsilon}^{i}(\sfe\sfl(V),\sfe(N))\cong \Ext_{\varepsilon A\varepsilon}^{i}(V,N\varepsilon)\]
for all $N\in\mod A$ and $i\geq 1$, we have $\sfl(V)\in\ptilt\calP^{\leq \infty}(\Gamma)$ by $V\in\ptilt^{\lex}(\varepsilon A\varepsilon)$. Thus we obtain the former assertion by Lemma \ref{lem:aus-apt}(1). 

We show the latter assertion. Assume that there exists an exact sequence $0\rightarrow \e A \rightarrow T_{0}\rightarrow T_{1}\rightarrow \cdots \rightarrow T_{l}\rightarrow 0$ in $\mod A$ such that all $T_{j}\in \add T$. Since $\sfe$ is exact, we have an exact sequence $0\rightarrow \e A\e \rightarrow T_{0}\e\rightarrow T_{1}\e\rightarrow \cdots \rightarrow T_{l}\e\rightarrow 0$ such that all $T_{j}\e\in \add(T\e)$. 
Conversely, assume that $V\in\mod^{\lex}(\varepsilon A\varepsilon)$ and there exists an exact sequence 
$0\rightarrow \e A\e \rightarrow V_{0}\rightarrow \cdots \rightarrow V_{l}\rightarrow 0$
in $\mod (\varepsilon A\varepsilon)$ such that $V_{i}\in \add V$. 
Since $\sfl$ is exact on $\mod^{\lex}(\varepsilon A\varepsilon)$, there exists an exact sequence 
$0\rightarrow \sfl(\e A\e)\cong \varepsilon A \rightarrow \sfl(V_{0})\rightarrow \cdots \rightarrow \sfl(V_{l})\rightarrow 0.$
Thus the proof is complete. 
\end{proof}

\section{Quasi-hereditary structures via tilting modules}\label{sec:qhs vs tilting}

In this section, we introduce the notion of IS-tilting modules and study their relationship with quasi-hereditary structures. 
Throughout, we fix a finite-dimensional algebra $A$ over an algebraically closed field $\Bbbk$, and let $\Lambda$ be the indexing set of the isomorphism classes of simple $A$-modules.

\subsection{Higher rigid modules with a semisimple direct summand}

In this subsection, we collect basic properties of $k$-rigid modules with a semisimple direct summand. 
For $1\leq k\leq \infty$, an $A$-module $M$ is said to be \defn{$k$-rigid} if $\Ext_{A}^{j}(M,M)=0$ for all $1\leq j\leq k$.
For a semisimple module $S=\bigoplus_{x\in\Gamma}S(x)$, note that $M/\rej_{S}(M),\tr_{S}(M)\in \add S$.
Thus we have $M\e\cong \rej_{S}(M)\e\cong (M/\tr_{S}(M))\e$, where $\e=e_{ \Lambda\setminus\Gamma}$.

\begin{lemma}\label{lem:rgd-smpl}
Let $M$ be a $k$-rigid module with a semisimple direct summand $S$.
For $N\in \add M$, the following statements hold.
\begin{itemize}
\item[(1)] For all $1\leq j\leq k-1$, we have $\Ext_{A}^{j}(\rej_{S}(M),N)=0$ and $\Ext_{A}^{j}(N,M/\tr_{S}(M))=0$.
\item[(2)] If $N$ is in $\add S$, then $\Hom_{A}(\rej_{S}(M),N)=0$ and $\Hom_{A}(N,M/\tr_{S}(M))=0$.
\end{itemize}
\end{lemma}

\begin{proof}
We show only the rejection part.
Consider the canonical short exact sequence
\begin{align}\label{seq:rej-seq}
0\rightarrow \rej_{S}(M)\xrightarrow{f} M\rightarrow M/\rej_{S}(M)\rightarrow 0.
\end{align}
Let $N\in \add M$. 
Applying $\Hom_{A}(-,N)$ to the exact sequence \eqref{seq:rej-seq} yields an exact sequence
\begin{align}
\Ext_{A}^{j}(M,N)\xrightarrow{\Ext_{A}^{j}(f, N)} \Ext_{A}^{j}(\rej_{S}(M),N)\rightarrow \Ext_{A}^{j+1}(M/\rej_{S}(M),N)\notag
\end{align}
for each $j\geq 0$.
Since $M$ is $k$-rigid and $M/\rej_{S}(M)\in \add M$ holds, we have $\Ext_{A}^{j}(M,N)=0$ for all $j\in [1,k]$ and $\Ext_{A}^{j+1}(M/\rej_{S}(M),N)=0$ for all $j\in [0, k-1]$.
Thus $\Ext_{A}^{j}(\rej_{S}(M),N)=0$ holds for all $j\in[1,k-1]$.
Moreover, $\Hom_{A}(f,N)$ is surjective.
This implies that each morphism in $\Hom_{A}(\rej_{S}(M), N)$ can be extended to a morphism in $\Hom_{A}(M,N)$.
If $N\in \add S$, then we have $\Hom_{A}(\rej_{S}(M),N)=0$ by definition of rejections.
Thus we have the assertion.
\end{proof}

\begin{lemma}\label{lem:Ifin}
Let $M$ be a $k$-rigid module and let $S:=\bigoplus_{x\in \Gamma}S(x)$ be a semisimple direct summand of $M$ for some $\Gamma\subsetneq \Lambda$. 
Then the following statements hold.
\begin{itemize}
\item[(1)] $\rej_{S}(M)\in \calP^{\leq k-1}(\Lambda\setminus\Gamma)$ and $M/\tr_{S}(M)\in \calI^{\leq k-1}(\Lambda\setminus\Gamma)$.
\item[(2)] $\rej_{S}(M)$ and $M/\tr_{S}(M)$ are $(k-1)$-rigid.
\end{itemize}
\end{lemma}

\begin{proof}
We show only the rejection part.

(1) By Lemma \ref{lem:rgd-smpl}, we have $\Ext_{A}^{j}(\rej_{S}(M),S)=0$ for all $j\in [0,k-1]$.
Thus the assertion follows from Lemma \ref{lem:apt-pfin}.

(2) Consider the canonical short exact sequence $0\rightarrow \rej_{S}(M)\rightarrow M\rightarrow M/\rej_{S}(M)\rightarrow 0$.
Applying $\Hom_{A}(\rej_{S}(M),-)$ to the exact sequence above yields an exact sequence 
\begin{align}
\Ext_{A}^{j-1}(\rej_{S}(M), M/\rej_{S}(M))\rightarrow \Ext_{A}^{j}(\rej_{S}(M),\rej_{S}(M))\rightarrow \Ext_{A}^{j}(\rej_{S}(M),M).\notag 
\end{align}
By Lemma \ref{lem:rgd-smpl}(1), we have $\Ext_{A}^{j}(\rej_{S}(M),M)=0$ for all $j\in [1,k-1]$.
Since $M/\rej_{S}(M)$ is in $\add S$, it follows from Lemma \ref{lem:rgd-smpl} that the first term of the sequence above vanishes for all $j\in [1,k]$.
Hence we obtain $\Ext_{A}^{j}(\rej_{S}(M),\rej_{S}(M))=0$ for all $j\in [1,k-1]$.
\end{proof}

\begin{proposition}\label{prop:truncation-tilting}
Let $M$ be an $A$-module and let $S:=\bigoplus_{x\in \Gamma}S(x)$ be a semisimple direct summand of $M$ for some $\Gamma\subsetneq \Lambda$.
If $M$ is $k$-rigid (respectively, tilting), then $M\e$ is a $(k-2)$-rigid (respectively, tilting) $\e A\e$-module, where $\e=e_{ \Lambda\setminus\Gamma}$. 
\end{proposition}

\begin{proof}
Assume that $M$ is a $k$-rigid $A$-module.
We show that $M\e$ is a $(k-2)$-rigid $\e A\e$-module.
Since $M\e\cong \rej_{S}(M)\e$ holds, it is enough to claim that $\rej_{S}(M)\e$ is $(k-2)$-rigid.
By Lemmas \ref{lem:apt-pfin} and \ref{lem:Ifin}(1), we have a natural isomorphism 
\begin{align}\label{seq:APT-isom}
\Ext_{A}^{j}(\rej_{S}(M), -)\cong \Ext_{\e A\e}^{j}(\rej_{S}(M)\e,(-)\e)
\end{align}
for all $j\in [1,k-2]$.
Since $\rej_{S}(M)$ is $(k-1)$-rigid by Lemma \ref{lem:Ifin}(2), the module $\rej_{S}(M)\e$ is $(k-2)$-rigid.

Assume that $M$ is a tilting $A$-module. 
By the first argument, $M\e$ satisfies (T2).
Due to the isomorphism \eqref{seq:APT-isom} for $k=\infty$, we have 
\begin{align}
\pdim(M\e)=\pdim(\rej_{S}(M)\e)=\pdim(\rej_{S}(M))\leq \pdim M <\infty.\notag
\end{align}
By (T3) for $M$, there exists an exact sequence 
\begin{align}
0\rightarrow \varepsilon A \rightarrow M_{0}\rightarrow M_{1}\rightarrow \cdots \rightarrow M_{l} \rightarrow 0 \notag
\end{align}
such that all $M_{j}\in \add M$. Applying the exact functor $-\otimes_{A}A\e$ to the sequence above yields an exact sequence
\begin{align}
0\rightarrow \varepsilon A\varepsilon \rightarrow M_{0}\e \rightarrow M_{1}\e\rightarrow \cdots \rightarrow M_{l}\e \rightarrow 0 \notag
\end{align}
with all $M_{j}\e\in \add(M\e)$. Thus $M\e$ is tilting. The proof is complete.
\end{proof}

\subsection{The condition (IST)}

We will now consider a class of modules satisfying \underline{\textbf{i}}terative elimination of \underline{\textbf{s}}imple direct summand under idempotent \underline{\textbf{t}}runcation -- we abbreviate this condition to `IST'.

We will use the following notation.
Let $(\Lambda,\unlhd)$ be a poset.
For $x\in \Lambda$, we consider the following up-set and associated idempotent:
\begin{align}
\Lambda_{x}:=\{y\in\Lambda\mid  y\ntriangleleft x\},\;\;
\e_{x}:= e_{\Lambda_{x}}= \sum_{y\in \Lambda_{x}} e_{y}. \notag
\end{align}
Note that $x$ is a minimal element in $\Lambda_{x}$.

\begin{definition}\label{definiton:IST}
Let $M$ be an $A$-module and $\unlhd$ a partial order on $\Lambda$.
For $x\in\Lambda$ and an indecomposable direct summand $M(x)$ of $M$, consider the following conditions.
\begin{itemize}
\item[(I1)] $M(x)\e_{x}\cong S(x)\e_{x}$ in $\mod(\e_{x}A\e_{x})$. 
\item[(I2)] $M(x)e_{\Gamma}$ is an indecomposable $e_{\Gamma}Ae_{\Gamma}$-module for all up-sets $\Gamma\subseteq \Lambda$ containing $x$. 
\end{itemize}
For $d\in\{1,2\}$, we say that $(M,\unlhd)$ \defn{satisfies \rm{(I}$d$\rm{)}} with respect to (indecomposable) decomposition $M=\bigoplus_{x\in\Lambda}M(x)$ if (I$d$) holds for all $x\in\Lambda$.  If $(M,\unlhd)$ satisfies both (I1) and (I2) with respect to the same decomposition, then we say that $(M,\unlhd)$ \defn{satisfies \rm{(IST)}}.
\end{definition}

If $(M,\unlhd)$ satisfies (I1) with respect to indecomposable decomposition $M=\bigoplus_{x\in\Lambda}M(x)$, then $M$ is basic. Suppose that $M(x)\cong M(y)$ with $x\neq y$. 
We may assume $y\in\Lambda_{x}$.
By (I1), we have $0\neq S(y)e_{y}\cong M(y)e_{y}\cong M(x)e_{y}\cong S(x)e_{y}=0$, a contradiction. 

If $(M,\unlhd)$ satisfies (IST), then $M(x)=S(x)$ for each minimal element $x$ in $(\Lambda,\unlhd)$.

We give an example of a pair $(M,\unlhd)$ satisfying (IST).

\begin{example}
Let $A=\Bbbk Q/I$ be the bound quiver algebra given by
\begin{align}
Q:\;\;\xymatrix{1\ar@<0.7mm>[r]^-{a}&2\ar@<0.7mm>[l]^-{b}\ar[r]^-{c}&3\ar[r]^-{d}&4} \;\; \text{ and } \;\;I=\langle ac, ba \rangle.\notag
\end{align}
Consider the module $M=S(1)\oplus P(1)\oplus (P(2)/\rad^{2} P(2)) \oplus P(2)$.  The Loewy structures of $A$ and $M$ are as follows.
\begin{align}
A = \sm{1\\2\\1} \oplus \sm{\phantom{9}2\phantom{9}\\1\phantom{9}3\\\phantom{9}\phantom{9}4} \oplus \sm{3\\4} \oplus \sm{4}, \;\;
M=\sm{1}\oplus \sm{1\\2\\1} \oplus \sm{\phantom{9}2\phantom{9}\\1\phantom{9}3}\oplus \sm{\phantom{9}2\phantom{9}\\1\phantom{9}3\\\phantom{9}\phantom{9}4}.\notag
\end{align}
Let $(\Lambda,\unlhd)=\{1\lhd2\lhd3\lhd4\}$. Then there is an indecomposable decomposition $M=\bigoplus_{x\in\Lambda}M(x)$, where 
\begin{align}
M(1)=\sm{1},\;\; M(2)=\sm{1\\2\\1},\;\; M(3)=\sm{&2&\\1&&3},\;\; M(4)=\sm{&2&\\1&&3\\&&4},\notag 
\end{align}
so that $M(x)\varepsilon_{x}\cong S(x)\varepsilon_{x}$ for all $x\in\Lambda$.
Thus $(M,\unlhd)$ satisfies (I1). 
Since $(\Lambda,\unlhd)$ is a totally ordered set, if $\Lambda'\subseteq\Lambda$ is an up-set, then $\Lambda'=\Lambda_{x}$ for some $x\in\Lambda$. 
Hence we can easily check that $(M,\unlhd)$ satisfies (I2).
\end{example}

We collect basic properties of the condition (IST).

\begin{lemma}\label{lem:IST}
Let $N$ be an $A$-module, $(\Lambda,\unlhd)$ be a poset, and $x\in\Lambda$. 
Then $N\e_{x}\cong S(x)\e_{x}$ holds if, and only if, 
\begin{align} 
[N:S(x)]=1,\text{ and } [N:S(y)]\neq 0\text{ implies }y \unlhd x. \notag
\end{align}
In this case, we have $Ne_{\Gamma}=0$ for any up-set $\Gamma$ not containing $x$.
\end{lemma}

\begin{proof}
Fix an element $x\in \Lambda$.
First, we show that the `only if' part. 
Since $Ne_{z}\cong S(x)e_{z}$ for all $z\in\Lambda_{x}$, 
\begin{align}
[N:S(z)]=[S(x):S(z)] =\begin{cases}
1, & \text{if $z=x$};\\
0, & \text{if $z\neq x$}.
\end{cases}\notag    
\end{align} 
Hence $[N:S(y)]\neq0$ implies that $y \unlhd x$.
Next, we show that the `if' part. 
First note that the condition `$[N:S(y)]\neq 0$ implies $y\unlhd x$' says that $\sum_{y\in\Lambda_x}[N:S(y)] = [N:S(x)]$, which in turns is equal to 1 by the other condition.  Hence, we have
\[\dim_{\Bbbk} N\varepsilon_{x}=\dim_{\Bbbk} \Hom_A(\e_x A , N)=\sum_{y\in\Lambda_x}[N:S(y)]= [N:S(x)] = 1.
\]
Therefore, we have $N\varepsilon_{x}\cong S(y)\varepsilon_{x}$ for exactly one $y\in\Lambda_x$, which means that $y$ must be $x$ itself.

For the final statement, suppose to the contrary that there exists $y\in\Gamma$ such that $[N:S(y)]\neq 0$. 
Then $y\unlhd x$ by the first part of the assertion.
Since $\Gamma$ is an up-set, we have $x\in \Gamma$, a contradiction.
\end{proof} 

With Lemma \ref{lem:IST} in hand, we can show that there is a unique labelling $M(x)$ of indecomposable direct summands of $M$ for a pair $(M,\unlhd)$ satisfying (I1).

\begin{lemma}\label{lem:unique-label}
Let $(\Lambda,\unlhd)$ be a poset. 
Assume that $M$ admits two indecomposable decompositions $M=\bigoplus_{x\in \Lambda}M(x)$ and $M=\bigoplus_{x\in\Lambda}M'(x)$ for which (I1) is satisfied.  Then $M(x)\cong M'(x)$ holds for all $x\in\Lambda$. 
 \end{lemma}

\begin{proof}
Assume $M(x)\cong M'(y)$. 
Then applying $\varepsilon_{x}$ to $M(x)\cong M'(y)$, we have $S(x)\varepsilon_{x}\cong M(x)\varepsilon_{x}\cong M'(y)\varepsilon_{x}$ by (I1). 
Since $[M'(y):S(x)]\neq0$, it follows from Lemma \ref{lem:IST} that $x\unlhd y$. 
Similarly, we obtain $y\unlhd x$. 
\end{proof}

Lemma \ref{lem:unique-label} says that the phrase `$(M,\unlhd)$ satisfies (I1)' makes sense without explicitly referring the exact labelling on its indecomposable direct summands.

We show that (IST) is a condition preserved by idempotent truncation; this is the key to invoking induction.

\begin{lemma}\label{lem:basic-results-ist} 
Assume that $(M,\unlhd)$ satisfies (IST).   
Then the following statements hold. 
\begin{itemize}
\item[(1)] For any linear extension $\leq$ of $\unlhd$, the pair $(M, \leq)$ satisfies (IST).
\item[(2)] Let $\Lambda'\subseteq \Lambda$ be an up-set. Then $(Me_{\Lambda'}, \unlhd|_{\Lambda'})$ satisfies (IST). 
\end{itemize}
\end{lemma}

\begin{proof}
(1) We show (I1). 
For $x\in\Lambda$, let $\e'_{x}:=\sum_{x\leq y}e_{y}$. Then we have $\varepsilon_{x}\varepsilon'_{x}=\varepsilon'_{x}$. Applying $\varepsilon'_{x}$ to $M(x)\varepsilon_{x}\cong S(x)\varepsilon_{x}$ yields an isomorphism $M(x)\varepsilon'_{x}\cong S(x)\varepsilon'_{x}$. We show (I2). Let $\Gamma$ be an up-set of $\Lambda$ with respect to $\leq$. 
Then $\Gamma$ is also an up-set of $\Lambda$ with respect to $\unlhd$. Thus (I2) follows from (I2) for $(M,\unlhd)$. 
Thus $(M, \leq)$ satisfies (IST).
    
(2) We can assume that $\Lambda\neq \Lambda'$; otherwise, there is nothing to show.  Let $\varepsilon:=e_{\Lambda'}$. 
Now we decompose $M':=M\e=\bigoplus_{y\in\Lambda'} M'(y)$ as follows. By (I1) and Lemma \ref{lem:IST}, we have $M(x)\e=0$ for $x\in\Lambda\setminus\Lambda'$. 
Thus it follows from (I2) for $(M,\unlhd)$ that there are exactly $|\Lambda'|$ indecomposable direct summands of $M'$, each of which takes the form $M(y)\e$ for $y\in\Lambda'$. 
Hence $M'(y)=M(y)\e$. 

We show that (I1) is satisfied with respect to $M'=\bigoplus_{y\in\Lambda'}M'(y)$. 
For $x\in\Lambda'$, let $\varepsilon'_{x}:=\sum_{y\in\Lambda'_{x}}e_{y}$. 
Since $\Lambda'_{x}\subseteq \Lambda_{x}\subseteq\Lambda$, we have $\varepsilon_{x}\varepsilon'_{x}=\varepsilon'_{x}=\varepsilon\varepsilon'_{x}$. 
By (I1) for $(M,\unlhd)$, we obtain $S(x)\varepsilon'_{x}=(S(x)\varepsilon_{x})\varepsilon'_{x}\cong (M(x)\varepsilon_{x})\varepsilon'_{x}=(M(x)\varepsilon)\varepsilon'_{x}$ for all $x\in\Lambda'$. 
Finally, since any up-set of $\Lambda'$ is also an up-set of $\Lambda$, (I2) for $(M',\unlhd|_{\Lambda'})$ is inherited directly from $(M,\unlhd)$. 
\end{proof}

We give a sufficient criteria for $(M,\unlhd)$ to satisfy (IST). 

\begin{lemma}\label{lem:suff-i1}
Assume that $(M,\unlhd)$ satisfies (I1).
Consider the following condition. 
\begin{itemize}
\item[(R2)]  $Me_{\Lambda'}$ is a $2$-rigid $e_{\Lambda'}Ae_{\Lambda'}$-module for every up-set $\Lambda'\subseteq \Lambda$. 
\end{itemize}
If $(M,\unlhd)$ satisfies (R2), then it also satisfies (I2).
In particular, if a pair $(M,\unlhd)$ satisfies both (I1) and (R2), then it also satisfies (IST).
\end{lemma}

\begin{proof}
Take an indecomposable decomposition $M=\bigoplus_{x\in \Lambda}M(x)$ for which (I1) is satisfied. 
We show the claim by induction on $|\Lambda|$. 
If $|\Lambda|=1$, then the assertion holds. 
Assume $|\Lambda|\geq 2$. 
For $x\in\Lambda$, we show that $M(x)e_{\Lambda'}$ is indecomposable for all up-sets $\Lambda'\subseteq\Lambda$ containing $x$.  
If $\Lambda'=\Lambda$, then there is nothing to show. 
Assume $\Lambda'\subsetneq\Lambda$. 
Since $\Lambda'\subsetneq\Lambda$ is an up-set, there exists a minimal (with respect to $\unlhd$) element $y\in\Lambda$ such that $y\not\in\Lambda'$. 
Note that it follows from (I1) that $S(y)\cong M(y)$.

Consider the subposet $\Gamma:=\Lambda\setminus\{y\}$ of $\Lambda$.  Note that $\Lambda'$ is an up-set of $\Gamma$ by design.
Hence, to show our claim, we want to apply the induction hypothesis on $(Me_{\Gamma}, \unlhd|_{\Gamma})$.
To this end, we first need to show that that $M(z)e_{\Gamma}$ is indecomposable for all $z\in \Gamma$. 
Applying $\Hom_{A}(-,M(z))$ to $0\rightarrow \rej_{S(y)}(M(z))
\rightarrow M(z)\rightarrow M(z)/\rej_{S(y)}(M(z))\rightarrow 0$ yields an exact sequence
\begin{align}
 0\rightarrow\Hom_{A}(M(z)/\rej_{S(y)}(M(z)), M(z))\rightarrow \Hom_{A}(M(z), M(z))\rightarrow \Hom_{A}(\rej_{S(y)}(M(z)), M(z))\rightarrow0,   \notag
\end{align}
where the right exactness follows from $M(z)/\rej_{S(y)}(M(z))\in\add(S(y))=\add(M(y))\subseteq\add M$. 
By Lemma \ref{lem:Ifin}(1) and (R2), we have $\rej_{S(y)}(M(z))\in\calP^{\leq1}(\Gamma)$. 
Hence it follows from Lemma \ref{lem:apt-pfin} and $M(z)e_{\Gamma}\cong (\rej_{S(y)}(M(z)))e_{\Gamma}$ that 
\begin{align}
\Hom_{A}(\rej_{S(y)}(M(z)),M(z))\cong \Hom_{e_{\Gamma}Ae_{\Gamma}}((\rej_{S(y)}(M(z)))e_{\Gamma}, M(z)e_{\Gamma})\cong\Hom_{e_{\Gamma}Ae_{\Gamma}}(M(z)e_{\Gamma},M(z)e_{\Gamma}).    \notag
\end{align}
Thus, the previous short exact sequence induces a ring homomorphism $\End_{A}(M(z))\rightarrow\End_{e_{\Gamma}Ae_{\Gamma}}(M(z)e_{\Gamma})$ that is surjective, and so indecomposability of the $A$-module $M(z)$ implies that the $e_\Gamma Ae_\Gamma$-module $M(z)e_{\Gamma}$ is indecomposable.
Since truncation at $e_\Gamma$ preserves indecomposability of $M(z)$ for all $z\in\Gamma$, and $M(y)e_{\Gamma}=S(y)e_{\Gamma}=0$ by $\Lambda\setminus\Gamma=\{y\}$, $M':=Me_{\Gamma}$ admits an indecomposable decomposition $M'=\bigoplus_{z\in\Gamma}M'(z)$ with $M'(z)=M(z)e_\Gamma$. 

Now we show that $(M', \unlhd|_{\Gamma})$ satisfies both (I1) and (R2).
Since $[M(z)e_{\Gamma}:S(z')e_{\Gamma}]=[M(z):S(z')]$ for all $z,z'\in\Gamma$, it follows from Lemma \ref{lem:IST} that $(M', \unlhd|_{\Gamma})$ satisfies (I1) with respect to this decomposition. 
On the other hand, since it follows from (R2) for $(M,\unlhd)$ that $M' e_{\Gamma'}=(Me_{\Gamma})e_{\Gamma'}=Me_{\Gamma'}$ is $2$-rigid for each up-set $\Gamma'\subseteq\Gamma$, the pair $(M',\unlhd|_{\Gamma})$ also satisfies (R2). 

Now we can apply the induction hypothesis on $(M',\unlhd|_{\Gamma})$ and get that $(M', \unlhd|_{\Gamma})$ satisfies (I2). Thus $M'(x)=(M(x)e_{\Gamma})e_{\Gamma'}=M(x)e_{\Gamma'}$ is indecomposable for every up-set $\Gamma'\subseteq\Gamma$ containing $x$. 
In particular, as $\Lambda'$ is an up-set of $\Gamma$, $M(x)e_{\Lambda'}$ is indecomposable, as required.
\end{proof}

\subsection{Quasi-heredity of endomorphism algebras}

In this subsection, we show the endomorphism algebra of a certain module satisfying (IST) is quasi-hereditary. 

\begin{lemma}\label{lem:heredity-ideal}
Let $M$ be an $A$-module and let $S:=\bigoplus_{x\in \Gamma}S(x)$ be a semisimple direct summand of $M$ for some $\Gamma\subsetneq\Lambda$. 
Let $B:=\End_{A}(M)$ and let $\phi\in B$ be the idempotent given by composing the canonical projection $M\rightarrow S$ with the canonical inclusion $S\rightarrow M$.
Then the following statements hold.
\begin{itemize}
\item[(1)] $B\phi B$ is a heredity ideal of $B$, that is, $B\phi B$ is a projective $B$-module and   
$\phi B\phi$ is semisimple. 
\item[(2)] If $M$ is $2$-rigid, then there exists an algebra isomorphism $B/B\phi B\cong \End_{\e A\e}(M\e)$, where $\e:=e_{\Lambda\setminus\Gamma}$.
\end{itemize}
\end{lemma}

\begin{proof}
Denote by $[S](X,Y)$ the subspace of $\Hom_{A}(X,Y)$ consisting of all homomorphisms that factor through some $S'\in \add S$.

(1) By $M/\rej_{S}(M)\in \add S\subseteq \add M$, we have $[S](M/\rej_{S}(M),M)=\Hom_{A}(M/\rej_{S}(M),M)\in \proj(B^{\op})$.
Since $[S](\rej_{S}(M),M)=0$ holds, the natural surjection $M\rightarrow M/\rej_{S}(M)$ induces an isomorphism 
\begin{align}\label{eq:[S](M,M)}
[S](M/\rej_{S}(M),M)\cong [S](M,M)=B\phi B.
\end{align}
Thus $B\phi B$ is a projective $B^{\mathrm{op}}$-module.
Since $\phi B\phi \cong \End_{A}(S)$ is a semisimple algebra, $B\phi B$ is a heredity ideal of $B$ by \cite[Statement 7]{DR89}.

(2) Assume that $M$ is $2$-rigid. By $M/\rej_{S}(M)\in \add S \subseteq\add M$ and Lemma \ref{lem:rgd-smpl}(2), we have
\begin{align}
\Hom_{A}(\rej_{S}(M),M/\rej_{S}(M))=0\ \text{and}\ \Ext_{A}^{1}(M/\rej_{S}(M),M)=0. \notag
\end{align}
Hence, applying $\Hom_{A}(-,M)$ to the exact sequence 
\begin{align}\label{eq:M rejS(M)}
0\rightarrow \rej_{S}(M)\rightarrow M\rightarrow M/\rej_{S}(M)\rightarrow 0 
\end{align}
yields a short exact sequence 
\begin{align}\label{eq:BphiB-middle-row}
0\rightarrow \Hom(M/\rej_{S}(M),M) \rightarrow \Hom_A(M,M) \rightarrow \Hom_A(\rej_{S}(M), M) \rightarrow 0. 
\end{align}
Now applying $\Hom_{A}(\rej_{S}(M),-)$ to the same exact sequence \eqref{eq:M rejS(M)} we get an isomorphism $\End_{A}(\rej_{S}(M))\cong \Hom_A(\rej_{S}(M),M)$.
In particular, the surjection in \eqref{eq:BphiB-middle-row} induces a surjective algebra homomorphism $\pi: \End_{A}(M)\rightarrow \End_{A}(\rej_{S}(M))$.

On the other hand, since $M/\rej_{S}(M)\in\add S$, the left-hand $[S]$ in the isomorphism  \eqref{eq:[S](M,M)} can be rewritten as $\Hom_A$.
The above discussion then combines to the following commutative diagram
\begin{align}
\xymatrix{
0\ar[r]&\Ker\pi \ar[r]\ar[d]^-{\cong}&\End_{A}(M)\ar@{=}[d]\ar[r]^-{\pi}&\End_{A}(\rej_{S}(M))\ar[d]^-{\cong}\ar[r]&0\\
0\ar[r]&\Hom_{A}(M/\rej_{S}(M),M)\ar[r]\ar[d]^-{\cong}& \Hom_{A}(M,M)\ar[r]\ar@{=}[d]&\Hom_{A}(\rej_{S}(M),M)\ar[r]\ar[d]^-{\cong}&0\\
0\ar[r]&[S](M,M)\ar[r]&\End_{A}(M)\ar[r]&\frac{\End_{A}(M,M)}{[S](M,M)}\ar[r]&0
}\notag
\end{align}
with exact rows.
In particular, the surjective algebra homomorphism $\pi:\End_{A}(M)\rightarrow \End_{A}(\rej_{S}(M))$ has kernel $\Ker\pi \cong [S](M,M)$.
Thus we have an algebra isomorphism 
\begin{align}
B/B\phi B=\End_{A}(M)/[S](M,M)\cong \End_{A}(\rej_{S}(M)).\notag
\end{align} 
On the other hand, by Lemma \ref{lem:Ifin}(1), $\rej_{S}(M)\in\calP^{\leq 1}(\Lambda\setminus\Gamma)$.
Hence, it follows from Lemma \ref{lem:aus-apt}(1) that $\End_{A}(\rej_{S}(M))\cong \End_{\e A \e}(\rej_{S}(M)\e)$.
The proof is complete.
\end{proof}

\begin{remark}
We proved the lemma using $\rej_{S}(M)$, but one can also show it using $M/\tr_S(M)$ in an analogous way.
\end{remark}

Our next task is to show how quasi-hereditary structures can be constructed on the endomorphism algebra of a module satisfying (IST).
For this, we prepare an elementary fact about heredity ideal that allows us to build a quasi-hereditary algebra from a smaller one.

\begin{lemma}\label{lem:qh}
Suppose that $\unlhd$ is a partial order on $\Lambda$ and $y\in \max \unlhd $ with $Ae_{y}A$ a heredity ideal of $A$.
Let $\Gamma:=\Lambda\setminus\{y\}$ and $\unlhd':=\unlhd|_{\Gamma}$. Then the following statements hold. 
\begin{itemize}
\item[(1)] $[\std(y):S(y)]=1$ and $\std(x)=\std'(x)$ for all $x\in\Gamma$, where $\std'(x)$ is the standard $A/Ae_{y}A$-module with respect to $\unlhd'$.
\item[(2)] Assume, moreover, that $\Hom_{A}(e_{x}A,e_{y}A)=\Hom_{A}(e_{y}A,e_{x}A)=0$ for all $x$ that is incomparable to $y$.
If $(A/Ae_{y}A,\unlhd')$ is quasi-hereditary, then so is $(A,\unlhd)$. 
\end{itemize}
\end{lemma}

\begin{proof}
Let $x\in \Gamma$.
Take an exact sequence 
\begin{align}\label{eq:lemqh-exseq1}
0\rightarrow K'(x)\rightarrow e_{x}A/e_{x}Ae_{y}A\rightarrow  \std'(x)\rightarrow 0
\end{align}
in $\mod(A/Ae_{y}A)$. 
By the definition of standard modules, we have $\top K'(x)\in\add(\bigoplus_{z\ntrianglelefteq x}S(z))$. 
Similarly, we have an exact sequence 
\begin{align}\label{eq:lemqh-exseq2}
0\rightarrow K(x)\rightarrow e_{x}A\rightarrow  \std(x)\rightarrow 0
\end{align}
in $\mod A$ such that $\top K(x)\in\add(\bigoplus_{z\ntrianglelefteq x}S(z))$. By $\Hom_{A}(K(x),\std'(x))=0$, exact sequences \eqref{eq:lemqh-exseq1} and \eqref{eq:lemqh-exseq2} yield the following commutative diagram with exact rows and columns:
\begin{align}
\xymatrix{
&0\ar[d]&0\ar[d]&0\ar[d]\\
&K \ar[r]\ar[d]&e_{x}Ae_{y}A\ar[r] \ar[d] &K'\ar[d] &\\
0 \ar[r] &K(x) \ar[r]\ar[d]& e_{x}A\ar[r]\ar[d]&\std(x)\ar[r]\ar[d]&0&\\
0 \ar[r] &K'(x)\ar[d]^-{f}\ar[r] &e_{x}A/e_{x}Ae_{y}A \ar[r]\ar[d]&\std'(x) \ar[d]\ar[r]&0.&\\
&C&0&0
}\notag
\end{align}
Since $Ae_{y}A$ is a heredity ideal, it follows from Lemma \ref{lem:proj-stridemp}(1)$\Rightarrow$(2) that $Ae_{y}A\in\add(e_{y}A)$ holds.
In particular, $e_{x}Ae_{y}A\in \add(e_{y}A)$.
Since $y$ is maximal in $(\Lambda,\unlhd)$, we obtain $\Hom_{A}(e_{y}A,\std(x))=0$. 
This implies that $\Hom_{A}(e_{y}A, K')=0$, and hence $\Hom_{A}(e_{x}Ae_{y}A, K')=0$.
Thus we have $K\cong e_{x}Ae_{y}A$ and $K'\cong C$ by the snake lemma.

(1) These are elementary properties of a heredity ideal; we include the proof for completeness.   
First, we show that $[\std(y):S(y)]=1$. Since $Ae_{y}A$ is a heredity ideal and $e_{y}A$ is indecomposable, $\End_{A}(e_{y}A)$ is semisimple and local. 
Hence $\Hom_{A}(e_{y}A,\std(y))\neq 0$ and $\End_{A}(e_{y}A)=\Bbbk$ implies that 
\begin{align}
0\neq [\std(y):S(y)]=\dim_{\Bbbk}\Hom_{A}(e_{y}A,\std(y))\leq \dim_{\Bbbk}\End_{A}(e_{y}A)=1.\notag 
\end{align}
Thus $[\std(y):S(y)]=1$ holds. 

Next, we show that $\std(x)=\std'(x)$ for each $x\in\Gamma$. 
Suppose to the contrary that $K'\neq 0$.
Since $f$ is surjective and $\top K'(x)\in\add(\bigoplus_{z\ntrianglelefteq x}S(z))$, there is some $z\in \{w\in\Gamma\mid w\ntrianglelefteq x\}$ such that $[K':S(z)]=[C:S(z)]\neq 0$.
This contradicts the fact that $K'$ is a submodule of $\std(x)=P(x)/\mathrm{Tr}_{\bigoplus_{w\ntrianglelefteq x} P(w)}(P(x))$.
Thus we have $K'=0$, which means that $\std'(x)=\std(x)$. 

(2) Assume that $(A/Ae_{y}A,\unlhd')$ is quasi-hereditary.
We check the defining conditions (qh1) and (qh2) in Definition \ref{def:qha} for $(A,\unlhd)$. 
Since the condition (qh1) follows from (1), we show the condition (qh2), that is, $K(x)\in \filt(\std(x\lhd))$.
First, we show $\std(y)=e_{y}A$. 
Since $y$ is maximal in $(\Lambda,\unlhd)$, the trace defining $\std(x)$ involves only $P(x)$ for $x$ incomparable to $y$. 
It follows from the assumption that $[e_yA:S(x)]=\dim_{\Bbbk}\Hom_{A}(e_{x}A,e_{y}A)=0$ for all $x$ that is incomparable to $y$. Thus we have $\std(y)=e_{y}A$.

It remains to show $K(x)\in \filt(\std(x\lhd))$.
In the proof of (1), we have $C\cong K'=0$, and hence a short exact sequence
\begin{align}
0\rightarrow K\rightarrow K(x)\rightarrow K'(x)\rightarrow 0 \notag
\end{align}
with $K\cong e_{x}Ae_{y}A\in \add(e_{y}A)=\add(\std(y))$.
By the assumption and (1), we have
\begin{align}
K'(x)\in \filt(\std'(x\lhd')) = \filt(\std(x\lhd'))\subseteq\filt(\std(x\lhd)).\notag
\end{align}
Thus we obtain $K(x)\in \filt(\{\std(y), \std(z)\mid x\lhd z\})$.
If $x\lhd y$, then $K(x)\in \filt(\std(x\lhd))$.
Assume that $x \ntriangleleft y$, that is, $x$ and $y$ are incomparable with respect to $\unlhd$. 
By $e_{x}Ae_{y}=\Hom_{A}(e_{y}A,e_{x}A)=0$, we obtain $K\cong e_{x}Ae_{y}A=0$. 
Thus we have $K(x)\cong K'(x)\in \filt(\std(x\lhd))$.
The proof is complete.  
\end{proof}

\begin{proposition}\label{prop:end-quasi-hered}
Assume that $(M,\unlhd)$ satisfies (I1) and (R2).   
Then $(\End_{A}(M),\unlhd^{\op})$ is quasi-hereditary.
\end{proposition}

\begin{proof}
We show the claim by induction on $|\Lambda|$. If $|\Lambda|=1$, then the assertion clearly holds.
Assume $|\Lambda|\geq2$. 
By Lemma \ref{lem:suff-i1}, $(M,\unlhd)$ satisfies (IST). 
Let $y$ be a minimal element of $(\Lambda,\unlhd)$. 
For simplicity, let $\Gamma:=\Lambda\setminus\{y\}$, $\unlhd':=\unlhd|_{\Gamma}$ and $\varepsilon:=e_{\Gamma}$. 
Since $\Gamma\subset\Lambda$ is an up-set, it follows from Lemma \ref{lem:basic-results-ist}(2) that $(M\varepsilon,\unlhd')$ satisfies (IST). 
Let $\Gamma'\subseteq\Gamma$ be an up-set. 
Since $\Gamma'$ is also an up-set of $\Lambda$, the module $(M\varepsilon)e_{\Gamma'}=Me_{\Gamma'}$ is $2$-rigid by (R2) for $(M,\unlhd)$. 
Hence the induction hypothesis on $(M\varepsilon, \unlhd')$ implies that $(\End_{\varepsilon A\varepsilon}(M\varepsilon),\unlhd'^{\op})$ is a quasi-hereditary algebra.

Let $B:=\End_{A}(M)$. 
Using Lemma \ref{lem:qh}(2), we show that $(B,\unlhd^{\op})$ is quasi-hereditary. 
Take an indecomposable decomposition $M=\bigoplus_{x\in \Lambda}M(x)$ satisfying (I1). 
For each $x\in\Lambda$, let $\phi_{x}:M\rightarrow M(x)\rightarrow M$ be the composition of the canonical projection and inclusion. 
Since 
$M(y)=S(y)$ by minimality of $y$ in $(\Lambda,\unlhd)$ and $M$ is $2$-rigid by (R2), it follows from Lemma \ref{lem:heredity-ideal} that $B\phi_{y}B$ is a heredity ideal of $B$ and $B/B\phi_{y}B\cong \End_{\varepsilon A\varepsilon}(M\varepsilon)$. 
Thus $(B/B\phi_{y}B, \unlhd'^{\op})$ is quasi-hereditary.
Since $y$ is maximal in $(\Lambda,\unlhd^{\op})$, it is enough to show $\Hom_{B}(\phi_{x}B,\phi_{y}B)=\Hom_{B}(\phi_{y}B,\phi_{x}B)=0$ for $x$ that is incomparable to $y$. 
By Lemma \ref{lem:IST}, $\Hom_{A}(M(x),S(y))=\Hom_{A}(S(y),M(x))=0$. Since 
$M(y)=S(y)$, we have
\begin{align}
\Hom_{B}(\phi_{x}B,\phi_{y}B)\cong
\phi_{y}B\phi_{x}\cong \Hom_{A}(M(x),M(y))=\Hom_{A}(M(x),S(y))=0, \notag\\
\Hom_{B}(\phi_{y}B,\phi_{x}B)\cong
\phi_{x}B\phi_{y}\cong \Hom_{A}(M(y),M(x))=\Hom_{A}(S(y),M(x))=0.\notag
\end{align} 
Therefore $(B,\unlhd^{\op})$ is quasi-hereditary by Lemma \ref{lem:qh}(2). 
\end{proof}

We provide an example of Proposition \ref{prop:end-quasi-hered}. 

\begin{example}
Let $A=\Bbbk Q/I$ be the bound quiver algebra given by
\begin{align}
Q:\;\; \xymatrix{1\ar@<0.7mm>[r]^-{a}&2\ar@<0.7mm>[l]^-{b}\ar[r]^-{c}&3\ar[r]^-{d}&4}\;\; \text{ and } \;\;I=\langle ac, ba, cd \rangle.\notag
\end{align}
Let $(\Lambda,\unlhd)=\{1\lhd4\lhd3\lhd2\}$ and $M=\bigoplus_{x\in\Lambda}M(x)$, where 
\begin{align}
M(1)=\sm{1},\ M(2)=\sm{&2&\\1&&3},\ M(3)=\sm{3\\4},\
M(4)=\sm{4}.\notag 
\end{align}
We can check that $(M,\unlhd)$ satisfies (I1) and (R2), but $M$ is not $3$-rigid by $\Ext_{A}^{3}(M(1),M(4))\neq 0$. 
Then $(\End_{A}(M), \unlhd^{\op})$ is a quasi-hereditary algebra. 
\end{example}

\subsection{IS-tilting modules and their endomorphism algebras}

In this subsection, we introduce IS-tilting modules, which are a central notion of this paper. 

\begin{definition}\label{def:IS-tilt}
We call $T$ an \defn{IS-tilting module} if $T$ is a tilting $A$-module and there is some poset $(\Lambda,\unlhd)$ such that $(T,\unlhd)$ satisfies (IST).  In this case, we write $(T,\unlhd)$ to specify the associated partial order.
\end{definition}

\begin{lemma}\label{lem:truncation-tilting2}
Suppose that $(T,\unlhd)$ is a tilting $A$-module that satisfies (I1).  Then $Te_\Gamma$ is a tilting $e_\Gamma Ae_\Gamma$-module for every up-set $\Gamma\subseteq\Lambda$.
\end{lemma}

\begin{proof}
We have a filtration
\begin{align}
\Gamma = \Gamma_{k} \subsetneq \Gamma_{k-1} \subsetneq \cdots \subsetneq \Gamma_{1} \subsetneq \Gamma_{0} = \Lambda \notag
\end{align}
with $\Gamma_{i-1}\setminus \Gamma_{i} = \{x_{i}\}$, $x_{i} \in \min\unlhd|_{\Gamma_{i-1}}$ and $k=|\Lambda\setminus\Gamma|$.

If $\Gamma=\Gamma_0=\Lambda$, then the assertion clearly holds as $e_{\Gamma}=1_{A}$.
Suppose that the claim holds for every up-set $\Gamma=\Gamma_\ell$ with $\ell<k$. 
By the induction hypothesis, $T\e$ is a tilting $\e A \e$-module for $\e:=e_{\Gamma_{\ell}}$.
Since $\Gamma_{\ell}$ is an up-set in $\Lambda$ with minimal element $x_{\ell+1}$, we have $\Gamma_{\ell}\subseteq \Lambda_{x_{\ell+1}}$. 
By Lemma \ref{lem:IST}, $[T(x_{\ell+1}):S(y)]\neq 0$ implies $y\unlhd x_{\ell+1}$.
Hence, combining with $\Gamma_{\ell}\subseteq \Lambda_{x_{\ell+1}}$, we get $T(x_{\ell+1}) \e \cong T(x_{\ell+1}) \e_{x_{\ell+1}} \cong S(x_{\ell+1}) \e_{x_{\ell+1}} = S(x_{\ell+1}) \e$. Thus the tilting $\e A\e$-module $T\e$ has a simple direct summand $S(x_{\ell+1})\e$. 
As $e_{\Gamma_{\ell+1}}=\e-e_{x_{\ell+1}}$, it follows from Proposition \ref{prop:truncation-tilting} that $Te_{\Gamma_{\ell+1}} = T\e e_{\Gamma_{\ell+1}}$ is a tilting $e_{\Gamma_{\ell+1}} A e_{\Gamma_{\ell+1}}$-module.
\end{proof}

In general, the condition (I2) can be difficult to check on arbitrary modules.
For tilting modules, it turns out that this can be simplified.

\begin{proposition}\label{prop:IS-tilt}
Let $T$ be a tilting $A$-module and let $(\Lambda,\unlhd)$ be a poset. 
Then the following statements are equivalent. 
\begin{itemize}
\item[(1)] The pair $(T,\unlhd)$ satisfies (I1). 
\item[(2)] The pair $(T,\unlhd)$ satisfies (I1) and (R2).
\item[(3)] The pair $(T,\unlhd)$ satisfies (IST), i.e. $T$ is IS-tilting.
\end{itemize}
Consequently, in such a case, the following statements hold. 
\begin{itemize}
\item[(a)] $(Te_\Gamma, \unlhd|_{\Gamma})$ is IS-tilting for every up-set $\Gamma\subseteq \Lambda$.
\item[(b)] $(\End_{A}(T), \unlhd^{\op})$ is a quasi-hereditary algebra.
\end{itemize}
\end{proposition}

\begin{proof}
(2)$\Rightarrow$(3) follows from Lemma \ref{lem:suff-i1} and (3)$\Rightarrow$(1) is clear.  
Since it follows from Lemma \ref{lem:truncation-tilting2} that $Te_{\Gamma}$ is $2$-rigid for any up-set $\Gamma\subseteq\Lambda$, (1)$\Rightarrow$(2) holds.
The claim (a) follows from combining Lemma \ref{lem:basic-results-ist}(2) and Lemma \ref{lem:truncation-tilting2}. The claim (b) follows from Proposition \ref{prop:end-quasi-hered}.
\end{proof}

By Definition-Proposition \ref{prop:ringel}, Lemma \ref{lem:IST} and Proposition \ref{prop:IS-tilt}, characteristic tilting modules are always IS-tilting.  

Recall that a tilting $A$-module $T$ is also a tilting $\End_{A}(T)^{\op}$-module.
It turns out that the condition (IST) is also inherited when transferring to the endomorphism ring side.

\begin{proposition}\label{prop:istilt-op}
Let $(T,\unlhd)$ be an IS-tilting $A$-module and $B:=\End_{A}(T)$.
Then $({}_BT,\unlhd^{\op})$ is an IS-tilting $B^{\op}$-module.
\end{proposition}

\begin{proof}
As $B^{\op}$-modules, since $T$ and $\Hom_{A}(A,T)$ are isomorphic, we have an indecomposable decomposition $T\cong \bigoplus_{x\in \Lambda}\Hom_{A}(P(x),T)$.
By \cite[Theorem 1.5]{M86}, $T$ is a tilting $B^{\op}$-module.
By Proposition \ref{prop:IS-tilt}, it is enough to show that $({}_{B}T,\unlhd^{\op})$ satisfies (I1).
Let $S'(x)$ be a simple $B^{\op}$-module which is isomorphic to $\top\Hom_{A}(T(x),T)$.
Then we have
\begin{align}
[\Hom_{A}(P(x),T):S'(y)]
&=\dim_{\Bbbk}(\Hom_{B^{\op}}(\Hom_{A}(T(y),T),\Hom_{A}(P(x),T)))\notag\\
&=\dim_{\Bbbk}(\Hom_{A}(P(x),T(y)))=[T(y):S(x)].\notag
\end{align}
Since the pair $(T_A, \unlhd)$ satisfies (I1), the equality above and Lemma \ref{lem:IST} together imply that the $B^{\op}$-module $(T,\unlhd^{\op})$ also satisfies (I1), as required.
\end{proof}

\subsection{IS-tilting modules and quasi-hereditary structures}

Since IS-tilting modules come with a chosen partial order, we are interested in comparing this with the one associated to a characteristic tilting module.
The goal of the subsection is to make clear this relationship, and thus make precise the relation between IS-tilting modules and quasi-hereditary structures.

We show that an IS-tilting module yields a quasi-hereditary structure on the original algebra. 

\begin{theorem}\label{thm:qh-tilting}
Let $(T,\unlhd)$ be an IS-tilting $A$-module.
Then $(A,\unlhd)$ is a quasi-hereditary algebra.
\end{theorem}

\begin{proof}
Let $B:=\End_{A}(T)$.
By Proposition \ref{prop:istilt-op}, $({}_BT,\unlhd^{\op})$ is an IS-tilting $B^{\op}$-module.
It follows from Proposition \ref{prop:IS-tilt}(b) that $(\End_{B^{\op}}(T), \unlhd)$ is a quasi-hereditary algebra.
By \cite[Theorem 1.5]{M86}, we have an algebra isomorphism $A\cong \End_{B^{\op}}(T)$.
Thus $(A,\unlhd)$ is quasi-hereditary as claimed.
\end{proof}

In the rest of this subsection, we show that the set of IS-tilting modules is in one-to-one correspondence with the set of quasi-hereditary structures. 
So far, it is unclear that if one IS-tilting module can be associated to two inequivalent partial orders; in contrast, we already know how to characterise the class of partial orders that give rise to the same characteristic tilting modules.
In the following two results, we show that IS-tilting modules are indeed the same thing as characteristic tilting modules.

\begin{lemma}\label{lem:chtilt}
Let $(T,\unlhd)$ be an IS-tilting $A$-module.
Then we have $\Ext_{A}^{1}(\std(x),T)=0$ and $\Ext_{A}^{1}(T,\costd(x))=0$ for all $x\in \Lambda$, where $\std(x)$ is the standard module and $\costd(x)$ is the costandard module with respect to $\unlhd$.
\end{lemma}

\begin{proof} 
We show the claim involving costandard modules; the other cases can be done similarly.

Let $x\in\Lambda$. 
Take a sequence of covering relations $x_{1}\lhd^{\bullet}x_{2}\lhd^{\bullet}\cdots\lhd^{\bullet}x_{n}=x$ of $x$ such that $x_1$ is minimal. 
We proceed by induction on $n$. 
If $n=1$, then $x$ is a minimal element in $(\Lambda, \unlhd)$, which means that $\costd(x)=S(x)$.
Since $(T,\unlhd)$ is IS-tilting, we have $S(x)\in \add T$, and so $\Ext_{A}^{1}(T,\costd(x))=0$.

Assume $n\geq2$. 
Applying $\Hom_{A}(-,\costd(x))$ to the canonical short exact sequence $0\rightarrow \rej_{S(x_{1})}(T)\rightarrow T\rightarrow T/\rej_{S(x_{1})}(T)\rightarrow 0$ yields an exact sequence
\begin{align}\label{eq:exseqExt}
\Ext_{A}^{1}(T/\rej_{S(x_{1})}(T),\costd(x))\rightarrow \Ext_{A}^{1}(T,\costd(x))\rightarrow\Ext_{A}^{1}(\rej_{S(x_{1})}(T),\costd(x)).
\end{align}
First, we show that $\Ext_{A}^{1}(T/\rej_{S(x_{1})}(T),\costd(x))=0$.
Since $x_{1}\lhd x$, we have $\costd(x)\in \calI^{\leq 1}(\Lambda\setminus\{x_{1}\})$.
Hence, it follows from Lemma \ref{lem:apt-pfin} that $\Ext_{A}^{1}(S(x_{1}),\costd(x))=0$ holds.
As $T/\rej_{S(x_{1})}(T) \in \add(S(x_{1}))$, we obtain $\Ext_{A}^{1}(T/\rej_{S(x_{1})}(T),\costd(x))=0$. 
Next, we show that $\Ext_{A}^{1}(\rej_{S(x_{1})}(T),\costd(x))=0$.
Let $\varepsilon:=1-e_{x_{1}}$.
Since Lemma \ref{lem:Ifin}(1) says that $\rej_{S(x_{1})}(T)\in\calP^{\leq \infty}(\Lambda\setminus\{x_{1}\})$ and $T\varepsilon\cong (\rej_{S(x_{1})}(T))\varepsilon$, we can apply Lemma \ref{lem:apt-pfin} and obtain the following isomorphisms 
\begin{align}
\Ext_{A}^{1}(\rej_{S(x_{1})}(T),\costd(x))\cong \Ext_{\varepsilon A\varepsilon}^{1}((\rej_{S(x_{1})}(T))\varepsilon,\costd(x)\varepsilon)\cong \Ext_{\varepsilon A\varepsilon}^{1}(T\varepsilon,\costd(x)\varepsilon).\notag
\end{align}
By minimality of $x_{1}$, the module $\costd(x)\varepsilon$ is the costandard $\varepsilon A\varepsilon$-module with respect to $\unlhd':=\unlhd |_{\Lambda\setminus\{x_{1}\}}$.
Since $\Lambda\setminus\{x_{1}\}$ is an up-set of $\Lambda$, $(T\varepsilon,\unlhd')$ is an IS-tilting module by Proposition \ref{prop:IS-tilt}(a).
By the induction hypothesis, $\Ext_{\varepsilon A\varepsilon}^{1}(T\varepsilon,\costd(x)\varepsilon)=0$ holds. 
Since the left-hand side and right-hand side of \eqref{eq:exseqExt} are both zero, the same holds for the middle term, i.e. $\Ext_{A}^{1}(T,\costd(x))=0$, as required.
\end{proof}

\begin{proposition}\label{prop:isttilt-chtilt}
Let $(T,\unlhd)$ be an IS-tilting $A$-module.
Then $T$ is isomorphic to the characteristic tilting module $T_\unlhd$ associated to a quasi-hereditary algebra $(A, \unlhd)$.
In particular, the following conditions are equivalent for IS-tilting modules $(T,\unlhd), (T',\unlhd')$.
\begin{itemize}
\item[(1)] $T\cong T'$.
\item[(2)] $T_{\unlhd} \cong T_{\unlhd'}$.
\item[(3)] $[\unlhd]=[\unlhd']$.
\end{itemize}
\end{proposition}

\begin{proof}
By Theorem \ref{thm:qh-tilting}, $(A,\unlhd)$ is quasi-hereditary.
We show that $T$ is isomorphic to the associated characteristic tilting module $T_{\unlhd}$.
By \cite[Theorems 4 and 4$^{\ast}$]{R91}, we have 
\begin{align}
&\filt(\std)=\{ X\in \mod A \mid \Ext_{A}^{1}(X,\costd(x))=0\ \text{for all $x\in \Lambda$}\},\notag\\
&\filt(\costd)=\{X\in \mod A \mid \Ext_{A}^{1}(\std(x), X)=0\ \text{for all $x\in \Lambda$}\}, \notag
\end{align}
where $\std(x)$ is the standard module and $\costd(x)$ is the costandard module with respect to $\unlhd$. 
By Lemma \ref{lem:chtilt}, $T\in \filt(\std)\cap \filt(\costd) = \add T_{\unlhd}$ holds, which means that $\add T\subseteq \add T_{\unlhd}$.
Since $T$ and $T_{\unlhd}$ are basic tilting modules, we have $|T|=|\Lambda|=|T_{\unlhd}|$, and hence $T\cong T_{\unlhd}$.

Assume that $(T',\unlhd')$ is an IS-tilting module. 
Then it follows from the previous paragraph that $(A,\unlhd')$ is quasi-hereditary with characteristic tilting $T_{\unlhd'}\cong T'$.  Hence, $T\cong T'$ if and only if $T_{\unlhd}\cong T_{\unlhd'}$.
By Proposition \ref{prop:FKR-lems}, $T_\unlhd \cong T_{\unlhd'}$ if and only if $[\unlhd]=[\unlhd']$.
\end{proof}

Let $\IStilt A$ denote the set of IS-tilting $A$-modules up to isomorphism.
By Proposition \ref{prop:isttilt-chtilt}, an equivalence class $[\unlhd]$ is uniquely determined by an IS-tilting module $T$. Thus we write such an equivalence class $[\unlhd]_{T}$. 

Our first main result of this paper is the following. 

\begin{theorem}\label{thm:bij-qh-IStilting}
Let $A$ be an arbitrary finite-dimensional algebra over an algebraically closed filed $\Bbbk$.
Then there exist mutually inverse isomorphisms of posets
\begin{align}
\xymatrix{\IStilt A\ar@<1mm>[r]^-{\varphi} &\qhs A\ar@<1mm>[l]^-{\psi}},\notag
\end{align}
where $\psi([\unlhd]):=T_{\unlhd}$, and $\varphi(T):=[\unlhd]_{T}$.
\end{theorem}

\begin{proof}
By Proposition \ref{prop:isttilt-chtilt}, $\IStilt A=\chtilt A$ holds. Hence the assertion follows from Proposition \ref{prop:FKR-lems} and Theorem \ref{thm:qh-tilting}.
\end{proof}

Ringel provided a characterisation of quasi-hereditary algebras in terms of tilting modules over their quotient algebras \cite[Appendix, Theorem]{R91}, namely,
$(A,\unlhd)$ is a quasi-hereditary algebra if and only if there exists a tilting module $T=\bigoplus_{x\in\Lambda} T(x)$ such that $[T(x): S(x)]=1$ and $\bigoplus_{y\in\Lambda\setminus\Lambda_x} T(y)$ is a tilting $A/A\varepsilon_{x} A$-module. 
Theorem \ref{thm:bij-qh-IStilting} can be regarded as a `subalgebra counterpart' to Ringel’s result. 
As mentioned in the introduction, a notable advantage of working with the subalgebra version (IS-tilting) is that, in contrast to the quotient algebra case where one must verify the tilting property over each respective quotient algebra, IS-tilting is a condition that considers only $A$-modules essentially.

\section{Criteria for all tilting modules being IS-tilting}\label{sec:all tilt = IS-tilt <=>}

Throughout this section, $A$ is a ring-indecomposable basic finite-dimensional algebra over an algebraically closed field $\Bbbk$.
In other words, $A\cong \Bbbk Q /I$ for a finite connected quiver $Q$ and an admissible ideal $I$ of $\Bbbk Q$. 
For a quiver $Q$, let $Q_{0}$ be its vertex set and $Q_{1}$ its arrow set.

Let $\vec{\bbA}_{n}$ be the linearly oriented Dynkin quiver of type $\bbA_n$ for $n\geq 1$, i.e.
\[
\Vec{\bbA}_{n}: \quad  1\xrightarrow{\alpha_{1}} 2 \xrightarrow{\alpha_{2}} \cdots \xrightarrow{\alpha_{n-1}} n.
\]
By a \defn{linear Nakayama algebra} we mean an algebra whose Gabriel quiver is of the form $\vec{\bbA}_{n}$ for some $n\geq 1$.

An algebra of the form $\Bbbk Q/I$ is \defn{quadratic} if $I$ is generated by quadratic relations.  In the case when $Q=\vec{\bbA}_{n}$, this is equivalent to say that $I$ is generated by a set of length two paths.

\subsection{Main result}

The aim of this section is to prove the following theorem.

\begin{theorem}\label{thm:tilt=IStit=>gentle-LNaka}
Let $A$ be a ring-indecomposable basic finite-dimensional algebra over an algebraically closed field.
Then $\tilt A= \IStilt A$ holds if, and only if, $A$ is a quadratic linear Nakayama algebra.
\end{theorem}

\begin{proof}
We show the `if' part in the next Proposition \ref{prop:tilt=IStilt}, and the `only if' part in Proposition \ref{prop:tilt=IStilt=>qln}.
\end{proof}

For a quadratic linear Nakayama algebra $A\cong \Bbbk \Vec{\bbA}_{n}/I$ with $I\neq 0$, let
\begin{align*}
& \relvx(A):=\{ \ell \in [1,n]\mid a_{\ell-1}a_{\ell}\in I\}=\{\ell_{1}<\ell_{2}<\cdots<\ell_{k}\}, \\
& \ell_{0}:=1, \;\; \ell_{k+1}:=n, \\
&f_{i}:=e_{\ell_{i}}+e_{\ell_{i}+1}+\cdots+e_{\ell_{i+1}} \text{ for each } i\in[0,k].
\end{align*}
Note that $A/A(1-f_{i})A \cong \Bbbk \Vec{\bbA}_{\ell_{i+1}-\ell_{i}+1}$.

\begin{proposition}\label{prop:tilt=IStilt}
Let $A$ be a quadratic linear Nakayama algebra.
Then $\tilt A=\IStilt A$ holds.
\end{proposition}

\begin{proof}
Let $A$ be a quadratic linear Nakayama algebra $A\cong \Bbbk \Vec{\bbA}_{n}/I$.
We show that $\tilt A \subseteq \IStilt A$.
Assume that $I=0$, that is, $A$ is a path algebra.
By Proposition \ref{prop:FKR-thm4.7}, all tilting modules are characteristic tilting, which are IS-tilting by design.

Let us assume now that $I\neq 0$.
Recall that $A/A(1-f_{i})A \cong \Bbbk \Vec{\bbA}_{\ell_{i+1}-\ell_{i}+1}$ for each $i\in [0,k]$.
Hence, as $\tilt \Bbbk \vec{\bbA}_{r} = \IStilt \Bbbk \vec{\bbA}_r$, each tilting $(A/A(1-f_{i})A)$-module has at least one simple module as a direct summand.
If $M\in \mod(A/A(1-f_{i})A)$ satisfies $\Ext_{A/A(1-f_{i})A}^{1}(M,M)=0$ and has no simple direct summand, then it is a direct summand of a tilting $(A/A(1-f_{i})A)$-module by \cite{B81}.
On the other hand, as $M$ has no simple direct summand, $M$ is not tilting.
Thus we have $|M|\leq |A/A(1-f_{i})A|-1 =\ell_{i+1}-\ell_{i}$.

Suppose that a tilting $A$-module $U$ has no simple direct summand.
Since $A$ is a quadratic linear Nakayama algebra, an indecomposable non-simple $A$-module $X$ is in (the image of the embedding of) $\mod (A/A(1-f_{i})A)$ for a unique $i\in[0,k]$.
Hence, we have a unique decomposition $U=\bigoplus_{i=0}^{k}U_{i}$ so that $U_{i}\in \mod(A/A(1-f_{i})A)$.
This $U_i$ is 1-rigid as a $(A/A(1-f_{i})A)$-module; this is because the embedding $\mod(A/A(1-f_{i})A) \rightarrow \mod A$ preserves short exact sequences, which means that $\Ext_{A/A(1-f_{i})A}^{1}(U_{i}, U_{i})\cong \Ext_A^{1}(U_{i}, U_{i})=0$.
Since each $U_{i}$ has no simple direct summand, the argument in the previous paragraph implies that 
\begin{align}
n=|U|=\sum_{i=0}^{k}|U_{i}|\leq \sum_{i=0}^{k}(\ell_{i+1}-\ell_{i})=\ell_{k+1}-\ell_{0}=n-1, \notag
\end{align}  
a contradiction. Thus each tilting $A$-module admits a simple direct summand.

Let $T$ be a tilting $A$-module and $S(x_{1})$ a simple direct summand of $T$.
Let $\e_{2}:=1_{A}-e_{x_{1}}$. Then $\e_{2} A\e_{2}$ is a quadratic linear Nakayama algebra.
Moreover, by Proposition \ref{prop:truncation-tilting}, $T\e_{2}$ is a tilting $\e_{2}A\e_{2}$-module.
Thus $T\e_{2}$ admits a simple direct summand $S(x_{2})\e_{2}$.
Repeating this argument, we have a sequence $S(x_{1}), S(x_{2}),\ldots, S(x_{n})$ of non-isomorphic simple modules with $S(x_{i})\e_{i}\in \add(T\e_{i})$, where $\e_{1}=1_{A}$ and $\varepsilon_{i}=1_{A}-\sum_{j=1}^{i-1}e_{x_{j}}$.
Namely, $(T,\unlhd)$ satisfies (I1) of Definition \ref{definiton:IST}, where $(\Lambda,\unlhd)=\{ x_{1}\unlhd x_{2}\unlhd \cdots\unlhd x_{n}\}$.
By Proposition \ref{prop:IS-tilt}, $T$ is an IS-tilting module.
\end{proof}

For the `only if' part of Theorem \ref{thm:tilt=IStit=>gentle-LNaka}, the key arguments are given by the following three results.

\begin{lemma}\label{lem:IStilt-acy-Gab}
If $A\in \IStilt A$ holds, then the Gabriel quiver $Q$ of $A$ is acyclic, i.e. there is no oriented cycle.
\end{lemma}

\begin{proof}
Since $A$ is IS-tilting, there is a poset $(\Lambda,\unlhd)$ such that $(A, \unlhd)$ satisfies (IST).
Note that $(A,\unlhd)$ is quasi-hereditary by Theorem \ref{thm:qh-tilting}. 
We show the assertion by induction on $n=|\Lambda|$.
If $n=1$, then $A$ is necessarily semisimple, and hence the assertion clearly holds.
Assume that $n\geq 2$. 
Take a minimal element $x$ in $(\Lambda,\unlhd)$ and let $\e:=1-e_x$.
By Proposition \ref{prop:IS-tilt}(a), $(A\e, \unlhd|_{\Lambda\setminus\{x\}})$ is an IS-tilting $\e A\e$-module.
Note that as $x$ is minimal with respect to $\unlhd$, we have $S(x)=P(x)$, which means that $x$ is a sink in $Q$.
In particular, we have $A\e \cong \e A\e$ as $\e A\e$-modules.

Now we can apply the induction hypothesis which says that the Gabriel quiver $Q'$ of $\e A\e$ is acyclic.
Since $Q'$ is the full subquiver of $Q$ containing $Q_{0}\setminus\{x\}$ and $x$ is a sink in $Q$, the quiver $Q$ must be acyclic.
\end{proof}

\begin{proposition}\label{prop:non-qln-non-ist}
Assume that the Gabriel quiver $Q$ of $A$ is acyclic.
If $A$ is not isomorphic to a quadratic linear Nakayama algebra, then there exists a tilting $A$-module without simple direct summands.
\end{proposition}

Before detailing the argument for Proposition \ref{prop:non-qln-non-ist}, let us first explain how the `only if' part of Theorem \ref{thm:tilt=IStit=>gentle-LNaka} is a simple consequence of it.

\begin{proposition}\label{prop:tilt=IStilt=>qln}
If $\tilt A=\IStilt A$ holds, then $A$ is isomorphic to a quadratic linear Nakayama algebra.
In particular, the `only if' part of Theorem \ref{thm:tilt=IStit=>gentle-LNaka} holds.
\end{proposition}

\begin{proof}
By $A_A\in\tilt A =\IStilt A$ and Lemma \ref{lem:IStilt-acy-Gab}, the Gabriel quiver $Q$ of $A$ is acyclic.
Since $\tilt A = \IStilt A$ implies that each tilting $A$-module has a simple direct summand, the claim now follows immediately by applying (the contrapositive of) Proposition \ref{prop:non-qln-non-ist} to $A$.
\end{proof}

\subsection{Tilting mutation for coextentions of linear Nakayama algebras}\label{subsec:coext-by-LNA}

Before starting the proof of Proposition \ref{prop:non-qln-non-ist}, we give one observation for tilting mutation for coextension of algebras. The results of this subsection will be used in subsection \ref{subsec:mutate-corner}.

Consider an arbitrary finite-dimensional algebra $A$ that can be written in an upper-triangular matrix form: 
\begin{align}
A = \begin{bmatrix} B & {}_{B}M_{C} \\ 0 & C \end{bmatrix}\notag
\end{align}
with the condition that $C$ is a linear Nakayama algebra and the right $C$-module $M_{C}$ is semisimple with composition factors all isomorphic to a simple $C$-module $S^{C}(1)$.  Here, $1\in (Q_{C})_{0}$ is the source vertex.

We use superscript to distinguish $B$-module and $C$-module whenever necessary; for example, $P^{B}(x)$ and $P^{C}(y)$ are projective $B$-module and projective $C$-module respectively.
An $A$-module can be specified as a triple $(X^{B},X^{C},\alpha_{X})$ where $\alpha_{X}:X^{B}\otimes_{B} M\rightarrow X^{C}$ is a $C$-module homomorphism.
We write this as $[X^{B},X^{C}]_{\alpha_X}$, and if $\alpha_{X}=0$ then we will omit the subscript as well.
In this form, an $A$-module homomorphism $f:[X^{B},X^{C}]_{\alpha_{X}}\rightarrow [Y^{B},Y^{C}]_{\alpha_{Y}}$ is given by a pair $(f^{B},f^{C})\in \Hom_{B}(X^{B},Y^{B})\times\Hom_{C}(X^{C},Y^{C})$ satisfying $\alpha_{Y}\circ(f^{B}\otimes 1)=f^{C}\circ\alpha_{X}$.
For details, see \cite{ARS95}.

Recall also the description of some structural $A$-modules:
\begin{itemize}
\item A simple $A$-module $S(x)$ is given by $[S^B(x),0]$ or $[0,S^{C}(x)]$.
\item An indecomposable projective $A$-module $P(x)$ takes the form $[P^{B}(x), e_{x}M]_{\mathrm{id}}$ or $[0,P^{C}(x)]$.
\item An indecomposable injective $A$-module $I(x)$ takes the form $[I^{B}(x),0]$ or $[\mathbb{D}(Me_{x}),I^C(x)]_{\mathrm{ev}}$, where $\mathrm{ev}$ is induced by $\mathrm{id}_{\Hom_C(M,I^{C}(x))}$.  In particular, we have  $I(x)=[0,I^{C}(x)]$ for $x\in (Q_{C})_{0}\setminus\{1\}$.
\end{itemize}

Note that the two-sided ideal $AeA$ generated by the idempotent $e=\left[ \sm{1_{B} & 0\\0& 0}\right]$ is just $[B,M]_{\mathrm{id}} = eA\in \add(eA) \subseteq \proj A$. By Lemmas \ref{lem:str-idemp}(2) and \ref{lem:proj-stridemp}, we have $\Ext_{A}^{k}([0,X],[0,Y])\cong \Ext_{C}^{k}(X,Y)$ for all $X,Y\in\mod C$ for all $k\ge 0$.

\begin{lemma}\label{lem:indec-upper-tri}
For an indecomposable $A$-module $X=[X^{B},X^{C}]_{\alpha}$ with both $X^{B}, X^{C}$ non-zero, we always have $X^{C}\in \add(S^{C}(1))$.
\end{lemma}

\begin{proof}
The condition on $M$ means that $X^{B}\otimes_{B} M \in \add(S^{C}(1))$ since $(X^{B}\otimes_{B} M)e_{u}=0$ for each $u\in(Q_{C})_{0}\setminus\{1\}$.   
Since $\alpha\neq 0$ by the indecomposability of $X$, we have a non-zero submodule $\alpha(X^{B}\otimes_{B} M)$ of $X^{C}$ that lies in $\add(S^{C}(1))$. Since $1$ being a source of $Q_{C}$ means that $\Ext_{C}^{1}(S^{C}(u),S^{C}(1))=0$ for all $u\in (Q_{C})_{0}$, we have $X^{C}=\mathrm{Im}(\alpha)\oplus Y$ for $Y=X^{C}/\mathrm{Im}(\alpha)$. 
If $Y\neq 0$, then we have $X=[X^{B}, \mathrm{Im}(\alpha)]_{\alpha}\oplus [0,Y]$. Since $X$ is indecomposable, we obtain $Y=0$. Thus, $X^{C}=\mathrm{Im}(\alpha)\in\add(S^{C}(1))$ as required.
\end{proof}

\begin{proposition}\label{lem:mutate-corner}
Consider $T\in \tilt A$ with $[0,U]\in \add T$ and $U\in \tilt C$.
Then, for any $X\in \add U$ such that $\mu_{X}(U)\in\tilt C$, we have $\mu_{[0,X]}(T) = T/[0,U] \oplus [0,\mu_{X}(U)]\in \tilt A$. 
\end{proposition}

\begin{proof}
Let $f=[0,f^{C}]:[0,X]\rightarrow [Y^{B}, Y^{C}]_{\alpha}$ be a minimal left $\add(T/[0,X])$-approximation. 
We want to show that $Y^{B}=0$.
Suppose to the contrary that $Y^{B}\neq 0$.
Then there exists an indecomposable direct summand $[Z^{B},Z^{C}]_{\beta}$ with $Z^{B}\neq 0$.
Note that $Z^{C}\neq 0$ as $f$ is minimal.
By Lemma \ref{lem:indec-upper-tri}, we must have $Z^{C}\in\add(S^{C}(1))$.  
Hence, $f^{C}\neq 0$ means that we have $[\top X: S^{C}(1)]\neq 0$.
Since $C$ is a Nakayama algebra, there is an indecomposable direct summand $X'$ of $X$ such that $\top X' = S^{C}(1)$.
Since $C$ is linear Nakayama, such an $X'$ is injective, which means that $X$ cannot be mutable. This contradicts the mutability assumption.
Thus, $Z^{B}$ is necessarily zero and so is $Y^{B}$.

Now we have a minimal left $\add(T/[0,X])$-approximation $f=[0,f^{C}]:[0,X]\rightarrow [0,Y^{C}]$ of $[0,X]$.
Note that if $[0,U']$ is a direct summand of $T$, then we have $\Ext_{C}^{k}(U'\oplus U,U'\oplus U) \cong \Ext_A^{k}( [0,U'\oplus U],[0,U'\oplus U])\cong 0$ for all $k>0$.
Since $U$ is a tilting $C$-module, $U'\in\add U$.
Hence, $Y^C\in \add U$ and $[0,f^{C}]$ is a minimal left $\add([0,U]/[0,X])$-approximation of $[0,X]$.
Thus $f^{C}:X\rightarrow Y^{C}$ is a minimal left $\add(U/X)$-approximation.
This implies that $f^{C}$ is injective as $X$ is a mutable direct summand of $U$.  The claim now follows.
\end{proof}

\begin{remark}
Note that it is possible to have $\mu_{[0,X]}(T)\in \tilt A$ while $X$ is not a mutable direct summand of $U$.  It is easy to construct such an example even in the case when $A$ is a quadratic linear Nakayama with $C=\Bbbk \vec{\bbA}_{n}$, one concrete example is when $X=S^C(1)$, $U=\kD C$ and $T=A/[0,C]\oplus [0,\kD C]$.
\end{remark}

\begin{example}\label{eg:mutate-corner-gentle-LNaka}
Let $C$ be a quadratic linear Nakayama algebra.
Since $C$ is representation-finite, every tilting $C$-module can be obtained by iterative mutation starting from $C$ \cite[Corollary 2.2]{HU05}.
If $T=V\oplus [0,C] \in \tilt A$, applying Proposition \ref{lem:mutate-corner} repeatedly, we have $V\oplus[0,U] \in \tilt A$ for all $U\in \tilt C$.
\end{example}

\subsection{Strategy of the proof of Proposition \ref{prop:non-qln-non-ist}}

In the rest of this section, we show Proposition \ref{prop:non-qln-non-ist}.
Assume that the Gabriel quiver $Q$ of $A$ is acyclic and $A$ is non-semisimple and not isomorphic to a quadratic linear Nakayama algebra.
Note that the ring-indecomposable assumption means that $Q$ is connected.

Let $\calS$ be the set of all sinks of $Q$ and decompose $\calS=\calS_{1}\sqcup\calS_{2}$, where $\calS_{1}$ is the subset of $\calS$ consisting of vertices $v$ with valency (i.e. the number of arrows attached to $v$) exactly one and $\calS_{2}$ is the subset of $\calS$ consisting of vertices with valency at least two.
Since $Q$ is acyclic, the set $\calS$ is not empty. 
Note that $P(v)$ is simple if and only if $v\in \calS$.

We will construct a tilting module via a sequence of mutations from $A$:
\begin{align}\label{eq:tilting-mutation}
A \overset{\text{ mutate }}{\leadsto} T_{0} \overset{\text{ mutate }}{\leadsto} T_{1} \overset{\text{ mutate }}{\leadsto} T_{2} \overset{\text{ mutate }}{\leadsto} \cdots \overset{\text{ mutate }}{\leadsto} T_{|\calS_{1}|},
\end{align}
where $T_{0}$ will eliminate all $S(v)$ for $v\in \calS_{2}$, and each $T_{a}$ for $a>0$ will eliminate a single simple direct summand $S(v)$ for $v\in \calS_{1}$ from $T_{a-1}$ while not creating any other new simple direct summand.
In particular, $T_{|\calS_{1}|}$ has no simple direct summands.

We first construct a tilting $A$-module $T_0$.

\begin{lemma}\label{lem:mutate-away-S2}
Let $X:=\bigoplus_{v\in \calS_{2}} P(v)$.
Then $X$ is a mutable direct summand of $A$ and $T_{0}:=\mu_{X}(A)$ is a tilting $A$-module such that $S(v)\in\add T_{0}$ if and only if $v\in\calS_{1}$.
\end{lemma}

\begin{proof}
For each $v\in\calS_{2}$, the map $\psi_{v} =  (\alpha\cdot -)_{\alpha} : P(v)\rightarrow \bigoplus_{(u\xrightarrow{\alpha} v)\in Q_{1}}P(u)$
is an injective minimal left $\add(A/P(v))$-approximation of $P(v)=S(v)$.
By $\bigoplus_{(u\xrightarrow{\alpha} v)\in Q_{1}}P(u)\in \add(A/X)$, the map $\psi_{v}$ is also a minimal left $\add(A/X)$-approximation of $P(v)$.
Then $\bigoplus_{v\in\calS_{2}}\psi_{v}$ is an injective minimal left $\add(A/X)$-approximation of $X$.

For each $v\in\calS_{2}$, the cokernel $\Cok\psi_{v}$ is not simple since there exist at least two arrows ending at $v$ and $\Cok \psi_{v}$ is indecomposable by Lemma \ref{prop:tilting-mutation}(2).
Hence, $\Cok \left( \bigoplus_{v\in\calS_{2}}\psi_{v}\right) = \bigoplus_{v\in\calS_{2}}\Cok\psi_{v}$ has no simple direct summands.
This implies that $S(v)\in\add T_{0}$ if and only if $v\in\calS_{1}$.
\end{proof}

In the following example, a dashed line connecting the starting and ending arrows of the path $p$ is used to indicate that $p$ is a defining relation.

\begin{example}\label{eg:mut-T-1}
We consider the following quiver with relations
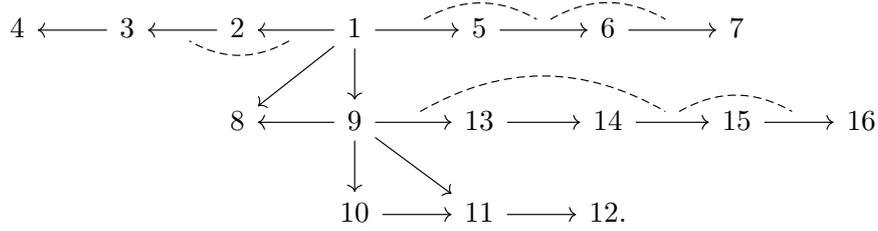
\[
\begin{tikzcd}
 4 & 3 \ar[l] & 2 \ar[l,""{name=a23}] & 1 \ar[l,""{name=a12}] \ar[r,""{name=a15}]  \ar[ld] \ar[d]   & 5 \ar[r,""{name=a56}]  & 6 \ar[r,""{name=a67}] & 7 \\
&&8  & 9 \ar[l] \ar[r,""{name=a910}] \ar[rd] \ar[d]  & 13 \ar[r]  & 14 \ar[r,""{name=a1112}] & 15 \ar[r, ""{name=a1516}] & 16 \\
 &&& 10 \ar[r] & 11 \ar[r] & 12 .
\ar[dashed, no head, bend left, from=a12, to=a23]
\ar[dashed, no head, bend left, from=a15, to=a56]
\ar[dashed, no head, bend left, from=a56, to=a67]
\ar[dashed, no head, bend left, from=a910, to=a1112]
\ar[dashed, no head, bend left, from=a1112, to=a1516]
\end{tikzcd}
\]
Then we have $\calS_{1}=\{4, 7, 12, 16\}$, $\calS_{2} =\{8\}$, and $T_{0} = (A/P(8)) \oplus C(8)$ with $C(8)=\Cok(P(8) \rightarrow P(1)\oplus P(9))$, where $P(8) \rightarrow P(1)$ is a morphism obtained by multiplying the arrow from $1$ to $8$.
\end{example}

Observe that for each $v\in \calS_{1}$, there is a full subquiver $\Gamma^{v}$ satisfying the following properties.

\begin{itemize}
\item[(i)] $\Gamma^{v}\cong \Vec{\bbA}_{n}$ for some $n\geq 2$.
Then we can label the vertices and arrows of $\Gamma^{v}$ as follows:
\begin{align}
1^{v}\xrightarrow{\alpha_{1}^{v}} 2^{v} \xrightarrow{\alpha_{2}^{v}} \cdots \rightarrow (n-1)^{v} \xrightarrow{\alpha_{n-1}^{v}} n^{v}=v.\notag
\end{align}
\item[(ii)] The valency of $j^{v}$ in $Q$ is two for all $1<j<n$ and that of $n^{v}=v$ in $Q$ is one.
\item[(iii)] $eIe$ is zero or generated by paths of length two for the idempotent $e:=e_{1^{v}}+e_{2^{v}}+\cdots +e_{n^{v}}\in \Bbbk Q$.
\item[(iv)] $\Gamma^{v}$ is maximal with respect to the properties (i), (ii), and (iii), that is, if $\Gamma'$ is a full subquiver of $Q$ satisfying (i), (ii) and (iii), then $\Gamma'\subseteq\Gamma^{v}$. 
\end{itemize}

To keep track of which simple modules are eliminated from the tilting module at hand, we consider the following subset of $\Lambda$.
\begin{align}
\Lambda' &:= \left\{1^w\mid \exists \alpha\neq\beta \in Q_{1} : s(\alpha)=1^w=s(\beta), w\in \calS_{1}\right\} \cup \left(\Lambda\setminus \Bigl(\calS_{2} \cup \bigcup_{u\in\calS_1} \Gamma^{u}_{0}\Bigr) \right).\notag
\end{align}
Note that $\Gamma^{v}_{0}\cap \Gamma^{w}_{0}\subseteq \Lambda'$ for $v\neq w$.

\begin{example}\label{eg:mut-T-Gamma-sets}
We continue with Example \ref{eg:mut-T-1}.
For each $v\in\calS_1=\{4,7,12,16\}$, the quiver $\Gamma^{v}$ is the following:
\begin{align}
\begin{array}{lll}
\Gamma^4 : \begin{tikzcd}[ampersand replacement=\&]
1 \ar[r,""{name=a12}]  \& 2 \ar[r,""{name=a23}]  \& 3 \ar[r,""{name=a34}] \& 4
\ar[dashed, no head, bend left, from=a12, to=a23]
\end{tikzcd},
&&
\Gamma^7 : \begin{tikzcd}[ampersand replacement=\&]
 1 \ar[r,""{name=a15}]  \& 5 \ar[r,""{name=a56}]  \& 6 \ar[r,""{name=a67}] \& 7
\ar[dashed, no head, bend left, from=a15, to=a56]
\ar[dashed, no head, bend left, from=a56, to=a67]
\end{tikzcd},
\\
\Gamma^{12} : \begin{tikzcd}[ampersand replacement=\&]
 11 \ar[r] \& 12 
\end{tikzcd},
&&
\Gamma^{16} : \begin{tikzcd}[ampersand replacement=\&]
13 \ar[r] \& 14 \ar[r,""{name=a14}] \& 15 \ar[r,""{name=a15}] \& 16
\ar[dashed, no head, bend left, from=a14, to=a15]
\end{tikzcd}.
\end{array}\notag
\end{align}
We have $\Lambda' = \{1, 9, 10\}$.
\end{example}

As mentioned after \eqref{eq:tilting-mutation}, for each $a\ge 1$, we will construct from a tilting $A$-module $T_{a-1}$ a new tilting $A$-module $T_a$ with fewer simple direct summands.
To construct $T_a$, we consider the following setup.
For an $A$-module $M$, let
\begin{align*}
\calS(M) := \{ u\in \Lambda \mid  S(u) \in \add M\}.
\end{align*}

\begin{definition}
We say that an $A$-module $M$ \defn{satisfies (S)} if all of the following hold.
\begin{itemize}
\item[(S1)] $\calS(M) \subseteq \calS_{1}$.
\item[(S2)] $P(u)\in \add M$ for all $u\in\Lambda'$.
\item[(S3)] If $w\in\calS(M)\cap \calS_{1}$, then $P(u)\in \add M$ for all $u\in \Gamma^{w}_{0}$.
\end{itemize}
\end{definition}

\begin{lemma}\label{T0-condition}
The module $T_{0}$ of Lemma \ref{lem:mutate-away-S2} satisfies (S) with $\calS(T_{0})=\calS_{1}$.
\end{lemma}

\begin{proof}
Recall that $T_{0}=\mu_{X}(A)$ with $X=\bigoplus_{v\in \calS_{2}}P(v)$.
By Lemma \ref{lem:mutate-away-S2}, $\calS(T_{0})=\calS_{1}$ holds, which clearly implies (S1).
For $u\in\Lambda$, $P(u)\in\add T_{0}$ if and only if $u\not\in\calS_{2}$.
It is clear that $1^w\not\in\calS_{2}$ for any $w\in\calS_{1}$.
Therefore $T_0$ satisfies (S2).
Moreover, by $\calS_{2} \cap \bigcup_{w\in\calS_{1}}\Gamma_{0}^{w} = \emptyset$, $T_{0}$ satisfies (S3).
\end{proof}

Assume that we are given a tilting module $T_{a-1}$ satisfying $({\rm S})$, and fix $v\in\calS(T_{a-1})$.
Our task is to construct a sequence of mutations
\begin{align}
T:=T_{a-1} \overset{\text{ mutate }}{\leadsto} T' \overset{\text{ mutate }}{\leadsto} T'' 
\end{align}
so that $T''$ satisfies ({\rm S}) and $\calS(T'')=\calS(T_{a-1})\setminus\{v\}$ holds (see Lemmas \ref{lem-T''-conditon(i)(ii)} and \ref{lem:T''-conditon(iii)}).
In particular, we can take $T_a:=T''$ and repeat this simple-eliminating procedure until reaching $a=|\calS_{1}|$; in which case, ({\rm S1}) and $\calS(T_{0})=\calS_{1}$ say that $\calS(T_{|\calS_{1}|})=\emptyset$, and so the final tilting module has no simple direct summand as desired.

\subsection{Construction of \texorpdfstring{$T'$}{T'}}\label{subsec:mutate-corner}

Let $T$ be a tilting $A$-module satisfying $({\rm S})$, and fix $v\in\calS(T)$.
We may drop the superscript $v$ in the notation of the quiver $\Gamma^v$, as well as its vertices and arrows, whenever there is no confusion.

Consider $e:=\sum_{x \in \Gamma_{0}}e_{x}$ and
\begin{align}
m:=\begin{cases}
n, & \text{if $eAe=\Bbbk \vec{\bbA}_{n}$}; \\
\min \relvx(eAe), & \text{otherwise}.
\end{cases}\notag
\end{align}
Note that $m>1$ always.
Intuitively, $m$ marks the `first' relation that appears in $\Gamma$.
By (i) and (ii) of $\Gamma$, we can write 
\begin{align}\label{eq:triangular}
 A =  \begin{bmatrix}
\e' A \e' & \e' A\e \\ 0 & \e A\e
\end{bmatrix}=:\begin{bmatrix}
B & M \\ 0 & C\end{bmatrix},
\end{align}
where $\e=\sum_{i=m}^{n}e_{i}$ and $\e'=1_{A} - \e$.
By (S3) for $T$, we have $P(j)=[0,P^{C}(j)]\in \add T$ for all $j\in [m,n]$, which implies that $[0,C]$ is a direct summand of $T$.

\begin{lemma}\label{lem:T-to-T'}
Define 
\begin{align}
T':= [0,\kD C]\oplus T/[0,C].\notag
\end{align}
Then $T'$ is a tilting $A$-module obtained from $T$ via a sequence of mutation, and it satisfies the following properties.
\begin{itemize}
\item[(a)] $\calS(T') = (\calS(T)\setminus\{v\}) \cup \{m\}$.
\item[(b)] $P(u)\in \add T'$ for all $u\in \Lambda'$.
\item[(c)] If $w\in\calS(T')\cap \calS_{1}$, then $P(u)\in \add T'$ for all $u\in \Gamma^{w}_0$.
\item[(d)] $\bigoplus_{i=1}^{m-1} P(i)\oplus [0, \kD C] \in \add T'$.  Moreover, for an indecomposable $U\in\add T'$, if $x\in \Gamma_0^v$ holds whenever $[U:S(x)]\neq 0$, then we have $U\in\add(\bigoplus_{i=1}^{m-1} P(i)\oplus [0, \kD C])$.
\end{itemize}
\end{lemma}

\begin{proof}
Note that $M$ in (\ref{eq:triangular}) satisfies the condition of subsection \ref{subsec:coext-by-LNA} since $Me_j=0$ for $j\in[m+1, n]$ by the definition of $m$. Hence, $T'$ can be obtained by repeatedly applying Proposition \ref{lem:mutate-corner}, as demonstrated in Example \ref{eg:mutate-corner-gentle-LNaka}.

Since $[0,C]$ (respectively, $[0,\kD C]$) has only one simple module $S(n)=S(v)$ (respectively, $S(m)$) as a direct summand, we have (a).

We show (b) and (c). If $u$ is in $\Lambda'$ or $\Gamma^{w}_{0}$ ($w\neq v$), then $P(u)$ is a direct summand of neither $[0,C]$ nor $[0,\kD C]$. Thus (b) and (c) follow from (S2) and (S3) for $T$ respectively.

Now $m>1$. Hence, $1^{v}\notin (Q_{C})_{0}$ by definition, and so $P(1^{v})\in\add T$ by (S3) for $T$. This implies that $P(1^{v})\in \add T'$, as required.

We show (d). By (S3) for $T$, we have $P(1),\ldots, P(m-1)\in \add(T/[0,C])$. Thus the first statement is immediate from the construction of $T'$.
For the second part, recall that, as $T'$ is a tilting $A$-module, the dimension vectors 
\begin{align}
\{ \mathbf{dim}\, U\mid U\text{ indecomposable direct summand of }T'\}\notag
\end{align}
form a $\Z$-basis of the Grothendieck group $K_{0}(\mod A)$ by Proposition \ref{prop:property-tilting}(2) and $\gldim A<\infty$. 
The dimension vectors of the $n$ indecomposable direct summands of $\bigoplus_{i=1}^{m-1} P(i)\oplus [0, \kD C]$ span the rank-$n$ sublattice $\langle [S(1)], \ldots, [S(n)]\rangle_{\Z}$ of $K_{0}(\mod A)$.
Thus, no other indecomposable direct summand of $T'$ can have composition factor concentrated in $\Gamma_{0}^{v}$.
\end{proof}

\begin{example}\label{eg:mut-T-to-T'}
We continue with Examples \ref{eg:mut-T-1} and \ref{eg:mut-T-Gamma-sets}, where $T:=T_{0}=(A/P(8)) \oplus C(8)$.
For $v=4\in\calS_{1}$, we have $m=2$.
The algebra $C$ inside the upper triangular matrix algebra \eqref{eq:triangular} is the path algebra of $(2 \rightarrow 3 \rightarrow 4)$.
Note that $[0, C]=P(2)\oplus P(3) \oplus S(4)$ and $[0, \kD C] = P(2) \oplus (P(2)/S(4)) \oplus S(2)$. Thus we have
\begin{align*}
T' = \frac{T_0}{[0, C]} \oplus [0, \kD C] = \frac{A}{P(2)\oplus P(3)\oplus P(4)\oplus P(8)} \oplus C(8) \oplus P(2) \oplus \frac{P(2)}{P(4)} \oplus S(2) .
\end{align*}
For other $v\in \calS_{1}$, we have $m$ and $C$ as follows.
\begin{align}
\begin{array}{cccc}
 v & (n, n^{v}) & (m, m^{v}) &  C\\ \hline
 7  & (4,7) & (2,5) &  \Bbbk (5\xrightarrow{a}6\xrightarrow{b}7)/(ab)\\
 12 & (2,12) & (2,12) & \Bbbk \cong \Bbbk(12)\\
 16 & (4,16) & (3,15) &  \Bbbk (15\rightarrow 16)\\
\end{array}\notag
\end{align}
\end{example}

\subsection{Construction of $T_{a} = T''$}

Let $T$ be a tilting $A$-module satisfying (S), and fix $v\in\calS(T)$ as in subsection \ref{subsec:mutate-corner}.
Let $T'$ be the tilting $A$-module constructed from $T$ in Lemma \ref{lem:T-to-T'}.

Let $\mathrm{dp}_{1}$ be the set of \defn{direct predecessors} of the vertex $1$ in $Q$, i.e. all $w\in Q_{0}$ such that $(w\rightarrow 1)\in Q_{1}$; likewise, let $\mathrm{ds}_{1}$ be the set of \defn{direct successors} of the vertex 1 in $Q$.

\begin{lemma}\label{lem:easy-obs}
If $|\mathrm{ds}_{1}|=1$, then $\mathrm{dp}_{1}\subseteq \Lambda'$.
\end{lemma}

\begin{proof}
Let $w\in \mathrm{dp}_{1}$. Clearly $w\notin \calS_{2}$ holds. If $w\notin \Gamma^{u}_{0}$ for all $u\in \calS_{1}$, then there is nothing to prove.
Assume that $w\in \Gamma_{0}^{u}$ holds for some $u\in \calS_{1}\setminus\{ v\}$. 
By (ii) for $\Gamma^{u}$, $w$ cannot be $j^{u}$ for all $j>1$. Thus $w=1^{u}$.
This implies that there is a subquiver $1^{v} \leftarrow 1^{u} \rightarrow 2^{u}$ in $Q$. Hence $w\in \Lambda'$.
\end{proof}

Define $r\in [1,m]$ as follows. 
\begin{align*}
r:=& 
\begin{cases}
m, & \text{if $|\mathrm{ds}_1|>1$};\\
\max\{ j\in [1,m] \mid \Hom_A(S(j), \bigoplus_{w\in \Lambda\setminus \Gamma_0}P(w))\neq 0\}, & \text{otherwise};
\end{cases}\\
= & \begin{cases}
m, & \text{if $|\mathrm{ds}_1|>1$};\\
\max\{j\in [1,m] \mid S(j)\in \add(\soc \bigoplus_{w\in \mathrm{dp}_{1}} P(w)) \}, & \text{otherwise}.
\end{cases}
\end{align*}
Note that, as $m>1$, $r=1$ only when $|\mathrm{ds}_{1}|=1$; in particular, $r=1$ implies that there is no other $w\in\calS_{1}$ such that $1^{w}=1^{v}$.
Intuitively, $r$ tells us the furthest point that a projective $P(u)\in \add T$ with $u\in\Lambda'$ reaches into $\Gamma^{v}$.

To construct $T''$ from $T'$, we consider three cases.
\begin{itemize}
\item[(I)] $|\mathrm{ds}_{1}|>1$
\item[(II)] $|\mathrm{ds}_{1}|=1$ and $r=m$
\item[(III)] $|\mathrm{ds}_{1}|=1$ and $r<m$
\end{itemize}
In the cases (I) and (II), we construct $T''$ from $T'$ by taking one mutation (Lemma \ref{lem-T''-conditon(i)(ii)}).
In the case (III), we construct $T''$ from $T'$ by taking mutation twice   (Lemmas \ref{lem:mut-T'(iii)} and \ref{lem:T''-conditon(iii)}).

\begin{example}\label{eg:mrcase}
We continue with Examples \ref{eg:mut-T-1}, \ref{eg:mut-T-Gamma-sets} and \ref{eg:mut-T-to-T'}.
For each $v\in\calS_{1}$, we have $r$ and $\mathrm{ds_{1}}$ as follows.
\begin{align}
\begin{array}{ccccc}
 v &  (m, m^{v}) & (r, r^{v}) & \mathrm{ds_{1}} & \mbox{Case} \\ \hline
 4  &  (2,2) & (2,2) & \{2, 5, 8, 9\} & {\rm (I)}\\
 7  &  (2,5) & (2,5) & \{2, 5, 8, 9\} & {\rm (I)}\\
 12 &  (2,12) & (2,12) & \{12\} & {\rm (II)} \\
 16 &  (3,15) & (2,14) & \{14\} & {\rm (III)}\\
\end{array}\notag
\end{align}
\end{example}

We first consider the cases (I) and (II).
Define
\begin{align}
X:=
\begin{cases}
S(m)\oplus \displaystyle\bigoplus_{1<j<m} P(j), & \text{if $|\mathrm{ds}_1|>1$};\\
S(m)\oplus\displaystyle\bigoplus_{1\leq j<m} P(j), & \text{if $|\mathrm{ds}_1|=1$ and $r=m$}.\\
\end{cases}\notag
\end{align}
Note that $X\in \add T'$ by Lemma \ref{lem:T-to-T'}(a) and (d).

\begin{lemma}\label{lem:mut-T'(i)(ii)}
Assume that (I) or (II) hold.
Let $X$ as above.
For any indecomposable module $U\in \add X$, if $\phi:U\rightarrow W$ is a minimal left $\add(T'/X)$-approximation, then $\phi$ is injective with indecomposable non-simple $\Cok\phi$.
\end{lemma}

\begin{proof}
\underline{Case (I)}:
By the setup of $\Gamma$ and Lemma \ref{lem:T-to-T'}(d), we can take $W=P(1)$ and $\phi$ given by left-multiplying the path $(1\rightarrow \cdots \rightarrow j)$ for $U=P(j)$ or $S(m)$ (in which case $j=m$).
Since $P(1)$ is not uniserial as $|\mathrm{ds}_1|>1$, $\Cok\phi$ is a quotient of $P(1)$ of length at least $2$; hence, non-simple.

\underline{Case (II)}:
By Lemma \ref{lem:easy-obs}, we have $\mathrm{dp}_{1}\subseteq \Lambda'$.
Thus Lemma \ref{lem:T-to-T'}(b) yields $P(w)\in \add T'$ for all $w\in\mathrm{dp}_1$.

For $w\in\mathrm{dp}_{1}$, let $\{\beta_{i}^{U,w} \mid 1 \leq i \leq \dim_{\Bbbk}\Hom_{A}(U,P(w))\}$ be a $\Bbbk$-basis of $\Hom_{A}(U,P(w))$.
When $U=P(j)$ for $1\le j<m$, we have $\Hom_{A}(U,P(w))\cong e_{w}Ae_{j}$, and so we can consider each $\beta_{i}^{U,w}$ as elements of $A$. 
When $U=S(m)$, we can still consider $\beta_{i}^{U,w} \in e_{w} A e_{m}$ since $\Hom_{A}(S(m),P(w))\cong\Hom_{A}(P(m), P(w))$ by the definition of $m(=r)$.
Then, by Lemma \ref{lem:T-to-T'}(d),
\begin{align}\label{eq:phi'}
\phi' = ( \beta_{i}^{U,w} \cdot - )_{\substack{w\in \mathrm{dp}_1 \\1\leq i\leq d({U,w})}} : U \rightarrow \bigoplus_{w\in \mathrm{dp}_1}P(w)^{\oplus d(U,w)}
\end{align}
is a left $\add(T'/X)$-approximation of $U$, where $d(U,w):=\dim_{\Bbbk}\Hom_A(U,P(w))$.
Since there must be some $w\in \mathrm{dp}_1$ with $[\soc P(w) : S(m)] \neq 0$ and $\soc P(j) = S(m)$ for all $1\le j< m$, we have $\Ker(\phi')=0$, i.e. $\phi'$ is always injective. 
A minimal left $\add(T'/X)$-approximation $\phi$ of $U$ is a direct summand of $\phi'$; such a map is also injective since $\phi'$ already is so.

It remains to show that $\Cok\phi$ is indecomposable non-simple.
By Proposition \ref{prop:tilting-mutation}, $\Cok\phi$ is indecomposable.
If $U=S(m)$ or $U=P(j)$ with $1<j<m$, then we always have $[\Cok\phi: S(1)]\neq 0$ and $[\top \Cok\phi: S(w)]\neq 0$ for some $w\in\mathrm{dp}_1$.  As $w\neq 1$, $\Cok\phi$ is necessarily non-simple.

In the case $U=P(1)$, assume on the contrary that $\Cok\phi$ is simple. 
Since $\phi$ can be taken as a direct summand of $\phi'$ given in \eqref{eq:phi'}, the assumption of $\Cok\phi$ being simple implies that $\mathrm{dp}_{1}=\{w\}$ with a unique connecting arrow $\alpha:w\rightarrow 1$ such that $\phi=(\alpha\cdot-): P(1)\rightarrow P(w)$.
Now we have a full subquiver $w \rightarrow 1 \rightarrow 2 \rightarrow \dots \rightarrow n$ of $Q$.
As $[\soc P(w) : S(m)]\neq 0$, this full subquiver satisfies (i) (ii) and (iii) of the assumption of $\Gamma$.
But this contradicts the maximal property (iv) in the definition of $\Gamma$.
Therefore $\Cok\phi$ is non-simple.
We complete the proof.
\end{proof}

In the setting of Lemma \ref{lem:mut-T'(i)(ii)}, we have $T''$ as follows.

\begin{lemma}\label{lem-T''-conditon(i)(ii)}
Assume that (I) or (II) hold.
Under the notation in Lemma \ref{lem:mut-T'(i)(ii)}, $T'':=\mu_X(T')$ is a tilting $A$-module that satisfies (S) with $\calS(T'') = \calS(T)\setminus\{v\}$.
\end{lemma}
\begin{proof}
By Lemma \ref{lem:mut-T'(i)(ii)}, $T''$ is tilting.
Moreover, Lemma \ref{lem:mut-T'(i)(ii)} also implies that $\calS(T'') = \calS(T')\setminus\{m\} =\calS(T)\setminus\{v\}$; hence, {\rm (S1)} is satisfied.
The condition {\rm (S3)} are inherited from (c) of $T'$ in Lemma \ref{lem:T-to-T'}.
The condition {\rm (S2)} follows from the choice of $X$ and (b) for $T'$ in Lemma \ref{lem:T-to-T'}.
Indeed, we only need to be careful when $P(1)\in \add X$, but this is the case (II). Since $|\mathrm{ds}_1|=1$, we have $1^v\notin \Lambda'$ by definition. Thus, $\add(\bigoplus_{w\in \Lambda'}P(w)) \cap \add X=0$ always.
\end{proof}

Next we consider the case (III).

\begin{lemma}\label{lem:mut-T'(iii)}
Assume that (III) holds.
Let $Y:= S(m)\oplus \bigoplus_{i=r+1}^{m-1}P(i)$. Then $Y$ is a mutable direct summand of $T'$. Moreover, the tilting $A$-module $\widetilde{T}:=\mu_{Y}(T')$ satisfies the following properties. 
\begin{itemize}
\item[(a)] $\calS(\widetilde{T}) = (\calS(T)\setminus\{v\}) \cup \{r\}$.
\item[(b)] $P(u)\in \add \widetilde{T}$ for all $u\in \Lambda'$.
\item[(c)] If $w\in\calS(\widetilde{T})\cap \calS_{1}$, then $P(u)\in \add \widetilde{T}$ for all $u\in \Gamma^{w}_{0}$.
\item[(d)] Let $U:=\bigoplus_{i=1}^{r-1} P(i) \oplus \bigoplus_{i=r+1}^{m-1} (P(r)/P(i)) \oplus (P(r)/S(m)) \oplus ([0, \kD C]/S(m) )$. Then $U\in\add \widetilde{T}$. Moreover, for an indecomposable $V\in\add \widetilde{T}$, if $x\in \Gamma_{0}^{v}$ holds whenever $[V:S(x)]\neq 0$, then we have $V\in\add U$.  
\end{itemize}
\end{lemma}

\begin{proof}
We show that $Y$ is a mutable direct summand of $T'$.
By Lemma \ref{lem:T-to-T'}(a) and (d), $Y\in\add T'$.
It is shown that for each indecomposable direct summand $Y_1$ of $Y$, a natural inclusion $\iota : Y_1 \rightarrow P(r)$ is a minimal left $\add(T'/Y)$-approximation of $Y_{1}$. 
Indeed, let $T_{1}\in\add(T'/Y)$ be an indecomposable module such that there exists a non-zero morphism $f:Y_{1} \rightarrow T_{1}$.
By the definition of $r$, the composition factors of $T_{1}$ are concentrated in $\Gamma_{0}^{v}$.
Thus Lemma \ref{lem:T-to-T'}(d) yields $T_{1}\in \add(\bigoplus_{i=1}^{m-1}P(i))$.
This implies that $f$ factors through $\iota$.
Thus $\iota$ is a minimal left $\add(T'/Y)$-approximation.
The direct sum $Y \rightarrow P(r)^{\oplus (m-r)}$ is a minimal left $\add(T'/Y)$-approximation of $Y$, which is injective.

We show the latter part. Since $Y$ is mutable, we have
\begin{align}
\widetilde{T}=\mu_Y(T')=T'/Y\oplus \bigoplus_{i=r+1}^{m-1}(P(r)/P(i))\oplus P(r)/S(m).\notag
\end{align}
The properties (a) and (c) follow from Lemma \ref{lem:T-to-T'}(a) and (c) respectively. We show (b). If $r=1$, then we have $|\mathrm{ds_{1}}|=1$, and hence $r=1\notin \Lambda'$. 
This implies $\{ r,r+1, \ldots, n\}\cap \Lambda'=\emptyset$ for each $r\geq 1$. Thus (b) follows from Lemma \ref{lem:T-to-T'}(b).
The former part of (d) follows from the construction of $\widetilde{T}$ and Lemma \ref{lem:T-to-T'}(d). The latter part of (d) comes from the same argument as the proof of the latter part of Lemma \ref{lem:T-to-T'}(d).
\end{proof}

We obtain $T''$ by mutating $\widetilde{T}$ as follows.

\begin{lemma}\label{lem:T''-conditon(iii)}
Assume that (III) holds.
Let $\widetilde{T}=\mu_Y(T')$ be the tilting $A$-module constructed in Lemma \ref{lem:mut-T'(iii)}.
Then $S(r)$ is a mutable direct summand of $\widetilde{T}$. Moreover, the tilting $A$-module $T'':=\mu_{S(r)}(\widetilde{T})$ satisfies (S) with $\calS(T'')=\calS(T)\setminus\{v\}$.
\end{lemma}

\begin{proof}
We show that $S(r)$ is a mutable direct summand of $\widetilde{T}$.
By Lemma \ref{lem:mut-T'(iii)}(a), we have $S(r)\in\add \widetilde{T}$.
Consider now
\begin{align}
W':=\bigoplus_{w\in\mathrm{dp}^{r}_{1}} P(w)^{\oplus \dim_{\Bbbk} \Hom_{A}(S(r),P(w))},\notag
\end{align}
where $\mathrm{dp}_{1}^{r}:=\{ w\in \mathrm{dp}_{1}\mid [P(w):S(r)]\neq 0\}$.
By Lemma \ref{lem:easy-obs}, we have $\mathrm{dp}_{1}\subseteq \Lambda'$.
Thus Lemma \ref{lem:mut-T'(iii)}(b) yields $W'\in \add(\widetilde{T}/S(r))$. Take a $\Bbbk$-basis $\{\beta_{1}^{w},\ldots, \beta^{w}_{k_{w}}\}$ of $\Hom_A(S(r),P(w))\cong e_{w}Ae_{r}$.
Then $\phi':=(\beta_{i}^{w}\cdot-)_{w\in\mathrm{dp}_{1}, \ 1\le i\le k_{w}}: S(r)\rightarrow W'$ is a left $\add(\widetilde{T}/S(r))$-approximation of $S(r)$.
Indeed, let $V\in \add(\widetilde{T}/S(r))$ be indecomposable with $\Hom_{A}(S(r),V)\neq 0$. By Lemma \ref{lem:mut-T'(iii)}(d), there exists $w\in \mathrm{dp}_{1}^{r}$ such that $Ve_{w}\neq 0$.
Thus each morphism $S(r) \rightarrow V$ factors through $P(w)$ for some $w\in \mathrm{dp}^{r}_{1}$. This implies that $\phi'$ is a left $\add(\widetilde{T}/S(r))$-approximation. Thus we can take a minimal left $\add(\widetilde{T}/S(r))$-approximation $\phi: S(r)\rightarrow W$. Since $S(r)$ is simple, $\phi$ is injective.

We now show that $\Cok\phi$ is indecomposable non-simple.
By Proposition \ref{prop:tilting-mutation}, $\Cok\phi$ is indecomposable.
If $r>1$, then we always have $[\Cok\phi: S(r-1)]\neq 0$ and $[\top \Cok\phi: S(w)]\neq 0$ for some $w\in\mathrm{dp}_{1}$. Thus $\Cok\phi$ is necessarily non-simple.
If $r=1$, then $\phi$ can be written as $S(r) \xrightarrow{(\alpha\cdot-)_{\alpha}} \bigoplus_{(\alpha:w\rightarrow 1)\in Q_{1}}P(w)$. 
We show that $k:=|\{ \alpha:w\rightarrow 1 \mid w\in \mathrm{dp}_{1}\}|>1$. If $k=1$, then $\Gamma:=\Gamma^{v}$ contradicts the maximal property (iv) in the definition of $\Gamma^{v}$. 
Thus we have $k>1$. This implies that $\Cok\phi$ is non-simple.

By the construction of $T''$, it follows from Lemma \ref{lem:mut-T'(iii)} that the latter assertion holds. The proof is complete.
\end{proof}

Now, we are ready to prove Proposition \ref{prop:non-qln-non-ist}.

\begin{proof}[Proof of Proposition \ref{prop:non-qln-non-ist}]
Let $T$ be a tilting $A$-module satisfying ({\rm S}).
Fix one $v\in\calS(T)$ if $\calS(T) \neq \emptyset$.
By Lemmas \ref{lem:T-to-T'}, \ref{lem-T''-conditon(i)(ii)} and \ref{lem:T''-conditon(iii)}, we can construct a tilting $A$-module $T''$ satisfying ({\rm S}) with $\calS(T'')=\calS(T)\setminus\{v\}$.

By Lemma \ref{T0-condition}, $T_{0}$ is a tilting module satisfying ({\rm S}) with $\calS(T_{0})=\calS_{1}$.
Thus we can repeat the above procedure $|\calS_{1}|$ times from $T_{0}$, and we have a sequence of tilting $A$-modules
\begin{align}
T_{0} \overset{\text{ mutate }}{\leadsto} T_{0}''=:T_{1} \overset{\text{ mutate }}{\leadsto} T_{1}''=:T_{2} \overset{\text{ mutate }}{\leadsto}\cdots \overset{\text{ mutate }}{\leadsto}T_{|\calS_{1}|-1}''=:T_{|\calS_{1}|}\notag
\end{align}
such that $|\calS(T_{a})| = |\calS_{1}| - a$.
Therefore we have $\calS(T_{|\calS_{1}|}) = \emptyset$, that is, $T_{|\calS_{1}|}$ has no simple direct summand as desired.
\end{proof}

\begin{example}
We continue with Examples \ref{eg:mut-T-1}, \ref{eg:mut-T-Gamma-sets}, \ref{eg:mut-T-to-T'} and \ref{eg:mrcase}.
The case (I) (respectively, (II), (III)) happens when $v=4, 7$ (respectively, $v=12$, $v=16$) by Example \ref{eg:mrcase}.
Write $U=A/(P(1)\oplus P(2)\oplus P(3)\oplus P(4)\oplus P(8)) \oplus C(8)$.
We already have $T_{0} = (A/P(8)) \oplus C(8)$ and $T' = U\oplus P(1) \oplus P(2) \oplus (P(2)/P(4)) \oplus S(2)$.  In the next step we take $X=S(2)$ which is a direct summand of $\soc P(1)$. 
Following the procedure yields
\begin{align}
T_{1} = T'' = \mu_{S(2)}(T') &= \mu_{S(2)}\left( U \oplus P(1) \oplus P(2) \oplus \frac{P(2)}{P(4)} \oplus S(2) \right) \notag \\
& = U \oplus P(1) \oplus P(2) \oplus \frac{P(2)}{P(4)} \oplus \frac{P(1)}{S(2)}.\notag
\end{align}
So far we have eliminated the simple modules $S(8)=P(8)$ and $S(4)=P(4)$, and it remains to do so for $S(7), S(12), S(16)$.

For the remaining steps, we summarise the data as follows, and leave the reader to check for themselves. 
For Cases (I) and (II), we have the following data.
\begin{align}
\begin{array}{cccccc}
a & v & \mbox{Case} &  T' & X & T_{a} = T''\\ \hline 
2 & 7 & {\rm (I)} &  \frac{T_{1}}{S(7)}\oplus S(5) & S(5) & \frac{T_{1}}{S(7)}\oplus \frac{P(1)}{S(5)}\\ [1ex]
3 & 12 & {\rm (II)} & T_{2} & P(11)\oplus S(12) & \frac{T_{2}}{P(11)\oplus S(12)}\oplus C_{11} \oplus C_{12} 
\end{array}\notag
\end{align}
Here $C_{11}=\Cok\left( P(11)\rightarrow P(9)\oplus P(10) \right)$ and $C_{12}=\Cok\left( S(12)\rightarrow P(9)\oplus P(10) \right)$.
For Case (III), we have the following data.
\begin{align}
\begin{array}{ccccccc}
a & v & \mbox{Case} &  T' & Y & \widetilde{T}  & T_{a} = T''\\ \hline 
4 & 16 & {\rm (III)} & \frac{T_{3}}{S(16)}\oplus S(15) & S(15) & \frac{T_{3}}{S(16)}\oplus S(14) &  \frac{T_{3}}{S(16)} \oplus \frac{P(9)}{S(14)}
\end{array}\notag
\end{align}
Now $T_{4}$ is a tilting module without simple direct summand.
\end{example}

\section{Tilting modules over quadratic linear Nakayama algebras}
\label{sec:tilt(qLNaka)}

In this section, we show how to recursively construct tilting modules over a quadratic linear Nakayama algebra $A$ from those of smaller ranks, namely, $eAe$ and $A/AeA$ for certain choices of idempotents $e$.

\subsection{Gluing tilting modules}\label{subsec:co/tilt gluing}

In this subsection, we provide some sufficient conditions under which (co)tilting modules of an algebra $A$ can be glued from, and decomposed into, (co)tilting modules of $eAe$ and $A/AeA$.

Let $A$ be an arbitrary basic finite-dimensional algebra and let $\{ e_{x}\}_{x\in \Lambda}$ be a complete set of primitive orthogonal idempotents of $A$.
Fix an idempotent $\e=\sum_{x\in \Gamma}e_{x} \in A$ and consider the adjoint triple $(\sfl, \sfe, \sfr)$ between $\mod A$ and $\mod(\e A\e)$ as in subsection \ref{subsec:recolle subcats}.
For simplicity, we write $\calP^{\le k}_{\e}$ and $\calI^{\le k}_{\e}$ instead of $\calP^{\le k}(\Gamma)$ and $\calI^{\le k}(\Gamma)$ respectively. 
For $M\in \mod A$, we consider a decompose $M=K_{\e}(M)\oplus R_{\e}(M)$, where $K_{\e}(M)$ is the maximal direct summand of $M$ in  $\Ker \sfe$ and $R_{\e}(M):=M/K_{\e}(M)$.  
In particular, $K_{\e}(M)$ is in the image of the embedding $\mod(A/A\e A)\rightarrow \mod A$; thus, we will also regard $K_{\e}(M)$ as an $(A/A\e A)$-module.

For pairs $(U,V), (U',V')\in \mod (A/A\e A)\times \mod(\e A\e)$, we write $(U,V)\geq (U',V')$ if $U\geq U'$ and $V\geq V'$, that is, $\Ext_{A/A\e A}^{>0}(U,U')=0$ and $\Ext_{\e A\e}^{>0}(V,V')=0$.  
Note that $(\tilt(A/A\e A)\times \tilt(\e A\e), \geq)$ is a poset given by the product of the individual tilting posets.

We say that an algebra $A$ \defn{satisfies the rank-tilting-completion property} if every pretilting $A$-module $T$ with $|T|=|A|$ is tilting, i.e. (T1)$+$(T2)$+$(T3') implies (T1)$+$(T2)$+$(T3).
Note that there is no known example of an algebra that fails the rank-tilting-completion property.
Any algebra that is representation-finite satisfies the rank-tilting-completion property.
Similarly, we can define rank-cotilting-completion property.
In the case when the algebra $A$ has finite global dimension, which implies that $\tilt A=\cotilt A$, then satisfying any one of the completion properties implies the other one.

The following poset isomorphisms are the keys in this section.

\begin{proposition}\label{prop:main-bijections}
Let $\e$ be an idempotent of a basic finite-dimensional algebra $A$.
\begin{itemize}
\item[(1)] Suppose that one of the following conditions holds.
\begin{itemize}
\item[(i)] $A/A\e A$ is projective as an $A$-module.
\item[(ii)] $\pdim_A(A/A\e A)<\infty$, $A\varepsilon A$ is a stratifying ideal, and $A$ satisfies the rank-tilting-completion property.
\end{itemize}
Then there exist poset isomorphisms
\begin{align}
\begin{aligned}
\xymatrix{
\{ T\in \tilt A\mid R_{\e}(T)\in \calP^{\leq \infty}_{\e}\}\ar@<-1mm>[d]_-{\varphi}\\
\{ (U,V)\in \tilt (A/A\e A)\times \tilt^{\lex}(\e A\e )\mid  \Ext_{A}^{>0}(U,\sfl(V))=0\},\ar@<-1mm>[u]_-{\psi}
}
\end{aligned}
\end{align}
where $\varphi(T):=(K_{\e}(T), \sfe(R_{\e}(T)))$ and $\psi(U,V):=U\oplus \sfl(V)$. 
Moreover, in the case where $A/A\e A$ is semisimple and projective as an $A$-module, then the isomorphisms induce the following poset isomorphisms
\begin{align}\label{eq:P-infty=lex}
\{ T\in \tilt A\mid R_{\e}(T)\in \calP^{\leq \infty}_{\e}\}=\tilt_{(1-\e)A}A \cong \tilt^{\lex}(\e A\e).
\end{align}
\item[(2)] Suppose that one of the following conditions holds.
\begin{itemize}
\item[(i)] $\kD (A/A\e A)$ is injective as an $A$-module.
\item[(ii)] $\idim_A(\kD(A/A\e A))<\infty$, $A\varepsilon A$ is a stratifying ideal, and $A$ satisfies the rank-cotilting-completion property.
\end{itemize}
Then there exist poset isomorphisms
\begin{align}
\begin{aligned}\label{eq:cotilt-glue-bijn}
\xymatrix{
\{ T\in \cotilt A\mid R_{\e}(T)\in \calI^{\leq \infty}_{\e}\}\ar@<-1mm>[d]_-{\varphi}\\
\{ (U,V)\in \cotilt (A/A\e A)\times \cotilt^{\rex}(\e A\e )\mid  \Ext_{A}^{>0}(\sfr(V), U)=0\},\ar@<-1mm>[u]_-{\psi}
}
\end{aligned}
\end{align}
where $\varphi(T):=(K_{\e}(T), \sfe(R_{\e}(T)))$ and $\psi(U,V):=U\oplus \sfr(V)$. 
Moreover, in the case where $\kD(A/A\e A)$ is semisimple and injective as $A$-modules, then the isomorphisms induce the following poset isomorphisms
\begin{align}
\{ T\in \cotilt A\mid R_{\e}(T)\in \calI^{\leq \infty}_{\e}\}=\cotilt_{\kD A(1-\e)}A \cong \cotilt^{\rex}(\e A\e).\notag
\end{align}
\end{itemize}
\end{proposition}

In the following, we show Proposition \ref{prop:main-bijections}.

\begin{lemma}\label{lem:mapvarphi}
Let $\e\in A$ be an idempotent.
The following statements hold.
\begin{itemize}
\item[(1)] Let $T$ be a basic tilting $A$-module with $R_{\e}(T)\in \calP^{\leq \infty}_{\e}$. 
\begin{itemize}
\item[(a)] $\Ext_{A}^{>0}(K_{\e}(T),\sfl\sfe(R_{\e}(T)))=0$. 
\item[(b)] $\sfe(R_{\e}(T))\in \tilt^{\lex}(\e A\e)$.
\item[(c)] If $A\e A$ is stratifying, then $K_{\e}(T)\in \tilt(A/A\e A)$.
\end{itemize}
\item[(2)] Let $T$ be a basic cotilting $A$-module with $R_{\e}(T)\in \calI^{\leq \infty}_{\e}$. 
\begin{itemize}
\item[(a)] $\Ext_{A}^{>0}(\sfr\sfe(R_{\e}(T)), K_{\e}(T))=0$.
\item[(b)] $\sfe(R_{\e}(T))\in \cotilt^{\rex}(\e A\e)$.
\item[(c)] If $A\e A$ is stratifying, then $K_{\e}(T)\in \cotilt(A/A\e A)$.
\end{itemize}
\end{itemize}
\end{lemma}

\begin{proof}
We show only (1).
Let $T$ be a basic tilting module with $R_{\e}(T)\in \calP^{\leq \infty}_{\e}$.

(a) This follows from Lemma \ref{lem:aus-apt}(1), which says that
$\Ext_{A}^{k}(K_{\e}(T), \sfl\sfe(R_{\e}(T)))\cong \Ext_{A}^{k}(K_{\e}(T),R_{\e}(T))=0$ for all $k>0$.

(b) By Proposition \ref{prop:ex-tilting}(1), it is enough to show that $R_{\e}(T)\in \tilt \calP^{\leq \infty}_{\e}$. 
Since $R_{\e}(T)$ is a direct summand of $T$, we show (TP3). 
By (T3) for $T$, we can take an exact sequence 
\begin{align}
0\rightarrow \e A \xrightarrow{f_{0}} T_{0} \xrightarrow{f_{1}} T_{1} \rightarrow \cdots \xrightarrow{f_{l}} T_{l} \rightarrow 0
\notag
\end{align}
such that $f_{i}$ is left minimal and $T_{i}\in \add T$ for all $i$.
Since $f_{0}$ is left minimal and $\Hom_{A}(\e A, K_{\e}(T))=0$, we have $T_{0}\in\add(R_{\e}(T))$. 
By $T_{0}\in\add(R_{\e}(T))\subseteq\calP^{\leq \infty}_{\e}$ and $K_{\e}(T)\in\mod(A/A\e A)$, we obtain $\Hom_{A}(T_{0}, K_{\e}(T))=0$, and hence $\Hom_{A}(\Cok f_{0}, K_{\e}(T))=0$.
It follows that the left minimality of $f_{1}$ yields $T_{1}\in\add(R_{\e}(T))$. 
Repeating this process, we have $T_{i}\in\add(R_{\e}(T))$ for all $0\leq i\leq l$. Thus, $R_{\e}(T)$ is a tilting object in $\calP^{\leq \infty}_{\e}$.

(c) Assume that $A\varepsilon A$ is a stratifying ideal. Then we have a natural isomorphism
\begin{align}
\Ext_{A/A\e A}^{k}(K_{\e}(T),-)\rightarrow \Ext_{A}^{k}(K_{\e}(T),-)\notag
\end{align}
on $\mod(A/A\e A)$. Thus (T1) and (T2) are satisfied by $K_{\e}(T)$. We show (T3). By (T3) for $T$, there exists an exact sequence
\begin{align}
0\rightarrow A \rightarrow T_{0} \rightarrow T_{1} \rightarrow \cdots \rightarrow T_{l} \rightarrow 0 \notag
\end{align}
such that $T_{i}\in \add T$ for all $i$. Applying $-\otimes_{A}A/A\e A$ to the sequence above yields a complex in $\mod(A/A\e A)$
\begin{align}
0\rightarrow A/A\e A \rightarrow T_{0}/T_{0}\e A \rightarrow T_{1}/T_{1}\e A \rightarrow \cdots \rightarrow T_{l}/T_{l}\e A\rightarrow 0\notag
\end{align}
with $T_{i}/T_{i}\e A\in \add(T/T\e A)$.
By $R_{\e}(T)\in \calP^{\leq \infty}_{\e}$, it follows from Lemma \ref{lem:aus-apt}(1) that $R_\e(T)\e\otimes_{\e A\e}\e A\cong R_\e(T)$ holds. Since $A\e A$ is stratifying, we have $R_{\e}(T)\cong R_{\e}(T)\e A$, and hence $R_{\e}(T)/R_{\e}(T)\e A=0$. By $K_{\e}(T)\e=0$, we have $T/T\e A\cong K_{\e}(T)$.

It remains to show that the complex above is exact.
It suffices to show that $\Tor_{>0}^A(T,A/A\e A)=0$.
Indeed, by $\pdim_{A}R_{\e}(T)<\infty$ and $R_{\e}(T)\in \calP^{\leq \infty}_{\e}$, there exists a projective resolution
\begin{align}
0\rightarrow P_{k}\rightarrow \cdots \rightarrow P_{1} \rightarrow P_{0} \rightarrow R_{\e}(T)\rightarrow 0\notag
\end{align}
such that $P_{i}\in \add(\e A)$. Since $\e A\otimes_{A} A/A\e A=0$, we have $\Tor_{>0}^{A}(R_{\e}(T),A/A\e A)=0$. Moreover, by Lemma \ref{lem:str-idemp}, we obtain $\Tor_{>0}^{A}(K_{\e}(T),A/A\e A)=0$. Thus we have $\Tor_{>0}^{A}(T,A/A\e A)=0$ as required.
\end{proof}

\begin{lemma}\label{lem:mappsi}
Assume that $A\varepsilon A$ is a stratifying ideal.
The following statements hold.
\begin{itemize}
\item[(1)] Let $U$ be a pretilting $(A/A\e A)$-module and $V$ a pretilting $\e A\e$-module with $\Tor_{>0}^{\e A\e}(V,\e A)=0$. If $\pdim_{A}(A/A\e A)<\infty$ and $\Ext_{A}^{>0}(U,\sfl(V))=0$, then $U\oplus \sfl(V)$ is a pretilting $A$-module.
\item[(2)] Let $U$ be a precotilting $(A/A\e A)$-module and $V$ a precotilting $\e A\e$-module with $\Ext^{>0}_{\e A\e}(A\e,V)=0$. If $\idim_{A}(A/A\e A)<\infty$ and $\Ext_{A}^{>0}(\sfr(V),U)=0$, then $U\oplus \sfr(V)$ is a precotilting $A$-module.
\end{itemize}
\end{lemma}

\begin{proof}
We show only (1).
By \cite[Lemma 5.8]{APT92}, we have 
\begin{align}
\pdim_{A}U\le \pdim_{A/A\e A}U+\pdim_{A}(A/A\e A).\notag
\end{align}
We have $\pdim_{A}(A/A\e A)<\infty$ by the assumption, and $U$ being  pretilting implies that $\pdim_{A/A\e A}U<\infty$. 
By Lemma \ref{lem:str-idemp}, we have an isomorphism $0=\Ext_{A/A\e A}^{k}(U,U)\cong\Ext_{A}^{k}(U,U)$. 
By Proposition \ref{prop:ex-tilting}(1), $\sfl(V)$ is a pretilting $A$-module, and hence we obtain $\pdim \sfl(V)<\infty$ and $\Ext_{A}^{>0}(\sfl(V), \sfl(V))=0$.
Since it follows from Lemma \ref{lem:apt-pfin} and $\sfl(V)\in\calP^{\leq \infty}_{\varepsilon}$ that $\Ext_{A}^{>0}(\sfl(V),U)=0$, we have $\Ext_{A}^{>0}(U\oplus \sfl(V),U\oplus \sfl(V))=0$. The proof is complete.
\end{proof}

Now, we are ready to prove Proposition \ref{prop:main-bijections}.

\begin{proof}[Proof of Proposition \ref{prop:main-bijections}]
We prove only (1).
For both cases (i) and (ii), $A\varepsilon A$ is a stratifying ideal by Lemma \ref{lem:proj-stridemp}(1)$\Rightarrow$(4). 
By Lemma \ref{lem:mapvarphi}(1), $\varphi$ is well-defined.
We show that $\psi$ is well-defined.
Let $(U,V)\in \tilt(A/A\e A)\times \tilt^{\lex}(\e A\e)$ with $\Ext_{A}^{>0}(U,\sfl(V))=0$.
By Lemma \ref{lem:mappsi}(1), $U\oplus \sfl(V)$ is a pretilting $A$-module.
If (i) is satisfied, then $A/A\e A\cong (1-\e)A$.
By (T3) for $U\in \tilt(A/A\e A)$, there exists an exact sequence
\begin{align}
0\rightarrow (1-\e)A \rightarrow U_{0}\rightarrow U_{1}\cdots \rightarrow U_{l}\rightarrow 0 \notag
\end{align}
such that $U_{i}\in \add U$.
By (T3) for $V$, there exists an exact sequence
\begin{align}
0\rightarrow \e A\e \rightarrow V_{0} \rightarrow V_{1}\rightarrow \cdots \rightarrow V_{l} \rightarrow 0 \notag
\end{align}
such that $V_{i}\in \add V$. Since $\sfl(V)\in \calP^{\leq\infty}_{\e}$, applying $\sfl$ to the sequence above yields an exact sequence
\begin{align}
0\rightarrow \e A \rightarrow \sfl(V_{0}) \rightarrow \sfl(V_{1}) \rightarrow \cdots \rightarrow \sfl(V_{l}) \rightarrow 0 \notag
\end{align} 
with $\sfl(V_{i})\in \add(\sfl(V))$.
Thus $U\oplus \sfl(V)$ is tilting.
Assume that (ii) is satisfied. Since 
\begin{align}
|U\oplus \sfl(V)|=|U|+|\sfl(V)|=|A/A\e A|+|\e A\e|=|A|,\notag
\end{align}
the rank-tilting-completion property yields that $U\oplus \sfl(V)$ is tilting.
Hence $\psi$ is well-defined.

Moreover, we have
\begin{align}
\varphi\psi(U,V) &=\varphi(U\oplus \sfl(V))=(K_{\e}(U\oplus \sfl(V)),\sfe R_{\e}(U\oplus \sfl(V)))=(U,V),\notag \\
\text{and }\quad \psi\varphi(T) &=\psi(K_{\e}(T),\sfe R_{\e}(T))=K_{\e}(T)\oplus \sfl\sfe R_{\e}(T)=K_{\e}(T)\oplus R_{\e}(T)=T.\notag
\end{align}
Thus $\varphi$ and $\psi$ are isomorphisms.

Finally, we show the first equality in \eqref{eq:P-infty=lex}.
First note that $A/A\e A$ being projective implies that $A/A\e A\cong (1-\e)A$ as right $A$-modules.  In particular, we have $(1-\e)A\e =0$.
Let $T$ be a tilting $A$-module with $R_{\e}(T)\in \calP^{\leq \infty}_{\e}$.
By Lemma \ref{lem:mapvarphi}(1), $K_{\varepsilon}(T)\in\tilt (A/A\varepsilon A)$. Since $A/A\varepsilon A$ is semisimple, we have  $K_{\varepsilon}(T)\cong A/A\varepsilon A=(1-\varepsilon)A$. Therefore $T\in \tilt_{(1-\e)A}A$. 
Conversely, let $T\in\tilt_{(1-\e)A}A$. 
To obtain $R_{\e}(T)\in \calP^{\leq\infty}_{\e}$, it is enough to show $\Ext_{A}^{\geq 0}(R_{\e}(T),\top (1-\e)A )=0$ by Lemma \ref{lem:apt-pfin}.
Since $A/A\e A$ is a semisimple $A$-module and $(1-\e)A\e=0$, we have $\top (1-\e)A\cong (1-\e)A\in \add(K_{\e}(T))$. Thus it follows from $\Ext_A^{\ge 0}(R_{\e}(T),K_{\e}(T))=0$ that $\Ext_{A}^{\geq 0}(R_{\e}(T),\top (1-\e)A)=0$, as required.
\end{proof}

As an application of Proposition \ref{prop:main-bijections}, we have the following result, which will be used later.

\begin{corollary}\label{cor:application_bijection-5.1}
Let $A$ be an algebra and $\varepsilon=1-e_{i}$.  
Assume that $\pdim S(i)=0$ and $\idim S(i)\leq 1$. 
Then the following statements hold.
\begin{itemize}
\item[(1)] We have $\tilt^{\lex}(\e A\e)=\tilt(\e A\e)$. Moreover, there exist poset isomorphisms
\begin{align}
\xymatrix{\tilt_{S(i)}A\ar@<1mm>[r]^-{\varphi}&\tilt(\varepsilon A\varepsilon)\ar@<1mm>[l]^-{\psi}},\notag
\end{align}
where $\varphi(T):=\sfe T=T\varepsilon$ and $\psi(V):=S(i)\oplus \sfl(V)$. 
\item[(2)] The restrictions of (1) give poset isomorphisms between the set of IS-tilting $A$-modules containing $S(i)$ as a direct summand and the set of IS-tilting $\varepsilon A\varepsilon$-modules.
In particular, $\{[\unlhd]\in\qhs A \mid i\in \min \unlhd^{\rm m}\}\cong \qhs(\varepsilon A\varepsilon)$ holds.
\end{itemize}
\end{corollary}

\begin{proof}
(1) Since $S(i)$ is projective, the $A$-module $A/A\varepsilon A$ is semisimple and projective. 
By $\idim S(i)\leq 1$, we have $Je_{i}\in \proj(A^{\op})$, where $J$ is the Jacobson radical of $A$. 
Thus 
\begin{align}
A\varepsilon A =A\varepsilon Ae_{i}\oplus A\varepsilon A\varepsilon=A\varepsilon Je_{i} \oplus A\varepsilon=Je_{i}\oplus A\varepsilon\in\proj(A^{\op}).\notag
\end{align}
By Lemma \ref{lem:proj-stridemp}(1)$\Rightarrow$(3), $\e A$ is a projective $(\e A\e)^{\op}$-module. This implies $\tilt^{\lex}(\varepsilon A\varepsilon)=\tilt (\varepsilon A\varepsilon)$. 
By Proposition \ref{prop:main-bijections}(1), we have the assertion. 

(2) If $T\in\IStilt_{S(i)}A$, then $\varphi(T)\in\IStilt(\varepsilon A\varepsilon)$ by Proposition \ref{prop:IS-tilt}(a). 
Conversely, let $V = (V,\unlhd') \in \IStilt(\varepsilon A\varepsilon)$. 
Let $\unlhd$ be a partial order obtained by adding $i$ to $\unlhd'$ as a minimum element. Then the pair $(T=S(i)\oplus \sfl(V),\unlhd)$ satisfies (I1) by taking $T(i)=S(i)$. 
Hence $\psi(V)=S(i)\oplus \sfl(V)$ is an IS-tilting $A$-module by Proposition \ref{prop:IS-tilt}.
This shows that the first assertion holds. 

By Proposition \ref{prop:isttilt-chtilt}, a quasi-hereditary structure is uniquely determined by an IS-tilting module. 
Thus it follows from Theorem \ref{thm:bij-qh-IStilting} that the set of IS-tilting $A$-module with a simple direct summand $S(i)$ bijectively corresponds to the set of quasi-hereditary structures $[\unlhd]$ such that $i$ is minimal with respect to $\unlhd^{\rm m}$. 
\end{proof}

Dually, we have the following result. 

\begin{remark}\label{rem:eAe-reduction-on-cotilt}
Let $A$ be an algebra and $\varepsilon=1-e_{i}$.  
Assume that $\idim S(i)=0$ and $\pdim S(i)\leq 1$. 
Then we have $\cotilt_{\kD (A(1-\varepsilon))}A\cong \cotilt(\varepsilon A\varepsilon)$ by Proposition \ref{prop:main-bijections}(2). Since $\tilt_{S(i)^{\op}}A^{\op}\cong \cotilt_{\kD (A(1-\varepsilon))}A\cong \cotilt(\varepsilon A\varepsilon) \cong \tilt(\varepsilon A\varepsilon)^{\op}$, we obtain
\begin{align}
\{[\unlhd]\in\qhs A\mid i\in\min\unlhd^{\rm{m}}\}&\cong  \{[\unlhd]\in\qhs(A^{\op})\mid i\in\min\unlhd^{\rm{m}}\}\notag\\&\cong \IStilt_{S(i)^{\op}}A^{\op}\cong \IStilt(\varepsilon A\varepsilon)^{\op}\cong \qhs(\varepsilon A\varepsilon)^{\op}\cong \qhs(\varepsilon A\varepsilon).  \notag
\end{align}
\end{remark}

\subsection{Application to quadratic linear Nakayama algebras}\label{subsec:glue-tilt(qLNaka)}

In this subsection, we will apply the results from the previous subsection to reveal a recursive construction of tilting modules over a quadratic linear Nakayama algebra in Theorem \ref{thm:decomp}.  We will then use this to recover information about the quasi-hereditary structures of such an algebra in Corollary \ref{prop:istilt-qhs}. 

Since a linear Nakayama algebra is representation-finite and is of finite global dimension, it follows from  Remark \ref{rem:tilting/qln} that a module $T$ is tilting if and only if it satisfies (T2) and (T3') if and only if it is cotilting. This implies that a linear Nakayama algebra satisfies the rank-tilting-completion property and the rank-cotilting-completion property. 

Throughout this subsection, $A$ is a quadratic linear Nakayama algebra, that is, $A\cong \Bbbk \Vec{\bbA}_{n}/I$, where 
\begin{align}
\Vec{\bbA}_{n}:  1\xrightarrow{\alpha_{1}} 2 \xrightarrow{\alpha_{2}} \cdots \xrightarrow{\alpha_{\ell-1}} \ell \xrightarrow{\alpha_{\ell}} \cdots \xrightarrow{\alpha_{n-2}} n-1 \xrightarrow{\alpha_{n-1}} n \notag
\end{align}
and $I$ is generated by a set of paths of length two.
Let $\relvx(A):=\{ i\in (\Vec{\bbA}_{n})_{0} \mid \alpha_{i-1}\alpha_{i}\in I\}$.
For simplicity, we assume that $A$ is not semisimple, i.e. $n
\geq 2$.
Then there exists an integer $\ell\in [1,n-1]$ such that $P(\ell)=I(n)$, or equivalently $\ell\in \relvx(A)$ and $\ell+1, \ldots, n\notin \relvx(A)$.

For an integer $i\in [\ell,n]$, we define an idempotent $f_{i}\in A$ as
\begin{align}
f_{i}:=
\begin{cases}
e_{n}, & \text{if $i=\ell$};\\
e_{1}+\cdots+e_{i-1}+e_{n}, & \text{if $i\in [\ell+1,n-1]$};\\
e_{1}+\cdots+e_{n-1}, & \text{if $i=n$}.
\end{cases}\notag
\end{align}
To simplify notation, for each $i\in [\ell,n]$ and $1\leq k \leq \infty$, we take 
\begin{align}
\sfe_{i}:=(-)f_{i},\quad \sfl_{i}:=-\otimes_{f_{i}Af_{i}}f_{i}A, \quad \sfr_{i}:=\Hom_{f_{i}Af_{i}}(Af_{i},-),\notag
\end{align}
\begin{align}
\calP^{\leq k}_{i}:=\calP^{\leq k}_{f_{i}} \quad\text{and}\quad \calI^{\leq k}_{i}:=\calI^{\leq k}_{f_{i}}.\notag
\end{align}
Note that $A/Af_{i}A$ and $f_{i}Af_{i}$ are also (quadratic linear) Nakayama algebras.
The idempotent $f_{i}$ also satisfies the following property.

\begin{lemma}\label{lem:AfA-projective}
Let $i\in [\ell,n]$.
Then we have $Af_{i}A\cong f_{i}A\oplus P(n)^{\oplus d_{i}}$, where $d_{i}$ is given by $d_{i}:=\mathrm{dim}_{\Bbbk}(1-f_{i})Ae_{n}$ if $i\neq n$, and $d_{n}=0$. 
In particular, $Af_{i}$ is projective as an $f_i Af_i$-module
and $Af_{i}A$ is stratifying. 
\end{lemma}

\begin{proof}
We have a decomposition $Af_{i}A\cong f_{i}A\oplus (1-f_{i})Af_{i}A$ as right $A$-modules.
If $i=n$, then we have $(1-f_{n})Af_{n}A=e_{n}A(1-e_{n})A=0$. 
Assume that $i\neq n$.
By a direct calculation using $e_{m}Ae_{k}=0$ for all $k<m$, we have $(1-f_{i})Af_{i}A=(1-f_{i})Ae_{n}A$, and hence 
\begin{align}
\Hom_{A}(e_{j}A, (1-f_{i})Ae_{n}A)\cong
\begin{cases}
0&(j\neq n)\\
(1-f_{i})Ae_{n}&(j=n).
\end{cases}\notag
\end{align}
Thus $(1-f_{i})Af_{i}A\cong P(n)^{\oplus \dim_{\Bbbk}(1-f_{i})Ae_{n}}$.
The proof is complete.
\end{proof}

Define a subset $\calT_{i}$ of $\tilt A$ as  
\begin{align}
\calT_{i}:=\{ T\in \tilt A \mid P(i)\in \add T,\ R_{i}(T)\in \calI^{\leq 1}_{i} \}\notag
\end{align}
for each $i\in [\ell,n-1]$ and $\calT_{n}:=\tilt_{P(n)}A=\tilt_{S(n)}A$.
Note that $\calI_{i}^{\leq 1}=\calI_{i}^{\leq \infty}$. Indeed, since $Af_{i}$ is projective, we have $\mod^{\rex}(f_{i} Af_{i})=\mod(f_{i}Af_{i})$.
By Lemma \ref{lem:aus-apt}(2), we obtain $\calI_{i}^{\leq 1}=\calI_{i}^{\leq \infty}$.

Applying Proposition \ref{prop:main-bijections} to a quadratic linear Nakayama algebra, we have the following theorems.

\begin{theorem}\label{thm:glue-bij-n}
There exist poset isomorphisms 
\begin{align}
\xymatrix{\calT_{n}\ar@<1mm>[r]^-{\varphi_{n}}&\tilt (A/Af_{n}A)\times \tilt^{\lex}(f_{n}Af_{n})\ar@<1mm>[l]^-{\psi_{n}}},\notag
\end{align}
where $\varphi_{n}(T):=(K_{n}(T), \sfe_{n}(R_{n}(T)))$ and $\psi_{n}(U,V):=U\oplus \sfl_{n}(V)$.
In particular, $\calT_{n}\cong \tilt^{\lex}(f_{n}Af_{n})$.
Moreover, we have the following equalities.
\begin{itemize}
\item[(a)] If $n-1\in \relvx(A)$, then $\tilt^{\lex}(f_nAf_n) =\tilt_{\sfe_{n}(S(n-1))}(f_{n}Af_{n})$.
\item[(b)] If $n-1\notin \relvx(A)$, then $\tilt^{\lex}(f_nAf_n) =\tilt(f_{n}Af_{n})$.
\end{itemize}
\end{theorem}

\begin{proof}
By the definition of $f_{n}$, the $A$-module $A/Af_{n}A$ is semisimple and projective, and $\Ext_{A}^{>0}(U,-)=0$ for all $U\in \mod(A/Af_{n}A)$.
By Proposition \ref{prop:main-bijections}(1), we have the desired poset isomorphisms. 

We show (a).
Assume that $n-1\in \relvx(A)$.
Let $V\in \tilt(f_{n}Af_{n})$.
We show that $V\oplus \sfe_{n}(S(n-1))$ is also tilting (or equivalently
$\Ext_{f_nAf_n}^{>0}(V\oplus \sfe_{n}(S(n-1)), V\oplus \sfe_{n}(S(n-1)))=0$) if and only if $\Tor_{>0}^{f_{n}Af_{n}}(V,f_{n}A)=0$.
Since $V$ is tilting and $\sfe_{n}(S(n-1))\cong \sfe_{n}(P(n-1))$ is projective, we have 
\begin{align}
\Ext_{f_{n}Af_{n}}^{k}(V\oplus \sfe_{n}(S(n-1)), V\oplus \sfe_{n}(S(n-1)))\cong \Ext_{f_{n}Af_{n}}^{k}(V,\sfe_{n}(S(n-1))).\notag
\end{align}
By \cite[Chapter VI, Section 5]{CE56}, we obtain
\begin{align}
\Tor_{k}^{f_{n}Af_{n}}(V,f_{n}A)
&\cong \Tor_{k}^{f_{n}Af_{n}}(V,f_{n}Af_{n})\oplus \Tor_{k}^{f_{n}Af_{n}}(V,f_{n}Ae_{n})\notag\\
&\cong \Tor_{k}^{f_{n}Af_{n}}(V,f_{n}Ae_{n})\cong \kD\Ext_{f_{n}Af_{n}}^{k}(V, \kD(f_{n}Ae_{n})) \notag\\
&\cong \kD\Ext_{f_{n}Af_{n}}^{k}(V, \sfe_{n}(S(n-1))).\notag
\end{align}
Thus we have $\tilt^{\lex}(f_nAf_n) =\tilt_{\sfe_{n}(S(n-1))}(f_{n}Af_{n})$.

We show (b). If $n-1\notin \relvx(A)$, then $\idim S(n)=1$. Thus the assertion follows from Corollary \ref{cor:application_bijection-5.1}(1).
\end{proof}

\begin{theorem}\label{thm:glue-bij-except-n}
For each $i\in [\ell,n-1]$, there exist poset isomorphisms 
\begin{align}
\xymatrix{\calT_{i}\ar@<1mm>[r]^-{\varphi_{i}}&\tilt (A/Af_{i}A)\times \tilt_{\sfe_{i}(S(n))}(f_{i}Af_{i})\ar@<1mm>[l]^-{\psi_{i}}},\notag
\end{align}
where $\varphi_{i}(T):=(K_{i}(T), \sfe_{i}(R_{i}(T)))$ and $\psi_{i}(U,V):=U\oplus \sfr_{i}(V)$. In particular,  we have
\begin{align*}
\calT_{n-1}&\cong \tilt_{\sfe_{n-1}(S(n))}(f_{n-1}Af_{n-1}) \;\;\;\text{ if } \; \ell\neq n-1,\\
\text{ and }\qquad  \calT_{\ell}&\cong \tilt(A/Ae_{n}A).
\end{align*}
Moreover, we have the following equalities.
\begin{itemize}
\item[(a)] If $n-1\in  \relvx(A)$ and $i=\ell =n-1$, then $\calT_{\ell}=\tilt^{S(n)}A:=\{ T\in \tilt A \mid S(n)\notin \add T\}$.
\item[(b)] If $n-1\notin  \relvx(A)$, then $\calT_{i}=\tilt_{P(i)\oplus P(i)/S(n)}A$.
\end{itemize}
\end{theorem}

To prove Theorem \ref{thm:glue-bij-except-n}, we need the following lemmas. 

\begin{lemma}\label{lem:nakayama-combs}
Let $M$ be an $A$-module with $\Ext_{A}^{1}(M,M)=0$.
Then the following statements hold for each $\ell\leq i<j\leq n$.
\begin{itemize}
\item[(1)] $\Ext_{A}^{1}(P(i)/P(j), P(k))\neq 0$ for all $k\in[i+1,j]$.
\item[(2)] If $P(i)/P(j)\in \add M$, then we have $P(k)\notin \add M$ for all $k\in [i+1, j]$.
\item[(3)] If $P(i), P(j+1)\in \add M$ (with the convention of $P(n+1):=0$), $j\neq \ell+1$, and $M$ is tilting, then the converse of (2) also holds.
\end{itemize}
\end{lemma}

\begin{proof}
(1) Let $k\in [i+1,j]$.
Since there exists a non-split short exact sequence
\begin{align}
0\rightarrow P(k) \rightarrow P(i)\oplus P(k)/P(j) \rightarrow P(i)/P(j) \rightarrow 0, \notag
\end{align}
we have $\Ext_{A}^{1}(P(i)/P(j), P(k))\neq 0$.

(2) This is immediate from (1).

(3) Assume that $P(k)\notin \add M$ for all $k\in [i+1,j]$.
We show that 
\begin{align}\label{eq:proof-ext}
\Ext_{A}^{>0}(M\oplus P(i)/P(j), M\oplus P(i)/P(j))=0.
\end{align}
If it is shown, then $M\oplus P(i)/P(j)$ is a tilting $A$-module. 
By Proposition \ref{prop:property-tilting}(1), we have $|M|=|A|=|M\oplus P(i)/P(j)|$. 
Therefore $P(i)/P(j)\in \add M$.
In the following, we show the equality \eqref{eq:proof-ext}.
Since $M$ is tilting, $\Ext_{A}^{>0}(M,M)=0$ holds.
By $\pdim_{A} P(i)/P(j)=1$, we have $\Ext_{A}^{>1}(P(i)/P(j), M\oplus P(i)/P(j))=0$.
Moreover, by the Auslander--Reiten formula (see \cite[Corollary IV.2.14(1)]{ASS06}), we obtain
\begin{align}
\Ext_{A}^{1}(P(i)/P(j), M\oplus P(i)/P(j))
&\cong \kD\Hom_{A}(M\oplus P(i)/P(j), \tau(P(i)/P(j))) \notag\\
&\cong \kD\Hom_{A}(M\oplus P(i)/P(j), P(i+1)/P(j+1)) \notag\\
&\cong \kD\Hom_{A}(M, P(i+1)/P(j+1)),\notag
\end{align}
where $\tau$ is the Auslander--Reiten translate.
Suppose that $\Hom_{A}(M,P(i+1)/P(j+1))\neq 0$, that is, $\Hom_{A}(X,P(i+1)/P(j+1))\neq 0$ for some indecomposable $X$.
Then our assumption yields $X\cong P(m)/P(m')$ for $m\in [i+1, j]$ and $m'\in [j+1,n+1]$.
By (1), we have $\Ext_{A}^{1}(P(m)/P(m'), P(j+1))\neq 0$, a contradiction.
Hence we have $\Ext_{A}^{>0}(P(i)/P(j), M\oplus P(i)/P(j))=0$.
We show that $\Ext_{A}^{>0}(M,P(i)/P(j))=0$.
Applying $\Hom_{A}(M,-)$ to $0\rightarrow P(j)\rightarrow P(i)\rightarrow P(i)/P(j)\rightarrow 0$ yields an exact sequence 
\begin{align}
\Ext_{A}^{p}(M,P(i)) \rightarrow \Ext_{A}^{p}(M,P(i)/P(j)) \rightarrow \Ext_{A}^{p+1}(M,P(j)) \notag
\end{align}
for all $p>0$.
Since $P(i)\in \add M$ and $M$ is tilting, we have $\Ext_{A}^{p}(M,P(i))=0$ for all $p>0$.  
By $j\neq \ell+1$, the injective dimension of $P(j)$ is one.
Thus $\Ext_{A}^{p+1}(M,P(j))=0$ for all $p>0$.
Hence we have the desired property.
The proof is complete.
\end{proof}

\begin{lemma}\label{lem:tech-Pinf}
Fix $i\in [\ell+1,n-1]$.
Let $V$ be an $(f_{i}Af_{i})$-module with $\Ext_{f_{i}Af_{i}}^{\geq 0}(V,\sfe_{i}(S(n)))=0$. 
Then we have $\sfr_{i}(V)\in \calP^{\leq \infty}_{i}$.
In particular, $\Ext_{A}^{>0}(\sfr_{i}(V),U)=0$ holds for each $U\in \mod(A/Af_{i}A)$.
\end{lemma}

\begin{proof}
Take a minimal projective resolution in $\mod(f_{i}Af_{i})$
\begin{align}
0\rightarrow P_{m}\rightarrow P_{m-1}\rightarrow \cdots \rightarrow P_{0}\rightarrow V \rightarrow 0.\notag
\end{align}
By Lemma \ref{lem:AfA-projective}, the functor $\sfr_{i}$ is exact.
Applying $\sfr_{i}$ to the sequence above yields an exact sequence 
\begin{align}
0\rightarrow \sfr_{i}(P_{m})\rightarrow \sfr_{i}(P_{m-1})\rightarrow \cdots \rightarrow \sfr_{i}(P_{0})\rightarrow \sfr_{i}(V)\rightarrow 0. \notag 
\end{align}
It suffices to show that $\sfr_{i}(P_{j})\in \add(f_{i}A)$ for all $j\in [0,m]$ to finish the proof.
Since $\Ext_{f_{i}Af_{i}}^{\geq 0}(V, \sfe_{i}(S(n)))=0$ holds, it follows from Lemma \ref{lem:apt-pfin} that $P_{j}\in \add((f_{i}-e_{n})Af_{i})$ for all $j\in [0,m]$.
We show $P(h)\in \calI^{\leq 1}_{i}$ for all $h\in [1, i-1]$.
If $h\leq\ell$, then this is clear.
Assume $h\in [\ell+1,i-1]$.
Take a minimal injective presentation $0\rightarrow P(h)\rightarrow I(n)\rightarrow I(h-1)$.
This implies $P(h)\in \calI^{\leq 1}_{i}$. 
By Lemma \ref{lem:aus-apt}(2), we have $\sfr_{i}(P_{j})\in \add(f_{i}A)$. Hence $\sfr_{i}(V)\in \calP^{\leq \infty}_{i}$.
\end{proof}

Now, we are ready to prove Theorem \ref{thm:glue-bij-except-n}. 

\begin{proof}[Proof of Theorem \ref{thm:glue-bij-except-n}]
By Lemma \ref{lem:AfA-projective}, $Af_{i}$ is a projective $f_{i}Af_{i}$-module. Thus $\cotilt(f_{i}Af_{i})=\cotilt^{\rex}(f_{i}Af_{i})$. 
Note that linear Nakayama algebras are representation-finite with finite global dimension, and this class of algebras is preserved under idempotent truncation $eAe$ and quotient $A/AeA$.  
This means that for any $B\in\{A/Af_{i}A, f_{i}Af_{i}\mid i\in[\ell,n]\}$, we have $\tilt B=\cotilt B$, and $B$ satisfies the rank-(co)tilting-completion property.  It also follows that $\idim_{A}(\kD(A/Af_{i} A))<\infty$.  

Since $Af_{i}A$ is stratifying by Lemma \ref{lem:AfA-projective}, we now obtain the poset isomorphisms in Proposition \ref{prop:main-bijections}(2).
Now, $\calT_i$ is a subset of the upper set of \eqref{eq:cotilt-glue-bijn} given by intersecting with $\tilt_{P(i)}(A)$.  Since $\sfe_iP(i)=\sfe_iS(n)$, we have $P(i)\in \add(R_{\e}(T))$ and the poset isomorphisms restrict to the following mutually inverse poset isomorphisms
\begin{align}
\xymatrix{
\calT_{i} \ar@<1mm>[r]^-{\varphi_{i}} & \{ (U,V)\in \tilt (A/Af_{i}A)\times \tilt_{\sfe_{i}S(n)}(f_{i}Af_{i})\mid \Ext_{A}^{>0}(\sfr_{i}(V),U)=0\} \ar@<1mm>[l]^-{\psi_{i}}
}.\notag
\end{align}

To complete the proof of the first part of the claim, it is enough to show that $\Ext_{A}^{>0}(\sfr_{i}(V),U)=0$ for each $V\in \tilt_{\sfe_{i}S(n)}(f_{i}Af_{i})$ and $U\in \mod(A/Af_{i}A)$.
If $i=\ell$, then $\tilt(f_{\ell}Af_{\ell})=\{ e_n Ae_n \}$, and hence we have $\sfr_{\ell}(e_{n} A e_{n}) \cong I(n) = P(\ell)$. Thus the desire Ext-vanishing always holds and the claim follows immediately.
Assume now that $i\in [\ell+1, n-1]$. Decompose $V$ as $V=\sfe_{i}(S(n))\oplus V'$. Since $V$ is tilting and $\sfe_{i}(S(n))$ is simple projective, $\Ext_{f_{i}Af_{i}}^{\geq 0}(V', \sfe_{i}(S(n)))=0$ holds. 
It then follows from Lemma \ref{lem:tech-Pinf} that $\Ext_{A}^{>0}(\sfr_{i}(V'), U)=0$. 
On the other hand, by $\sfr_{i}\sfe_{i}(S(n))\cong P(i)$, we have $\Ext_{A}^{>0}(\sfr_{i}\sfe_{i}(S(n)), U)=0$.
This completes the proof for the first part of the claim.

We show (a).
Assume that $n-1\in \relvx(A)$ and $i=\ell$.
Let $T$ be a tilting $A$-module.
We can decompose $T$ as $T=K_{\ell}(T)\oplus R_{\ell}(T)$.
We have
\begin{align}
S(n)\in \add T \Leftrightarrow S(n)\in \add(R_{\ell}(T)) \Leftrightarrow R_{\ell}(T)\notin \calI^{\leq 1}_{\ell},\notag
\end{align}
where the first equivalence follows from $S(n)f_\ell = S(n)e_n\neq 0$, and the second equivalence follows from $\calI_\ell^{\le 1}=\add(I(n))$. 
Hence, we have $\calT_{\ell}=\tilt^{S(n)}A$.

We show (b). Assume that $n-1\notin \relvx(A).$
Let $T\in \calT_{i}$. 
For each $j\in [i+1,n]$, since $\sfe_{i}P(j)\neq 0$ and $P(j)\notin\calI^{\leq 1}_{i}$ hold, we have $P(j)\notin \add(K_{i}(T))$ and $P(j)\notin \add(R_{i}(T))$ respectively.
Thus $P(j)\notin\add T$.
By Lemma \ref{lem:nakayama-combs}(3), we obtain $P(i)/S(n)\in \add T$.
Conversely, let $T\in \tilt_{P(i)\oplus P(i)/S(n)}A$. 
It is enough to show $R_{i}(T)\in \calI^{\leq 1}_{i}$. 
If $i=\ell$, then $R_{\ell}(T)\in \add(\bigoplus_{j=\ell}^{n}P(j))$.
By Lemma \ref{lem:nakayama-combs}(2), $P(\ell)/S(n)\in \add T$ implies $P(\ell+1), \ldots, P(n)\notin \add T$.
Since $R_{\ell}(T)$ is basic, we have $R_{\ell}(T)\cong P(\ell)\in \calI^{\leq 1}_{\ell}$.
Assume that $i\in[\ell+1,n-1]$.
By Lemma \ref{lem:apt-pfin}, it is enough to show that $\Ext_{A}^{k}(S(j),R_{i}(T))=0$ for all $k\in \{ 0, 1\}$ and $j\in [i,n-1]$.
Suppose that $\Hom_{A}(S(j),R_{i}(T))\neq 0$ for some $j\in [i,n-1]$.
Then the socle of $R_{i}(T)$ contains $S(j)$ as a direct summand.
By $\sfe_{i}(R_{i}(T))\neq 0$, there exists $k\in [\ell,i-1]$ such that $P(k)/P(j+1)\in \add(R_{i}(T))$.
By Lemma \ref{lem:nakayama-combs}(1), we have $\Ext_{A}^{1}(P(k)/P(j+1),P(i))\neq 0$.
This contradicts $P(k)/P(j+1), P(i)\in \add T$.
Suppose that $\Ext_{A}^{1}(S(j),R_{i}(T))\neq 0$ for some $j\in [i,n-1]$.
Since $\pdim_{A}S(j)\leq 1$ holds, the Auslander--Reiten formula yields
\begin{align}
0\neq \Ext_{A}^{1}(S(j),R_{i}(T))\cong \mathbb{D}\Hom_{A}(R_{i}(T),\tau S(j))\cong \mathbb{D}\Hom_{A}(R_{i}(T),S(j+1)). \notag
\end{align}
Then there exists $X\in \add(R_{i}(T))$ such that $\top X\cong S(j+1)$.
By $\sfe_{i}(X)\neq 0$, we have $X(f_{i}-e_{n})=0$, and hence $Xe_{n}\neq 0$.
Thus $X$ must be $P(j+1)$. 
On the other hand, Lemma \ref{lem:nakayama-combs}(2) induces $P(j+1)\notin \add T$.
This is a contradiction.
Thus $R_{i}(T)\in \calI^{\leq 1}_{i}$.
\end{proof}

Using Theorems \ref{thm:glue-bij-n} and \ref{thm:glue-bij-except-n}, we have the following result.

\begin{theorem}\label{thm:decomp}
We have $\tilt A=\bigsqcup_{i=\ell}^{n}\calT_{i}$.
\end{theorem}

\begin{proof}
Assume that $n-1\in \relvx(A)$, that is, $\ell=n-1$.
By Theorem \ref{thm:glue-bij-except-n}(a), we have 
\begin{align}
\tilt A =\tilt^{S(n)}A\sqcup \tilt_{S(n)}A=\calT_{n-1}\sqcup\calT_{n}.\notag
\end{align}

Assume that $n-1\notin\relvx(A)$. 
First we show that $\tilt A =\bigcup_{i=\ell}^{n}\calT_{i}$.
Let $T$ be a basic tilting $A$-module.
Since $P(\ell)$ is projective-injective, we can take the maximum integer $i\in [\ell,n]$ satisfying $P(i)\in \add T$.
If $i=n$, then we have $P(n)\in \add T$, and hence $T\in \calT_{n}$.
If $i\in [\ell,n-1]$, then it follows from Lemma \ref{lem:nakayama-combs}(3) that $P(i)/S(n)\in \add T$.
Thus we have $T\in \calT_{i}$.
Next we show $\calT_{j}\cap\calT_{k}=\emptyset$ for all $j\neq k\in [\ell,n]$.
Suppose $\calT_{j}\cap \calT_{k}\neq \emptyset$ for some $j>k$.
Let $T\in \calT_{j}\cap \calT_{k}$.
Then we have $P(k)/S(n)\in \add T$.
By Lemma \ref{lem:nakayama-combs}(2), we obtain $P(j)\notin \add T$, a contradiction. 
Hence we have the assertion.
\end{proof}

As an application, we can inductively obtain the number of tilting modules over every quadratic linear Nakayama algebra. 
Let $A$ be an arbitrary quadratic linear Nakayama algebra whose Gabriel quiver is
\begin{align}
1\rightarrow 2\rightarrow \cdots \rightarrow j. \notag
\end{align}
For integers $k\geq -1$ and $m\geq 1$, let $A(k,m)$ be the quadratic linear Nakayama algebra whose Gabriel quiver is 
\begin{align}
1\rightarrow 2\rightarrow \cdots \rightarrow j\xrightarrow{\alpha_{j}} j+1 \xrightarrow{\alpha_{j+1}} \cdots \rightarrow j+k \xrightarrow{\alpha_{j+k}}  j+k+1 \xrightarrow{\alpha_{j+k+1}} j+k+2 \rightarrow \cdots \rightarrow j+k+m+1, \notag
\end{align}
and the relations are given by those of $A$ and $\{ \alpha_{p}\alpha_{p+1}\mid p\in [j,j+k] \}$.
Note that $\alpha_{j-1}\alpha_{j}$ is not contained in the relation. 
All quadratic linear Nakayama algebras can be obtained by this form. Namely, if $B$ is a quadratic linear Nakayama algebra, then there exists an algebra isomorphism $B\cong A(k,m)$ such that $A$ is a quadratic linear Nakayama algebra with $|A|<|B|$ and integers $k\geq -1$ and $m\geq 1$.

\begin{corollary}\label{cor:count-by-tilting}
Keep the notation above.
Then the following statements hold.
\begin{itemize}
\item[(1)] For each integer $k\geq 0$, we have 
\begin{itemize}
\item[(a)] $|\tilt_{S(n)}(A(k,1))|=|\tilt A|$, where $n:=j+k+2$, 
\item[(b)] $|\tilt(A(k,1))|=|\tilt(A(k-1,1))|+|\tilt A|
=|\tilt(A(-1,1))|+(k+1)|\tilt A|$.
\end{itemize}   
\item[(2)] For each integer $k\geq 0$ and $m\geq 1$, we have
\begin{align}
|\tilt(A(k,m))|
&=C_{m}|\tilt(A(k,1))|+(C_{m+1}-2C_{m})|\tilt A|\notag\\
&=C_{m}|\tilt(A(-1,1))|+(C_{m+1}+(k-1)C_{m})|\tilt A|,\notag
\end{align}
where $C_{i}:=\frac{1}{i+1}\binom{2i}{i}$ is the $i$-th Catalan number.
\end{itemize}
\end{corollary}

\begin{proof}
(1) We show (a). 
By repeating Theorem \ref{thm:glue-bij-n}(a), we have
\begin{align}
\tilt_{S(n)}(A(k,1))\cong \tilt_{S(n-1)}(A(k-1,1))\cong\cdots\cong\tilt_{S(j+1)}(A(-1,1)).\notag
\end{align}
Since $\alpha_{j-1}\alpha_{j}\neq 0$ in $A(-1,1)$, it follows from Theorem \ref{thm:glue-bij-n}(b) that $\tilt_{S(j+1)}(A(-1,1))\cong \tilt A$.
Thus we have the desired equality.
We show the first equality in (b).
By (a) and Theorem \ref{thm:glue-bij-except-n}(a), we have 
\begin{align}
\tilt(A(k,1))\cong \tilt_{S(n)}(A(k,1))\sqcup \tilt^{S(n)}(A(k,1))
\cong \tilt A \sqcup \tilt(A(k-1,1)).\notag
\end{align}
By repeating this process, we obtain 
\begin{align}
|\tilt(A(k,1))|
=|\tilt(A(k-2,1))|+2|\tilt A|
=\cdots = |\tilt(A(-1,1))|+(k+1)|\tilt A|. \notag
\end{align}

(2) Note that $\ell=j+k+1$ and $n=\ell+m$. 
If $m=1$, then the assertion follows from (1-b).
In the following, we assume that $m\geq 2$.
By Theorem \ref{thm:decomp}, we have $\tilt(A(k,m))=\bigsqcup_{i=\ell}^{\ell+m}\calT_{i}$.
We show the first equality by induction on $m$. 
Assume $m=2$.
By Theorems \ref{thm:glue-bij-n} and \ref{thm:glue-bij-except-n}, we have 
$\calT_{\ell}\cong \tilt(A(k,1))$, $\calT_{\ell+1}\cong \tilt_{S(\ell+1)}(A(k,1))$, and $\calT_{\ell+2}\cong \tilt(A(k,1))$.
By (1-a), $\calT_{\ell+1}\cong \tilt A$.
Thus we have 
\begin{align}
|\tilt(A(k,2))|
&=2|\tilt(A(k,1))|+|\tilt A|\notag\\
&=C_{2}|\tilt A(k,1)|+(C_{3}-2C_{2})|\tilt A|. \notag
\end{align}

Assume $m\geq 3$. Let $b_{p}:=|\tilt(A(k,p))|$ for $p\geq 1$ and $b_{0}:=|\tilt A|$.
By Theorems \ref{thm:glue-bij-n} and \ref{thm:glue-bij-except-n}, and (1-a), we have
\begin{align}
&|\calT_{\ell}|=|\tilt(A(k,m-1))|=b_{m-1},\notag\\
&|\calT_{\ell+1}|=|\tilt (\Bbbk\Vec{\bbA}_{m-1})|\times|\tilt A|=C_{m-1}b_{0},\notag\\
&|\calT_{\ell+p}|=|\tilt(\Bbbk\Vec{\bbA}_{m-p})|\times |\tilt(A(k, p-1))|=C_{m-p}b_{p-1}\ \ (p\in[2,m-1]),\notag\\
&|\calT_{\ell+m}|=|\tilt(A(k,m-1))|=b_{m-1}.\notag
\end{align}
Thus
\begin{align}
b_{m}
&=|\tilt(A(k,m))|=\sum_{p=0}^{m}|\calT_{\ell+p}|=b_{m-1}+\sum_{p=1}^{m}C_{m-p}b_{p-1}\notag\\
&=b_{m-1}+C_{m-1}b_{0}+C_{m-2}b_{1}+\sum_{p=3}^{m}C_{m-p}b_{p-1}.\notag
\end{align}
Moreover, by the induction hypothesis, we obtain
\begin{align}
\displaystyle b_{m}
&=\{ C_{m-1}b_{1}+(C_{m}-2C_{m-1})b_{0}\}+C_{m-1}b_{0}+C_{m-2}b_{1}\notag\\
&\hspace{10mm}+\sum_{p=3}^{m}C_{m-p}\{ C_{p-1}b_{1}+(C_{p}-2C_{p-1})b_{0}\}\notag\\
&=(C_{m-1}+C_{m-2}+\sum_{p=3}^{m}C_{m-p}C_{p-1})b_{1}\notag\\
&\hspace{10mm}+(C_{m}+C_{m-1}+\sum_{p=3}^{m}C_{m-p}C_{p})b_{0}-2(C_{m-1}+\sum_{p=3}^{m}C_{m-p}C_{p-1})b_{0}\notag\\
&=(\sum_{p=1}^{m}C_{m-p}C_{p-1})b_{1}\notag\\
&\hspace{10mm}+\{(\sum_{p=0}^{m}C_{m-p}C_{p})-C_{m-2}C_{2}\}b_{0}-2\{ (\sum_{p=1}^{m}C_{m-p}C_{p-1})-C_{m-2}C_{1}\}b_{0}\notag\\
&=C_{m}b_{1}+C_{m+1}b_{0}-2C_{m-2}b_{0}-2C_{m}b_{0}+2C_{m-2}b_{0} \notag\\
&=C_{m}b_{1}+(C_{m+1}-2C_{m})b_{0},\notag
\end{align}
where $C_{0}=C_{1}=1$ and $C_{i}=\sum_{p=1}^{i}C_{i-p}C_{p-1}$.
\end{proof}

Finally, we characterise each $\calT_{i} \subset \tilt A = \IStilt A$ in terms of the associated quasi-hereditary structures.

\begin{proposition}\label{prop:qln-costd}
Let $A$ be a quadratic linear Nakayama algebra, $T$ be a characteristic tilting module corresponding to $\unlhd=\unlhd_{T}$, and $\costd(n)$ be the arising costandard module associated to the sink $n$ in the quiver of $A$. 
Then the following statements hold. 
\begin{itemize}
\item[(1)] We have $\costd(n)\in\add T$.
\item[(2)] Assume $n-1\notin\relvx(A)$. If $T\in\calT_{i}$ for some $i\in[\ell, n-1]$, then $j\unlhd n$ for all $j\in[i,n-1]$. 
\item[(3)] For $i\in[\ell,n]$, $T\in\calT_{i}$ if and only if $\costd(n)=P(i)$.
\end{itemize}
\end{proposition}

\begin{proof}
(1) Since $n$ is sink, there exists $i$ such that $\costd(n)=P(i)\in\filt(\std)\cap\filt(\costd)=\add T$, where the last equality follows from Definition-Proposition \ref{prop:ringel}. 
Hence the assertion follows. 

(2) Let $T\in\calT_{i}$. 
Then $P(i)\oplus P(i)/S(n)$ is a direct summand of $T$ by Theorem \ref{thm:glue-bij-except-n}(b). 
By Lemma \ref{lem:nakayama-combs}(2), $P(i+1), \ldots, P(n)\notin \add T$. This implies that $P(i)$ is the indecomposable direct summand of $T$ with minimal length among those having $S(n)$ as a composition factor. 
This condition means that $i, i+1, \ldots, n-1 \unlhd n$. 

(3) First consider the case when $i=n$. Assume that $T\in \calT_{n}$. Since $S(n)$ is a direct summand of $T$, we may assume that $n$ is minimal with respect to $\unlhd$. Hence $\costd(n)=P(n)$. By (1), the converse holds. 

Consider now the case when $n-1\notin\relvx(A)$ and $i\in[\ell, n-1]$. 
Assume $T\in\calT_{i}$. If $i=\ell$, then the claim follows from (2). Assume $i\in[\ell+1, n-1]$. Since $P(i)\in \add T$, an indecomposable direct summand of $T$ having $S(i-1)$ as a composition factor is projective by Lemma \ref{lem:nakayama-combs}(1). Thus we get that $T(i-1)\cong P(k)$ for some $k\in [\ell, i-1]$. 
This implies $n\unlhd i-1$, and hence, we have $\costd(n)=P(i)$ by (2). 
Conversely, we assume $\costd(n)=P(i)$. By (1), $P(i)$ is a direct summand of $T$. Hence it is enough to show $P(i+1), \ldots, P(n)\notin \add T$ by Lemma \ref{lem:nakayama-combs}(3). 
Suppose to the contrary that $P(j)\in \add T$ for some $j\in [i+1, n]$. 
Then there exists $k\in[j,n]$ such that $P(j)=T(k)$. This means $k\unrhd n$, and this contradicts to $\costd(n)=P(i)$. 

Consider now the case when $n-1\in\relvx(A)$ and $i=n-1$. 
By Theorem \ref{thm:glue-bij-except-n}(a), $\calT_{n-1}=\tilt^{S(n)}A$. Hence, we have
\begin{align}
T\in\calT_{n-1} \Leftrightarrow T(n)\neq S(n) \Leftrightarrow \costd(n) \supsetneq S(n) \Leftrightarrow \costd(n)=I(n)=P(n-1).\notag
\end{align}
This completes the proof.
\end{proof}

\begin{corollary}\label{prop:istilt-qhs}
Let $A$ be a quadratic linear Nakayama algebra.
Then the correspondence $\IStilt(A)\longleftrightarrow \qhs A$ in Theorem \ref{thm:bij-qh-IStilting} restricts to the following isomorphisms of subposets.
\begin{itemize}
\item[(1)] $\calT_{n} \cong \{ [\unlhd] \mid n\in \min \unlhd^{\rm m} \}$.
\item[(2)] If $n-1\notin \relvx(A)$, then $\calT_{i} \cong \{ [\unlhd] \mid i+j\unlhd^{\rm m} n \unlhd^{\rm m} i-1 \text{ for all }0\le j < n-i \}$ for $i\in[\ell+1,n-1]$.
\item[(3)] If $n-1\notin \relvx(A)$, then $\calT_{\ell} \cong \{[\unlhd] \mid n\in \max \unlhd^{\rm m}\}$.
\item[(4)] If $n-1\in \relvx(A)$, then $\calT_{\ell} =\calT_{n-1} \cong \{[\unlhd]\mid n\notin \min \unlhd^{\rm m}\}$.
\end{itemize}
\end{corollary}

\begin{proof}
The set on the right-hand side of each item is equivalent to saying that $\costd(n)=P(i)$ when (1) $i=n$, (2) $i\in[\ell+1,n-1]$ and $n-1\notin\relvx(A)$, (3) $i=\ell$ and $n-1\notin \relvx(A)$, (4) $i=\ell$ and $n-1\in\relvx(A)$.  Thus, the claim follows immediately from Proposition \ref{prop:qln-costd}(3).
\end{proof}

\section{Description of the minimal adapted orders}\label{sec:qhs description}

To describe the quasi-hereditary structures of a quadratic linear Nakayama algebra explicitly, we consider a more general setting where one glues two algebras in a specific way, and study how to glue the quasi-hereditary structures from the two quotients.

\subsection{Nodal gluing of algebras induces gluing of minimal adapted order}\label{subsec:nodal gluing}

Recall that a \defn{node} in a basic algebra $A=\Bbbk Q/I$ is a vertex $v\in Q_{0}$ such that $pq=0$ for all length 2 paths $pq$ with $t(p)=v=s(q)$.
In the case of a linear Nakayama algebra $A$, this is the same as saying that $v\in \relvx(A)$ in our notation.

Suppose $A=\Bbbk Q_{A}/I_{A}, B=\Bbbk Q_{B}/I_{B}$ are basic algebras such that $A$ has a sink vertex (in $Q_{A}$) and $B$ has a source vertex.
Fix a choice of such a sink in $A$ and source in $B$, and call it $v$.
Denote by $Q_{A} \vee_{v} Q_{B}$ the quiver obtained by gluing $Q_{A}$ with $Q_{B}$ at $v$.
Then the \defn{nodal gluing} $A\vee_{v}B$ of $A$ and $B$ at $v$ is given by $\Bbbk (Q_{A}\vee_{v}Q_{B})/I_{A,B}^{v}$, where $I_{A,B}^{v}=(I_{A}, I_{B}, pq\mid t(p)=v=s(q) \,\forall p,q\in Q_{A} \vee_{v} Q_{B})$. Clearly, $v$ is a node in $A\vee_{v}B$.

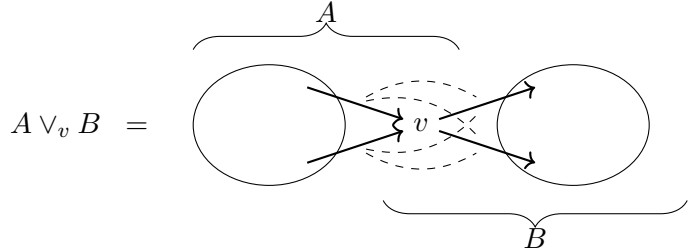
\begin{center}
\begin{tikzpicture}
\node (v) at (-2.5,0) {$A\vee_{v} B\;\;=$};
\draw  (0,0) ellipse (1 and 0.8);
\draw  (4,0) ellipse (1 and 0.8);
\node (v) at (2,0) {$v$};
   \draw [decorate,decoration={brace,amplitude=10pt}] (-1,1) -- (2.5,1);
   \draw [decorate,decoration={brace,amplitude=10pt}] (5.5,-1) -- (1.5,-1);
\draw (0.5,0.5) edge[thick,->] node[midway](tov1){} (v); 
\draw (0.5,-0.5) edge[thick,->] node[midway](tov3){} (v);
\draw (v) edge[thick,->] node[midway](fromv1){} (3.5,0.5) ;
\draw (v) edge[thick,->] node[midway](fromv3){} (3.5,-0.5);

\draw (tov1) edge[dashed,bend left] (fromv1);
\draw (tov1) edge[dashed,bend left] (fromv3);
\draw (tov3) edge[dashed,bend right] (fromv1);
\draw (tov3) edge[dashed,bend right] (fromv3);
\node at (0.75,1.5) {$A$};
\node at (3.5,-1.5) {$B$};
\end{tikzpicture}
\end{center}
We can also write the nodal gluing as the matrix algebra
\begin{align}
A\vee_{v} B = \begin{bmatrix} eAe & eAe_{v} & 0 \\
0 & \Bbbk & f_{v}Bf \\ 0 & 0 & fBf
\end{bmatrix},\notag
\end{align}
where $e:=1_{A}-e_{v}$, $f:=1_{B}-f_{v}$, and $e_{v}, f_{v}$ are the primitive idempotents corresponding to $v$ in $A$ and $B$, respectively.
We omit the subscript $v$ if it is clear from context -- for example, when $A$ has a unique sink and $B$ has a unique source.

\begin{example}
Denote by $A_{n}=\Bbbk\vec{\bbA}_{n}$ the path algebra of a linearly oriented quiver $\vec{\bbA}_n$.
Every quadratic linear Nakayama algebra is obtained by nodal gluing algebras of $A_{n}$.
For example, if $A_{n}^{!}:=A_{n}/\rad^{2}(A_{n})$, then $A_{n}^{!}=(((A_{2}\vee A_{2})\vee A_{2})\cdots )\vee A_{2}$ with $n-1$ copies of $A_{2}$ here.
\end{example}

We give one observation of partial orders of quasi-hereditary algebras.

\begin{lemma}\label{lem:arrow-comparable}
Let $(A, \unlhd)$ be a quasi-hereditary algebra.
For each $(u\rightarrow v)\in (Q_{A})_{1}$, it holds that $u$ and $v$ are comparable with respect to $\unlhd$.  
In fact, either $v\lhd^{\dec} u$ or $u\lhd^{\inc} v$ holds.
\end{lemma}

\begin{proof}
If $[\rad\std(u):S(v)]\neq 0$, then $v\lhd u$; otherwise by Lemma \ref{prop:BH-recip}, we have $(I(v):\costd(u)) = [\std(u):S(v)]=0$, hence $S(u)$ being a direct summand of $\soc (I(v)/S(v))$ implies that $[\costd(v):S(u)]\neq 0$, which gives $u\lhd v$.
\end{proof}

\begin{remark}
Note that this result applies to any quasi-hereditary partial order $\unlhd$.  In particular, for $u,v$ as in the lemma, and any equivalent $\unlhd_1, \unlhd_2\in[\unlhd]$,  $u \unlhd_1 v$ holds if and only if $u \unlhd_2 v$ holds; similarly, the same holds for $v\unlhd u$.
\end{remark}

To state the main theorem, we need to consider the following conditions.

\begin{definition}\label{def:gluing conditions}
For $([\unlhd_{A}],[\unlhd_{B}])\in \qhs A \times \qhs B$, consider the following conditions.
\begin{itemize}
\item[($\std$)] If there exists an arrow $x \rightarrow v$ in $Q_{A}$ with $x\lhd_{A}^{\rm m} v$, then $v\lhd_{B}^{\rm m}y$ holds for any arrow $v \rightarrow y$ in $Q_{B}$.
\item[($\costd$)] If there exists an arrow $v \rightarrow y$ in $Q_{B}$ with $y\lhd_{B}^{\rm m} v$, then $v\lhd_{A}^{\rm m}x$ holds for any arrow $x \rightarrow v$ in $Q_{A}$.
\end{itemize}
Denote by $\qhs(A , B; \std)$ and $\qhs(A , B; \costd)$ the subsets of $\qhs A \times \qhs B$ consisting of elements $([\unlhd_{A}], [\unlhd_{B}])$ satisfying the conditions $(\std)$ and $(\costd)$, respectively.
\end{definition}

Write $\Lambda_{C}:=(Q_{C})_{0}$ for $C\in\{A,B\}$; we will regard these as subsets of $\Lambda:=(Q_{A}\vee_{v}Q_{B})_{0}$.

\begin{theorem}\label{thm:pro-qhs-node-gluing}
Let $A \vee_{v} B$ be a nodal gluing of $A$ and $B$.
Then we have an injective map
\begin{align}
\Phi : \qhs(A \vee_{v} B) \longrightarrow \qhs A \times \qhs B, \quad [\unlhd] \mapsto ([\unlhd|_{\Lambda_{A}}], [\unlhd|_{\Lambda_{B}}]).\notag
\end{align}
Moreover, we have $\qhs(A , B; \std) = \Img \Phi = \qhs(A , B; \costd)$.
\end{theorem}

The proof of Theorem \ref{thm:pro-qhs-node-gluing} will be given after Proposition \ref{prop:Psi-qhs}.

First note that there exists a surjection from $A\vee_vB$ to $A$ (respectively, to $B$) with kernel $\langle e_i \mid i \in \Lambda_B\setminus\{v\}\rangle$ (respectively, $\langle e_i \mid i \in \Lambda_A\setminus\{v\}\rangle$).
As a result, we can regard $A$-modules and $B$-modules as $A\vee_vB$-modules, and that an $A\vee_vB$-module $X$ is an $A$-module (respectively, a $B$-module) if and only if $Xe_i=0$ for any $i\in\Lambda_B\setminus\{v\}$ (respectively, $i\in\Lambda_A\setminus\{v\}$).

Denote by $P(i), S(i)$ the projective and the simple $A\vee_v B$-modules, respectively.
For $C\in\{A, B\}$, we write $P^{C}(i)$ the projective $C$-module.
For $i\in\Lambda_{A}$, we have $P^{A}(i)=P(i)$ if $i\neq v$ and $P^{A}(v)=S(v)$.
For $i\in\Lambda_{B}$, we have $P^{B}(i)=P(i)$.
Likewise, for $C\in\{\emptyset,A,B\}$, we will use $\std^{C}(i), \costd^{C}(i)$ to denote the standard and costandard module associated to the appropriate partial order under consideration on $\Lambda_C$.

We start with the following simple observation.

\begin{lemma}\label{lem:std-AvB}
For the restrictions $\unlhd_A:=\unlhd|_{\Lambda_A}, \unlhd_B:=\unlhd|_{\Lambda_B}$ of a partial order $\unlhd$ on $\Lambda$, the associated standard modules are given by
\begin{align}
\std^{A}(i) = 
\begin{cases}
\std(i), & \text{if $i\in \Lambda_A\setminus\{v\}$};\\ 
S(v), & \text{if $i=v$},
\end{cases}
\;\; \text{ and }\;\;\std^{B}(i)=\std(i) \text{ for all }i\in\Lambda_B.\notag
\end{align}
Dually, the associated costandard modules satisfy
\begin{align}
\costd^{A}(i) = \costd(i)  \text{ for all }i\in\Lambda_A, \text{ and }\;\; \costd^{B}(i) = 
\begin{cases} \costd(i), & \text{if $i\in \Lambda_{B}\setminus\{v\}$};\\ 
S(v), & \text{if $i=v$}.
\end{cases}\notag
\end{align}
\end{lemma}

\begin{proof}
This immediately follows from the definition of $A\vee_{v}B$.
\end{proof}

The following lemma implies that $\Phi$ is a well-defined map.

\begin{proposition}\label{prop:glue-MAO}
If $(A\vee_{v} B,\unlhd)$ is a quasi-hereditary algebra, then $(A,\unlhd|_{\Lambda_{A}})$ and $(B,\unlhd|_{\Lambda_{B}})$ are also quasi-hereditary algebras. In particular, $\Phi$ is well-defined.
Moreover, these are minimal adapted orders if so is $\unlhd$.
\end{proposition}

\begin{proof}
For simplicity, we write $\unlhd_A=\unlhd|_{\Lambda_{A}}$ and $\unlhd_B=\unlhd|_{\Lambda_{B}}$.

We show that $(B,\unlhd_B)$ is quasi-hereditary.
By Lemma \ref{lem:std-AvB}, $(B,\unlhd_B)$ satisfies (qh1).
For each $i\in\Lambda_B$, let $K(i)$ be the kernel of $P^{B}(i)\rightarrow \std(i)$. By Lemma \ref{lem:std-AvB}, we have 
\begin{align}
K(i)\in \filt(\std(j)\mid j\in \Lambda_{B}, i\lhd j)=\filt(\std^{B}(\rhd i)).\notag
\end{align}
Therefore $(B,\unlhd_{B})$ satisfies (qh2).

We show that $(A,\unlhd_{A})$ is quasi-hereditary.
By Lemma \ref{lem:std-AvB}, $(A,\unlhd_{A})$ satisfies (qh1).
We show that $(A, \unlhd_{A})$ satisfies (qh2).
Clearly, $P^{A}(v)=\std^{A}(v)=S(v)$ satisfies the condition of (qh2).
Let $i\in\Lambda_{A}\setminus\{v\}$.
Consider the canonical exact sequence $0 \rightarrow K(i) \rightarrow P(i) \rightarrow \std(i) \rightarrow 0$ in $\mod (A\vee_{v} B)$ with $K(i)\in\filt(\std(\rhd i))$.
Each term of this sequence belongs to $\mod A$ since $P(i)=P^{A}(i)$.
Thus we have $K(i)\in\filt(\std(j) \mid j\in\Lambda_{A}, i \lhd j)$.
If $(K(i):\std(v))=0$, then 
\begin{align}
K(i)\in\filt(\std(j) \mid j\in\Lambda_{A}\setminus\{v\}, i \lhd j) = \filt(\std^{A}(j) \mid j\in\Lambda_{A}\setminus\{v\}, i \lhd j)\notag
\end{align}
holds.
On the other hand, assume that $(P(i):\std(v))\neq 0$. 
Since $P(i)=P^{A}(i)$, we have $\std(v)e_{j}=0$ for all $j\in\Lambda_{B}\setminus\{v\}$.
Thus $\std(v)=S(v)=\std^{A}(v)$ holds.
Then we have $K(i) \in \filt(\std^{A}(\rhd i))$.
Therefore $(A,\unlhd_{A})$ satisfies (qh2).

Finally, by Lemma \ref{lem:std-AvB} the composition factors of the standard modules do not change except $\std^{A}(v)$.
Thus for any $i, j\in\Lambda_C$ with $C\in\{A, B\}$, $i\unlhd_{C}^{\ast}j$ if, and only if, $i\unlhd^{\ast}j$ for $\ast\in \{\dec,\inc \}$.
Therefore minimal adaptedness is preserved.
\end{proof}

To help the exposition of the proof of Theorem \ref{thm:pro-qhs-node-gluing}, we introduce a few notations as follows.
For posets $(\Lambda_{A}, \unlhd_{A})$ and $(\Lambda_{B}, \unlhd_{B})$, we define a partial order $\unlhd=\unlhd_{A} \vee_{v} \unlhd_{B}$ on $\Lambda=\Lambda_{A} \cup \Lambda_{B}$ as the transitive closure of the union of the binary relations $\unlhd_{A}$ and $\unlhd_{B}$.
In order words, for $x,y\in\Lambda$, we have $x\unlhd y$ if one of (i) (ii) holds: (i) $x\unlhd_{C} y$ if $x, y \in \Lambda_{C}$ for $C\in\{A,B\}$, (ii) $x \unlhd_{C} v \unlhd_{D}y$ if $x\in\Lambda_{C}, y\in\Lambda_{D}$ for $\{C, D\}=\{A, B\}$.
Note that, by construction, we have $\unlhd|_{\Lambda_{C}}=\unlhd_{C}$ for $C\in\{A,B\}$.
Hence, the relation between the standard modules over $A,B, A\vee_{v} B$ as described in Lemma \ref{lem:std-AvB}.
This means that $(\unlhd_{A}\vee_{v}\unlhd_{B})^{\rm m} = \unlhd_{A}^{\rm m}\vee_{v} \unlhd_{B}^{\rm m}$ and $ [\unlhd_{A}\vee_{v}\unlhd_{B}]=[\unlhd_{A}^{\rm m}\vee_{v} \unlhd_{B}^{\rm m}]$ hold.

\begin{proposition}\label{prop:Psi-qhs}
For $\lozenge\in\{\std,\costd\}$ and $([\unlhd_{A}], [\unlhd_{B}])\in\qhs (A,B;\lozenge)$, we have a quasi-hereditary algebra $(A\vee_{v} B, \unlhd_{A}\vee_{v}\unlhd_{B})$.
In particular, there are well-defined maps
\begin{align}
\Psi_{\lozenge} : \qhs(A,B;\lozenge) \longrightarrow \qhs(A\vee_{v} B), \quad \Psi_{\lozenge}([\unlhd_{A}], [\unlhd_{B}]):=[\unlhd_{A}\vee_{v}\unlhd_{B}].\notag
\end{align}
such that  $\Phi\circ\Psi_{\lozenge}=\id_{\qhs(A,B;\lozenge)}$.
\end{proposition}

\begin{proof}
Without loss of generality, we may assume that $\unlhd_{C}=\unlhd_{C}^ {\rm m}$ for each $C\in\{A, B\}$.
We first show that $(A\vee_{v}B, \unlhd)$ is quasi-hereditary for $\unlhd:=\unlhd_{A}\vee_{v} \unlhd_{B}$ by checking (qh1) and (qh2); the claim for $\Psi_{\costd}$ can be shown dually by checking (qh$1^{\op}$) and (qh$2^{\op}$). 

The condition (qh1) follows directly from Lemma \ref{lem:std-AvB}.
For (qh2), if $i\in\Lambda_{B}$, then we have $\filt(\std^{B}(\rhd i)) = \filt(\std(j)\mid j\in\Lambda_{B}, i\lhd j)$, which implies the desired condition.
It remains to show that (qh2) holds for $i\in\Lambda_{A}\setminus\{v\}$.
Let $0\rightarrow K(i) \rightarrow P(i) \rightarrow \std(i) \rightarrow 0$ be an exact sequence in $\mod A$ with $K(i) \in \filt(\std^{A}(\rhd i))$.
If $(K(i): \std^{A}(v))=0$, then $K(i)\in\filt(\std^{A}(j) \mid j\in\Lambda_{A}\setminus\{v\}, i\lhd_{A} j) \subset \filt(\std(\rhd i))$ holds; thus (qh2) holds.
Assume now that $(K(i):\std^{A}(v))\neq 0$.
By Lemma \ref{prop:BH-recip}(1), we have $0\neq(P^{A}(i):\std^{A}(v))=[\costd^{A}(v):S(i)]$. Thus $\costd^{A}(v)$ cannot be simple.
This implies that there exists an arrow $(x\rightarrow v) \in (Q_{A})_{1}$ with $x\lhd v$.
Since $([\unlhd_{A}], [\unlhd_{B}])\in \qhs(A,B,\std)$, $v\lhd_{B} y$ holds for any arrow $(v \rightarrow y)\in(Q_{B})_{1}$.
This implies that $\std(v)=\std^{B}(v)=S(v)=\std^{A}(v)$.
Thus, we have $K(i) \in \filt(\std^{A}(\rhd i))=\filt(\std(j) \mid j\in\Lambda_A, i \lhd j)$, i.e. (qh2) holds for $i\in \Lambda_{A}\setminus\{v\}$.

Finally, since $(\unlhd_{A}\vee_{v}\unlhd_{B})|_{\Lambda_{C}}=\unlhd_{C}$ for $C\in\{A, B\}$ by construction, we have $\Phi\Psi_\lozenge([\unlhd_A],[\unlhd_{B}]) = \Phi([\unlhd])=([\unlhd_{A}],[\unlhd_{B}])$; this completes the proof.
\end{proof}

We are ready to prove the theorem.

\begin{proof}[Proof of Theorem \ref{thm:pro-qhs-node-gluing}]
By Lemma \ref{lem:std-AvB} and Proposition \ref{prop:glue-MAO}, $\Phi$ is a well-defined injective map. 

We show $\Img\Phi =\qhs(A, B;\std)$. 
Since $\Img\Phi \supseteq\qhs(A, B;\std)$ follows from Proposition \ref{prop:Psi-qhs}, we prove the converse inclusion.
Let $([\unlhd_{A}], [\unlhd_{B}])\in\Img\Phi$. 
Without loss of generality, we assume that $\unlhd_{C}=\unlhd_{C}^{\rm m}$ for $C\in \{ A,B\}$. 
Then there exists $[\unlhd]\in\qhs(A\vee_{v}B)$ such that $([\unlhd_{A}], [\unlhd_{B}])=\Phi([\unlhd])=([\unlhd|_{\Lambda_{A}}],[\unlhd|_{\Lambda_{B}}])$.
Assume that there exists an arrow $x\rightarrow v$ in $Q_{A}$ such that $x\lhd_{A}v$. 
By $[\unlhd|_{\Lambda_{A}}]=[\unlhd_{A}]$, we have $[\std(x):S(v)]=0$. 
Since $v\in (Q_{A})_{0}$ is a sink, $S(v)$ is a direct summand of $\rad P(x)$. 
Thus the kernel $K(x)$ of a surjection $P(x) \rightarrow \std(x)$ has $S(v)$ as a direct summand. 
By $[\unlhd]\in\qhs(A\vee_{v}B)$, we obtain $K(x)\in\filt(\std)$, and hence $\std(v)=S(v)$. 
Thus $v\lhd^{\rm m}y$ holds for any arrow $v\rightarrow y$ in $Q_{B}$. 
By $[\unlhd|_{\Lambda_{B}}]=[\unlhd_{B}]$, we obtain $\Img\Phi \subseteq\qhs(A, B;\std)$. 
Similarly, we have $\Img\Phi =\qhs(A, B;\costd)$.
\end{proof}

Let us give an example before moving on.
The following observation is straightforward but useful.

\begin{lemma}\label{lem:cover-around-node}
Suppose $A\vee_v B$ is quasi-hereditary with respect to a minimal adapted order $\unlhd$.
For a covering relation $x\lhd^{\bullet} y$, if $x \in (Q_{A})_{0}$ and $y\in (Q_{B})_{0}$, then one of $x,y$ is $v$.
Similarly, the same holds for $x\in (Q_{B})_{0}$ and $y\in (Q_{A})_{0}$.
\end{lemma}
\begin{proof}
Since $\unlhd$ is a minimal adapted order, $x\lhd^\bullet y$ implies that $x\lhd^\ast y$ for some $\ast\in\{\dec,\inc\}$.
Thus the claim follows from Remark \ref{rmk:covering relation => there is path}.
\end{proof}

Consider now the case when the node $v$ has only one incoming arrow $x\rightarrow v$ in $A$ and only one out-going arrow $v\rightarrow y$ in $B$.
By combining Lemma \ref{lem:cover-around-node} with Theorem \ref{thm:pro-qhs-node-gluing}, we have the following observation: for any $\unlhd \in \qhs(A\vee_{v}B)$, the covering relations involving $x,v, y$ can be described in one of the following three forms 
\begin{align}
{\rm(i) }\; x\lhd^{\bullet} v\lhd^{\bullet} y, \quad {\rm(ii) }\; x\rhd^{\bullet} v\lhd^{\bullet} y, \quad {\rm(iii) }\; x \rhd^{\bullet} v\rhd^{\bullet} y.\notag
\end{align}

\begin{example}\label{eg:qhsA_n^!}
Consider first the algebra $A:=\Bbbk(1\rightarrow 2)$. 
We have $\tilt A =\{A, \kD A\}$. The corresponding quasi-hereditary structures are given by $\unlhd_A = \sm{1\\\downarrow\\ 2}$ and $\unlhd_{DA} = \sm{2\\\downarrow\\1}$ as Hasse quivers.
Now consider $A_{3}^{!} = \Bbbk(1\rightarrow 2) \vee_{2} \Bbbk(2\rightarrow 3)$. 
Applying Theorems \ref{thm:bij-qh-IStilting} and \ref{thm:tilt=IStit=>gentle-LNaka} to Example \ref{eg:A_n^!}, we can write down all minimal adapted orders in $\qhs (A_3^!)$, namely,
\begin{align}
\sm{3\\\downarrow\\2\\\downarrow\\ 1}=(\sm{2\\\downarrow\\1})\vee(\sm{3\\\downarrow\\2}),\;\;
\sm{1 & & 3 \\ &\searrow\swarrow\\&2}=(\sm{1\\\downarrow\\2})\vee(\sm{3\\\downarrow\\2}),\;\;
\sm{1\\\downarrow\\2\\\downarrow\\ 3}=(\sm{1\\\downarrow\\2})\vee(\sm{2\\\downarrow\\3}).\notag
\end{align}
More generally, we have $A_{n+1}^{!} = A_{n}^{!}\vee_{n} \Bbbk(n\rightarrow n+1)$ for all $n\ge 2$, and so applying nodal gluing inductively, we get that $\qhs(A_{n+1}^{!}) = \{ [\unlhd^{(i)}] \mid 1\le i\le n+1\}$, where
\begin{align} 
\unlhd^{(i)}:=\sm{1 & & n+1 \\ \downarrow & & \downarrow \\ 2 & & n \\ \vdots & & \vdots \\  i-1 & & i+1 \\ &\searrow\;\swarrow\\&i} \text{ for any } 1\le i\le n+1.\notag
\end{align}
Note that $T_{\unlhd^{(i)}}$ is the tilting $A_{n+1}^{!}$-module containing $S(i)$ as a direct summand, and we have 
\begin{align}\label{eq:chtilt-An!}
T_{\unlhd^{(i)}}(j) = 
\begin{cases}
P(j), & \text{if $j<i$};\\
S(i), & \text{if $j=i$};
\\P(j-1), & \text{if $j>i$}.
\end{cases}
\end{align}
\end{example}

\subsection{Application: Recurrence relation on the number of quasi-hereditary structures}

Let $B$ be a bound quiver algebra with a sink $1\in Q_{0}$ in the Gabriel quiver $Q$ of $B$.
For each $k\geq 0$, let $A_{k+1}^{!}$ be the radical square zero Nakayama algebra given by the following quiver with relations 
\begin{align}
A^{!}_{k+1}=\Bbbk( 1\xrightarrow{\alpha_{1}} 2 \xrightarrow{\alpha_{2}} \cdots \xrightarrow{\alpha_{k}} k+1 ) /\langle \alpha_i\alpha_{i+1} \mid 1 \leq i < k \rangle. \notag
\end{align}
Moreover, for $m\ge 0$, let $A_{m+1}$ be the path algebra given as follows
\begin{align}
A_{m+1}=\Bbbk( k+1\rightarrow k+2 \rightarrow \cdots \rightarrow k+m+1). \notag
\end{align}
Note that any quadratic linear Nakayama algebra is isomorphic to $B\vee_{1} A_{k+1}^{!}\vee_{k+1}A_{m+1}$ for some $B$, $k$ and $m$. As an application of Theorem \ref{thm:pro-qhs-node-gluing}, we give a recurrence relation on the number of quasi-hereditary structures of this algebra.

\begin{proposition}\label{prop:recurrence-formula-qhs}
Under the notation as above, let $\calN:=\{[\unlhd]\in\qhs B \mid 1 \lhd^{\rm m} x \text{ for all }(x \rightarrow 1) \in Q_{1}\}$ and $N:=|\calN|$.
Then we have
\begin{align}
|\qhs(B\vee_{1} A_{k+1}^{!} \vee_{k +1} A_{m+1})| = |\qhs B |C_{m} + N(C_{m+1} + (k-1)C_{m}). \notag
\end{align}
Moreover, if $\idim_{B}S(1)\leq 1$, then we have $\calN\cong \qhs((1-e_1)B(1-e_1))$.
\end{proposition}

We need the following lemma before proving the proposition.

\begin{lemma}\label{lem:B2B3-number}
Under the notation as above, we have the following equations.
\begin{itemize}
\item[(1)] $|\{[\unlhd] \in \qhs(A_{k+1}^{!}) \mid 1 \lhd^{\rm m} 2 \}|=1$ and $|\{[\unlhd] \in \qhs(A_{k+1}^!) \mid 2 \lhd^{\rm m} 1 \}|=k$.
\item[(2)] $|\qhs(A_{m+1})|=C_{m+1}$ and $|\{ [\unlhd] \in \qhs(A_{m+1}) \mid k+1 \lhd^{\rm m} k+2 \}| = C_{m}$.
\end{itemize}
\end{lemma}

\begin{proof}
(1) By Example \ref{eg:qhsA_n^!}, $\qhs(A_{k+1}^{!})$ is given by the $k+1$ elements $[\unlhd^{(i)}]$ for $1\le i\le k+1$.  The claim follows immediately.

(2) It was shown in \cite[Theorem 4.7]{FKR22} that there exists a bijection between $\qhs(A_{m+1})$ and binary trees with $m+1$ nodes, which implies that $|\qhs(A_{m+1})|=C_{m+1}$.
We may assume that $\unlhd = \unlhd^{\rm m}$.
If $k+1\lhd k+2$, then we have $\std(k+1)=S(k+1)$ and $K(k+1)\cong P(k+2)$.
This implies that $\pdim S(k+1)\leq 1$. Since $\costd(k+1)=S(k+1)$, we have $k+1\in \min \unlhd$, which means that we can apply Remark \ref{rem:eAe-reduction-on-cotilt} to get that
\begin{align*}
\{[\unlhd] \in \qhs(A_{m+1}) \mid k+1 \lhd k+2 \} &=  \{[\unlhd] \in \qhs(A_{m+1}) \mid k+1 \in \min \unlhd \} \\
&= \qhs( (1-e_{k+1}) A_{m+1} (1-e_{k+1}) ) = \qhs (A_{m}).
\end{align*}
The claim now follows as $|\qhs(A_{m})|=C_{m}$.
\end{proof}

\begin{proof}[Proof of Proposition \ref{prop:recurrence-formula-qhs}]
In this proof, let $B_1:=B$, $B_2:=A_{k+1}^!$, and $B_3:=A_{m+1}$. 
Let $\Lambda_i$ be the corresponding indexing set of simple modules for all $1\le i\le 3$.
Let $A:=B_{1}\vee_{1} B_{2} \vee_{k+1} B_{3}$.
By Lemma \ref{lem:arrow-comparable}, we have a disjoint union
\begin{align}\label{eq:qhsA-sqcup-empty}
\qhs A = \{[\unlhd] \in \qhs A  \mid k +1 \lhd^{\rm m} k+2\} \sqcup \{[\unlhd] \in \qhs A  \mid k+1 \rhd^{\rm m} k+2\}.
\end{align}
We calculate the number of elements of the two sets on the right-hand side of this equation.
By Theorem \ref{thm:pro-qhs-node-gluing}, restricting to $\Lambda_{12}:=\Lambda_{1}\cup \Lambda_{2}$ and $\Lambda_{3}$ defines a bijection
\begin{align}\label{eq:qhs-A-Phi}
\Phi : \qhs A  \longrightarrow \qhs(B_{1}\vee_{1}B_{2}, B_{3}; \std).
\end{align}
This $\Phi$ restricts to the following bijection
\begin{align}
\Phi^{\lhd}: \{[\unlhd] \in \qhs A \mid k +1 \lhd^{\rm m} k+2\} \longrightarrow &\{([\unlhd_{12}], [\unlhd_3])\} \mid k+1 \lhd_{3}^{\rm m} k +2 \}\notag\\
&=\qhs(B_{1}\vee_{1}B_{2}) \times \{[\unlhd_{3}] \mid k+1 \lhd_{3}^{\rm m} k +2\}\notag
\end{align}
By Lemma \ref{lem:B2B3-number}(2), we have $|\{[\unlhd_{3}] \mid k+1 \lhd_{3}^{\rm m} k +2\}|=C_{m}$.

Assume that $k=0$.
We have $B_{1}\vee_{1} B_{2} = B_{1}$ and $A=B_{1}\vee_{1} B_{3}$.
By the bijection $\Phi^{\lhd}$ and Lemma \ref{lem:B2B3-number}(2), we have
\begin{align}\label{eq:A12k=0}
|\{[\unlhd] \in \qhs A \mid 1 \lhd^{\rm m} 2\}| = |\qhs(B_{1})| C_{m}.
\end{align}
The bijection $\Phi : \qhs A \rightarrow \qhs(B_{1}, B_{3}; \std)$ is restricted to the following bijection
\begin{align}
\{[\unlhd] \in \qhs A \mid 1 \rhd^{\rm m} 2\} \longrightarrow & \{([\unlhd_{1}], [\unlhd_3]) \mid 1 \lhd_{1}^{\rm m} x \text{ for all }(x \rightarrow 1) \in Q_{1}, \, \text{and} \, 1\rhd_{3}^{\rm m} 2\} \notag \\ 
& = \calN \times \{[\unlhd_{3}] \mid 1\rhd_{3}^{\rm m} 2\}. \notag
\end{align}
By Lemma \ref{lem:B2B3-number}(2), we have 
\begin{align}\label{eq:A21k=0}
|\{[\unlhd] \in \qhs A \mid 1 \rhd^{\rm m} 2\}| = N (C_{m+1}-C_{m}).
\end{align}
Therefore, by combining \eqref{eq:qhsA-sqcup-empty}, \eqref{eq:A12k=0} and \eqref{eq:A21k=0}, we have the desired equality.

Assume that $k>0$.
Applying Theorem \ref{thm:pro-qhs-node-gluing} to $\qhs(B_{1}\vee_{1}B_{2})$ yields a bijection 
\begin{align}\label{eq:qhs-B1B2-Phi}
\qhs(B_{1}\vee_{1}B_{2}) \longrightarrow \qhs(B_{1},B_{2};\std).
\end{align}
By Lemma \ref{lem:arrow-comparable}, the vertices $1$ and $2$ must be comparable under $\unlhd_{2}$; this implies that the right-hand set in \eqref{eq:qhs-B1B2-Phi} can be rewritten as 
\begin{align}
& \left\{([\unlhd_{1}], [\unlhd_{2}]) \mid 1 \lhd_{2}^{\rm m} 2\right\}  \sqcup \left\{([\unlhd_{1}], [\unlhd_{2}]) \mid 1\lhd_{1}^{\rm m} x\text{ for all } (x\rightarrow 1)\in Q_{1},\text{ and }1 \rhd_{2}^{\rm m} 2\right\} \notag \\
&=  \left( \qhs(B_{1}) \times \{[\unlhd_{2}] \mid 1 \lhd_{2}^{\rm m} 2 \} \right) \sqcup \left( \calN \times \{[\unlhd_{2}] \mid 1 \rhd_{2}^{\rm m} 2\} \right). \notag
\end{align}
Hence, applying Lemma \ref{lem:B2B3-number}(1) we get that
\begin{align}
|\qhs(B_{1}\vee_{1}B_{2})| = |\qhs(B_{1})| + Nk. \notag
\end{align}
Combining these with the bijection $\Phi^{\lhd}$ yields
\begin{align}\label{eq:qhsA-sqcup-left}
|\{[\unlhd] \in \qhs A \mid k +1 \lhd^{\rm m} k+2\}| = (|\qhs(B_{1})| + Nk)C_{m}.
\end{align}

On the other hand, the bijection \eqref{eq:qhs-A-Phi} restricts to the another bijection:
\begin{align}
\Phi^{\rhd}: \{[\unlhd] \in \qhs A \mid k+1 \rhd^{\rm m} k+2\} \longrightarrow &
\{([\unlhd_{12}], [\unlhd_3]) \mid k \rhd_{12}^{\rm m} k+1\text{ and }k+1 \rhd_{3}^{\rm m} k+2\}\notag \\ 
& = \{[\unlhd_{12}] \mid k \rhd_{12}^{\rm m} k+1 \} \times \{[\unlhd_3] \mid k+1 \rhd_{3}^{\rm m} k+2\}.\notag
\end{align}
From Example \ref{eg:qhsA_n^!}, we have $
\{[\unlhd_2] \mid k \rhd_{2}^{\rm m} k +1\} = \{ [\unlhd^{(k+1)}] \}$.
This implies that the bijection \eqref{eq:qhs-B1B2-Phi} restricts to the following bijection
\begin{align}
\{[\unlhd_{12}] \mid k \rhd_{12}^{\rm m} k+1 \} \longrightarrow & \{([\unlhd_1], [\unlhd_2]) \mid \mbox{ $1 \lhd_{1}^{\rm m} x$ for all $(x\to1)\in Q_{1}$ and $k \rhd_{2}^{\rm m} k+1$ }\}\notag \\
&= \calN \times \{ [\unlhd^{(k+1)}] \in\qhs(A_{k+1}^!) \}.\notag
\end{align}
Recall also from Lemma \ref{lem:B2B3-number}(2) that $|\{[\unlhd_3] \mid k+1 \rhd_{3}^{\rm m} k+2\}|=C_{m+1}-C_{m}$.
Now we can apply the above calculations to $\Phi^\rhd$ and get that
\begin{align}\label{eq:qhsA-sqcup-right}
|\{[\unlhd] \in \qhs A \mid k +1 \rhd^{\rm m} k+2\}| = &  N \times |\{ [\unlhd_2] \mid k \rhd_{2}^{\rm m} k+1 \}| \times |\{[\unlhd_3] \mid k+1 \rhd_{3}^{\rm m} k+2\}| \notag \\
= &  N(C_{m+1}-C_m).
\end{align}
Finally, by \eqref{eq:qhsA-sqcup-empty}, \eqref{eq:qhsA-sqcup-left} and \eqref{eq:qhsA-sqcup-right}, we have
\begin{align}
|\qhs A|=|\qhs(B_{1})|C_{m} + N\left(C_{m+1} + (k-1)C_{m} \right). \notag
\end{align}

It remains to show that $\calN\cong\qhs((1-e_1)B_{1}(1-e_{1}))$.
We claim that
\begin{align}
\calN:=\{[\unlhd] \in \qhs(B_{1}) \mid \mbox{ $1 \lhd^{\rm m} x$ for all $(x\rightarrow 1)\in Q_{1}$}\} = \{[\unlhd] \in \qhs(B_{1}) \mid 1 \in \min\unlhd^{\rm m} \}.\notag
\end{align}
Indeed, by Lemma \ref{lem:arrow-comparable}, the left-hand side contains the right-hand side.
Conversely, let $[\unlhd]$ be an element in the left-hand set and assume that, without loss of generality, $\unlhd=\unlhd^{\rm m}$.
Then we have $\std(1)=S(1)=\costd(1)$ for the (co)standard modules associated to $\unlhd$.
Hence, $\unlhd$ being minimal adapted order implies that $1\in \min \unlhd$.
Since $\idim_{B} S(1)\leq 1$, it follows from Corollary \ref{cor:application_bijection-5.1}(2) that $\calN \cong \qhs((1-e_{1})B(1-e_{1}))$.
This finishes the proof.
\end{proof}

\subsection{Application: Quasi-hereditary structures of quadratic linear Nakayama algebras}

Let $A$ be a quadratic linear Nakayama algebra.
Throughout this subsection, to simplify exposition, we assume that every partial order $\unlhd$ appearing in the expression `$[\unlhd]\in\qhs A$' is a minimal adapted order $\unlhd=\unlhd^{\rm m}$.
By Theorems \ref{thm:bij-qh-IStilting} and \ref{thm:tilt=IStit=>gentle-LNaka}, we have a one-to-one correspondence
\begin{align}
\tilt A  \longleftrightarrow \qhs A \quad \text{ given by } T\mapsto \unlhd_{T} \text{ with inverse } \unlhd\mapsto T_{\unlhd}. \notag
\end{align}
We give a more explicit description of this bijection and, in particular, the quasi-hereditary structures (minimal adapted orders) of $A$ in terms of gluing of binary trees by combining \cite[Theorem 4.7]{FKR22} and Theorem \ref{thm:pro-qhs-node-gluing}.

We will use Theorem \ref{thm:pro-qhs-node-gluing} with the above algorithm and Example \ref{eg:qhsA_n^!} to describe all quasi-hereditary structures of a quadratic linear Nakayama algebra $A$.
Take $n_0:=1, n_{r+1}:= |A|$, and the minimal subset of the relational vertices $\{ n_1 < n_2 < \cdots < n_r\} \subset \relvx(A)$ so that 
\begin{align}\label{eq:nodal decompose A}
A \cong B_1 \vee_{n_1} B_2 \vee_{n_2} \cdots  \vee_{n_r} B_{r+1} 
\end{align}
with $B_i\cong A_{n_i-n_{i-1}+1}^*$ for $*\in\{\emptyset,!\}$.
Note that the minimality condition means that $A_{m}^{!} \vee_{v} A_{n}^{!}$ does not appear anywhere in the decomposition above.

From the previous subsections, we knew that each $[\unlhd]\in \qhs A$ can be obtained by gluing the quasi-hereditary structures of $B_i$'s at each node.
Recall from Example \ref{eg:qhsA_n^!} that $\qhs(A_{m}^{!}) = \{ [\unlhd^{(i)}]\mid 1\le i\le m\}$, where $\unlhd^{(i)}$ denotes a `V-shaped partial order' with a unique minimum $i$.
To give an explicit description of quasi-hereditary structures of $A$, we consider the following sequences of partial orders.

\begin{definition}
In the setup of \eqref{eq:nodal decompose A}, a sequence $(\unlhd_i)_{1\le i\le r+1}$ of partial orders is \defn{$A$-admissible} if the following conditions are satisfied.
\begin{itemize}
\item[(A0)] $[\unlhd_i] \in \qhs (B_{i})$ for all $1\le i\le r+1$ and $\unlhd_{i} =\unlhd_{i}^{\rm m}$.
\item[(A1)] If $B_i \cong A_{n_i-n_{i-1}+1}^!$, then (exactly) one of the following is satisfied; see Figure \ref{fig:adm seq} for their pictorial form.
\begin{align}
{\rm (i) }\begin{cases}
n_{i-1}-1 \lhd_{i-1} n_{i-1},\\
\unlhd_{i} = \unlhd^{(n_{i-1})},\\
n_i \lhd_{i+1} n_i+1;
\end{cases} \;\; {\rm (ii) }\begin{cases}
n_{i-1}-1 \rhd_{i-1} n_{i-1},\\
\unlhd_{i} = \unlhd^{(j)}\text{ some }n_{i-1}\le j \le n_i,\\
n_i \lhd_{i+1} n_i+1;
\end{cases}\;\; {\rm (iii) }\begin{cases}
n_{i-1}-1 \rhd_{i-1} n_{i-1},\\
\unlhd_{i} = \unlhd^{(n_i)},\\
n_i \rhd_{i+1} n_i+1.
\end{cases}\notag
\end{align}
\item[(A2)] If $B_{i-1}\cong A_{n_{i-1}-n_{i-2}+1}$ and $B_i\cong A_{n_i-n_{i-1}+1}$, then $n_{i-1}-1 \lhd_{i-1} n_{i-1}$ implies that $ n_{i-1}\lhd_i n_{i-1}+1$.
\end{itemize}
\end{definition}

\begin{figure}[!htbp]
\centering
\begin{tikzpicture}[scale=.7]
\tikzset{every node/.style={fill=white,font=\small}}

%case 1
\begin{scope}[shift={(0,-0.7071)}]
\draw[rotate=45]  (0,0) ellipse (1.6 and 1.8);
\end{scope}
\begin{scope}[shift={(3.9,3.2929)}]
\draw  (0,0) ellipse (1.4 and 1.6);
\end{scope}
\node at (-0.4,0.8929) {$\unlhd_{i-1}$};\node at (3.8,4.8) {$\unlhd_{i+1}$};

\node (v4) at (0,-1.5071) {$n_{i-1}-1$};\node (v3) at (1,0.2929) {$n_{i-1}$};
\node (v2) at (3,2.2929) {$n_i$};\node (v1) at (4,3.7929) {$n_i+1$};
\node[inner sep=0] (v5) at (2.1,1.3929) {};\node[inner sep=0] (v6) at (1.8,1.0929) {};

\draw  (v1) edge[->,decorate,decoration=snake] (v2);\draw  (v3) edge[->,decorate,decoration=snake] (v4);\draw  (v2) edge[->] (v5);\draw  (v6) edge[->] (v3);\draw  (v5) edge[dotted] (v6);

%case 2
\begin{scope}[shift={(8,2.9)}]
\draw[rotate=-45]  (0,0) ellipse (1.8 and 1.4);
\end{scope}
\begin{scope}[shift={(11.9,5.5)}]
\draw[rotate=45]  (-1.5,-1.5) ellipse (1.6 and 1.4);
\end{scope}

\node (v7) at (7.6,3.5) {$n_{i-1}-1$};
\node (v8) at (9,1.5) {$n_{i-1}$};
\node (v9) at (9,-0.5) {$j-1$};
\node (v10) at (10,-1.5) {$j$};
\node (v11) at (11,-0.5) {$j+1$};
\node (v12) at (11,1) {$n_i-1$};
\node (v13) at (11,2) {$n_i$};
\node (v14) at (12.3,4) {$n_i+1$};
\draw  (v7) edge[->,decorate,decoration=snake] (v8);
\draw  (v9) edge[->] (v10);
\draw  (v11) edge[->] (v10);
\draw  (v13) edge[->] (v12);
\draw  (v14) edge[->,decorate,decoration=snake] (v13);
\node (v21) at (9,0.5) {};
\draw  (v8) edge[->] (v21);
\draw  (v21) edge[dotted] (v9);
\draw  (v12) edge[dotted] (v11);

% case 3
\begin{scope}[shift={(17,3.8)}]
\draw[rotate=-40]  (-0.5,-1) ellipse (1.8 and 1.4);
\end{scope}
\begin{scope}[shift={(16.2,0)}]
\draw  (4,-1) ellipse (1.4 and 1.8);
\end{scope}
\node at (15.5,4.8) {$\unlhd_{i-1}$};\node at (20.3,.8) {$\unlhd_{i+1}$};
\node (v15) at (15.5,3.5) {$n_{i-1}-1$};
\node (v16) at (17,2) {$n_{i-1}$};
\node (v4) at (20.5,-1.5) {$n_{i}+1$};
\node[inner sep=0] (v17) at (17.8,1.2) {};
\node[inner sep=0] (v18) at (18.2,0.8) {};
\node (v19) at (19,0) {$n_i$};
\draw  (v15) edge[->,decorate,decoration=snake] (v16);
\draw  (v16) edge[->] (v17);
\draw  (v18) edge[->] (v19);
\draw  (v17) edge[dotted] (v18);
\draw  (v19) edge[->,decorate,decoration=snake] (v4);

% general
\node at (-1,5) {(i)};
\node at (6.5,5) {(ii)};
\node at (14,5) {(iii)};
\node at (7.5,4.5) {$\unlhd_{i-1}$};\node at (12,4.8) {$\unlhd_{i+1}$};
\end{tikzpicture}
\caption{Local picture of an $A$-admissible sequence}
\label{fig:adm seq}
\end{figure}
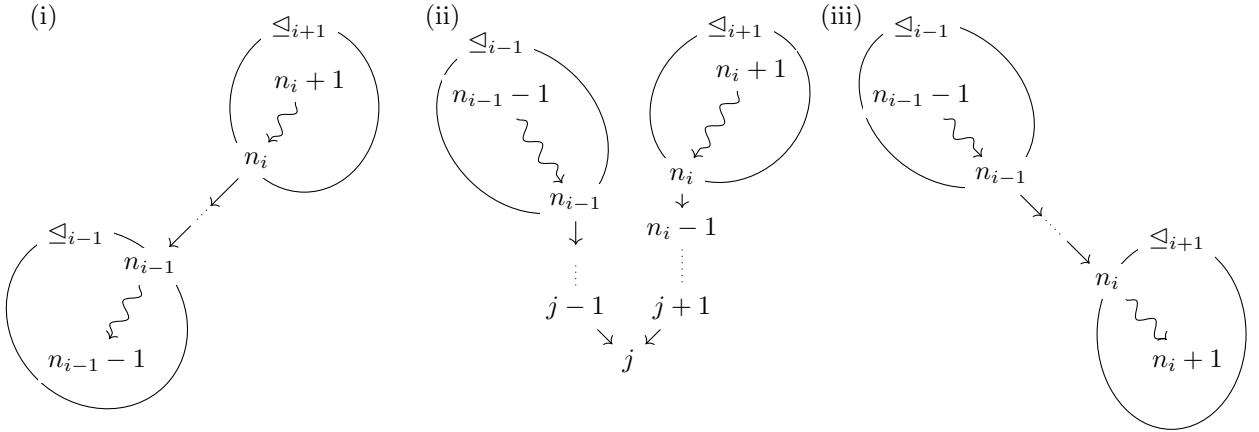

Note that the condition on the points $n_{i-1}-1, n_{i-1}, \ldots, n_{i+1}$ comes from the condition ($\std$) of Definition \ref{def:gluing conditions} in case (i); likewise, from ($\costd$) in case (iii).
For case (ii), the condition on $n_{i-1}-1, \ldots, j$ comes from  ($\std$), whereas that on $j,\ldots, n_i+1$ comes from ($\costd$).
It is now straightforward to see that the definition is designed so that we have the following consequence of Theorem \ref{thm:pro-qhs-node-gluing}.

\begin{corollary}\label{cor:qhs-qLNaka-tilt}
For $A=B_1\vee_{n_1} B_2 \vee_{n_2} \cdots \vee_{n_k} B_{k+1}$ as above, we have bijections
\begin{align}
\{ A\text{-admissible sequence } (\unlhd_{i})_{i}  \} \longleftrightarrow \qhs A \longleftrightarrow \tilt A.\notag
\end{align}
\end{corollary}

To better describe the elements in the left-hand side of Corollary \ref{cor:qhs-qLNaka-tilt}, we recall the description of $\qhs(A_{m})$ for $A_m:=\Bbbk \vec{\bbA}_m$ shown in \cite[Theorem 4.7]{FKR22}.  They showed that there is a bijection
\begin{align}
\qhs (\Bbbk \vec{\bbA}_{m}) \longleftrightarrow \{\text{binary trees with }m\text{ nodes}\}\notag
\end{align}
given by sending $\unlhd$ to its Hasse diagram (and forgetting the node labelling), which we denote by $\mathfrak{t}_\unlhd$.
Here, we view a binary tree with the unique root on top, representing the maximum element of $\unlhd$.
Denote by $\unlhd_{\mathfrak{t}}$ the quasi-hereditary partial order corresponding to a binary tree $\mathfrak{t}$.
This is the partial order whose Hasse diagram is given by labelling nodes of $\mathfrak{t}$ via in-order transversal; see \cite[subsection 4.1]{FKR22} for details.

In this way, the tilting module corresponding to $\unlhd$ (or the labelled tree $\mathfrak{t}$) has indecomposable (IS-tilting) decomposition
\begin{align}\label{eq:chtilt-An-FKR}
T_{\unlhd} = \bigoplus_{j=1}^{m} T(j) \text { with } [T(j):S(k)]\neq 0 \Leftrightarrow k \text{ is a descendant node of $j$ in $\mathfrak{t}$.} 
\end{align}

For completeness, let us also recall the inverse construction, i.e. for $T=T_\unlhd\in\tilt(\Bbbk\vec{\bbA}_{m})$, the algorithm to obtain $\mathfrak{t}=\mathfrak{t}_{\unlhd}$.
This is basically the same construction as applying $T\mapsto \unlhd_{T}$ for general IS-tilting modules over an arbitrary algebra but enhanced in the way so that it matches the labelling convention on the nodes of the resulting tree.  

\begin{algorithm}\label{alg:tilt-poset}
Start with $\ell=1$ and an empty forest $\mathfrak{t}_{0} = \emptyset$.
For each $\ell\ge 1$, we construct $\mathfrak{t}_{\ell}$ from $\mathfrak{t}_{\ell-1}$ as follows.
\begin{enumerate}
\item Insert one new node for each indecomposable direct summand $X$ of $T$ that satisfies
\begin{itemize}
\item $X$ is of length $\ell$ with $[X:S(i)]=1$ and $i\notin \mathfrak{t}_{\ell-1}$; and
\item $[X:S(j)]\neq 0$ and $j\neq i$ implies that $j\in \mathfrak{t}_{\ell-1}$.
\end{itemize}
We index such an $X$ by $T(i)$, and insert a new node $i$ into $\mathfrak{t}_{\ell-1}$ according to the following rule (2).
\item If $T(i)=S(i)$, then insert $i$ as a new leaf.  Otherwise, insert $i$ as the parent root of the root $j$ of every component in $\mathfrak{t}_{\ell-1}$ that satisfies $[T(i):S(j)]\neq 0$.
There will be at most two such $j$, and we place $j$ in the left (respectively, right) branch of $i$ if $j<i$ (respectively, $j>i$). 
Note that now $i$ is a root of some component in $\mathfrak{t}_\ell$.
\item Increase $\ell$ and repeat (1) until $\mathfrak{t}_\ell$ has $m$ nodes.
\end{enumerate}
\end{algorithm}

Note that the indexing of indecomposable direct summands coincides with the indexing when $T$ is viewed as a characteristic tilting module (equivalently, IS-tilting).

\begin{example}\label{eg:FKR-binary-tree}
Consider a path algebra $\Bbbk(1\rightarrow 2\rightarrow 3\rightarrow 4\rightarrow 5)$ and a tilting module
\begin{align}
T = 1\oplus 3 \oplus 5 \oplus \sm{1\\2\\3} \oplus \sm{1\\2\\3\\4\\5}
=T(1)\oplus T(3)\oplus T(5)\oplus T(2)\oplus T(4). \notag
\end{align}
Applying Algorithm \ref{alg:tilt-poset} to $T$, we have
\[
\mathfrak{t}=
\begin{tikzcd}[row sep=tiny,column sep=tiny,font = \scriptsize]
&& 4\ar[ld]\ar[rd]&\\
&2\ar[ld]\ar[rd]&&5.\\
1&&3&\end{tikzcd}
\]
\end{example}

An element in the left-hand set in Corollary \ref{cor:qhs-qLNaka-tilt} is given by a sequence of (in-order-transversal-labelled) binary trees and V-shaped posets satisfying $A$-admissibility.

Going from a tilting $A$-module $T$ to such an admissible sequence (hence, also to $\unlhd_{T}$) is no-different from applying Theorem \ref{thm:bij-qh-IStilting}.
This is similar to running Algorithm \ref{alg:tilt-poset}, as we demonstrate below.

\begin{example}
Consider $A$ given as follows.
\[
\begin{tikzpicture}
\draw (6.5,-0.4) edge[decorate,decoration={brace,amplitude=5}] node[midway,below=4] {$B_{2}$} (3.5,-0.4);
\draw (9.5,-0.4) edge[decorate,decoration={brace,amplitude=5}] node[midway,below=4] {$B_{4}$} (7.5,-0.4);
\draw (0,0.5) edge[decorate,decoration={brace,amplitude=5}] node[midway,above=4] {$B_{1}$} (4.5,0.5);
\draw (5.5,0.5) edge[decorate,decoration={brace,amplitude=5}] node[midway,above=4] {$B_{3}$} (8.5,0.5);
\node (v1) at (0,0) {1};\node (v2) at (1,0) {2};\node (v3) at (2,0) {3};
\node (v4) at (3,0) {4};\node (v5) at (4,0) {5};\node (v6) at (5,0) {6};
\node (v7) at (6,0) {7};\node (v8) at (7,0) {8};\node (v9) at (8,0) {9};
\node (v10) at (9,0) {10.};

\draw  (v1) edge[->] (v2);
\draw  (v2) edge[->]  (v3);
\draw  (v3) edge[->]  (v4);
\draw  (v4) edge[->] node[midway](r1){} (v5);
\draw  (v5) edge[->] node[midway](r2){}(v6);
\draw  (v6) edge[->] node[midway](r3){} (v7);
\draw  (v7) edge[->] node[midway](r4){} (v8);
\draw  (v8) edge[->] node[midway](r5){} (v9);
\draw  (v9) edge[->] node[midway](r6){}(v10);
\draw (r1) edge[bend left=50,dashed] (r2);
\draw (r2) edge[bend left=50,dashed] (r3);
\draw (r3) edge[bend left=50,dashed] (r4);
\draw (r5) edge[bend left=50,dashed] (r6);
\node at (-1,0) {$A$:};
\end{tikzpicture}
\]
Then we have $B_{1}:=\e_{1} A\e_{1}\cong \Bbbk\vec{\bbA}_{5}$, $B_{2}:=\e_{2}A\e_{2} = A_{3}^{!}$, $B_{3}:= \e_{3}A\e_{3}\cong \Bbbk\vec{\bbA}_3$, $B_{4}:=\e_{4}A\e_{4} \cong \Bbbk\vec{\bbA}_2$, where
\begin{align}
\e_{1}:= e_{1}+e_{2}+\cdots+e_{5},\; \e_{2}:= e_{5}+e_{6}+e_{7},\; \e_{3}:=e_{7}+e_{8}+e_{9},\; \e_{4}:=e_{9}+e_{10}.\notag
\end{align}
Take the following tilting module $T$ and index the indecomposable direct summands accordingly.
\begin{align}
\arraycolsep=1.4pt
\begin{array}{rccccccccccccccccccc} T=&  1 &\oplus& \sm{1\\2\\3} &\oplus& 3  &\oplus& \sm{1\\2\\3\\4\\5} &\oplus& \sm{5\\6} &\oplus& 6 &\oplus& \sm{6\\7} &\oplus& \sm{7\\8\\9} &\oplus& \sm{9\\10}  &\oplus& 10\\
 =& T(1)&\oplus& T(2) &\oplus& T(3) &\oplus& T(4) &\oplus& T(5) &\oplus&  T(6)&\oplus& T(7)&\oplus& T(8) &\oplus& T(9) &\oplus& T(10).
\end{array}\notag
\end{align}

Applying $(-)\e_{i}$ to $T(j)$ whenever $\e_{i}e_{j}\neq 0$, we get a tilting $B_{i}$-module $T_{i}:=\bigoplus_{j}T(j)\e_{i}$ for each $1\le i\le 4$:
\begin{align}
T_{1} = 1\oplus \sm{1\\2\\3} \oplus 3 \oplus \sm{1\\2\\3\\4\\5} \oplus 5, \quad T_{2} =\sm{5\\6}\oplus6\oplus \sm{6\\7}, \quad T_{3} = \sm{7}\oplus \sm{7\\8\\9} \oplus 9, \quad T_{4} = \sm{9\\10}\oplus 10.\notag
\end{align}
Applying Algorithm \ref{alg:tilt-poset} (for $T_{1}$, this is just Example \ref{eg:FKR-binary-tree}) and the analogous calculation for $A_r^!$ (Example \ref{eg:qhsA_n^!}) we obtain the following respective local partial orders:
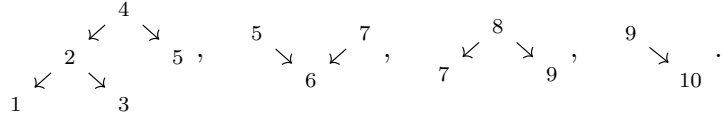
\[
\begin{tikzcd}[row sep=tiny,column sep=tiny,font = \scriptsize]
&& 4\ar[ld]\ar[rd]&\\
&2\ar[ld]\ar[rd]&&5\\
1&&3&\end{tikzcd},\quad 
\begin{tikzcd}[row sep=tiny,column sep=tiny,font = \scriptsize]
5\ar[rd]& &7\ar[ld]\\&6&\end{tikzcd},\quad 
\begin{tikzcd}[row sep=tiny,column sep=tiny,font = \scriptsize]
&8\ar[ld]\ar[rd]& \\ 7& &9\\\end{tikzcd}, \quad
\begin{tikzcd}[row sep=tiny,column sep=tiny,font = \scriptsize]
9\ar[rd]&\\ &10\end{tikzcd}.
\]
Finally, we get the (Hasse quiver of) $\unlhd_{T}$ by gluing these partial orders at vertices $5, 7, 9$, which gives :
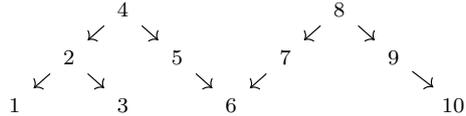
\[
\begin{tikzcd}[row sep=tiny,column sep=tiny,font = \scriptsize]
&& 4\ar[ld]\ar[rd]&&&&8\ar[ld]\ar[rd]&&\\
&2\ar[ld]\ar[rd]&&5\ar[rd]&&7\ar[ld]&&9\ar[rd]&\\
1&&3&&6&&&&10.
\end{tikzcd}
\]
\end{example}

We can observe the following from the construction.

\begin{proposition}
Every minimal adapted order of a quadratic linear Nakayama algebra is a tree.
\end{proposition}

\begin{proof}
We have explained that every $[\unlhd]\in\qhs A$ for $A \in \{ \Bbbk\vec{\bbA}_m, A_{m}^{!} \mid m\ge 1\}$ has $\unlhd^{\rm m}$ being a tree.  Generally, for a quadratic linear Nakayama algebra $A$ and any $[\unlhd]\in \qhs A$, since $\unlhd^{\rm m}$ is given by gluing these local pieces at the end points by Corollary \ref{cor:qhs-qLNaka-tilt}, the resultant must also be a tree.
\end{proof}

\end{document}